\documentclass[aop,MSNbibl,seceqn,dvips]{arximspdf}

\doi{10.1214/13-AOP896} 
\volume{43}
\issue{1}
\pubyear{2015}
\firstpage{188}
\lastpage{239}

\makeatletter

\newcommand{\rrVert}{\Vert}
\newcommand{\rrvert}{\vert}
\newcommand{\llVert}{\Vert}
\newcommand{\llvert}{\vert}
\newcommand{\eqref}[1]{(\ref{#1})}
\newtheorem{theorem}{Theorem}[section]

\newtheorem{condition}{Condition}

\newtheorem{corollary}[theorem]{Corollary}

\newproclaim{definition}[theorem]{Definition}
\newproclaim{example}[theorem]{Example}

\newtheorem{lemma}[theorem]{Lemma}
\newproclaim{notation}{Notation}
\newproclaim{problem}[theorem]{Problem}
\newtheorem{proposition}[theorem]{Proposition}
\newproclaim{remark}[theorem]{Remark}

\def\Var{\operatorname{Var}}
\def\cov{\operatorname{Cov}}
\def\CD{\mathcal{D}}
\def\CC{\mathrm{CC}}
\def\opn{\mathrm{op}}
\makeatother

\begin{document}
\begin{frontmatter}

\title{Smoothness of the density for solutions to Gaussian rough
differential equations}
\runtitle{Smoothness of Gaussian RDEs}

\begin{aug}
\author[A]{\fnms{Thomas} \snm{Cass}\corref{}\ead[label=e1]{thomas.cass@imperial.ac.uk}},
\author[B]{\fnms{Martin} \snm{Hairer}\ead[label=e2]{M.Hairer@Warwick.ac.uk}\thanksref{T1}},
\author[C]{\fnms{Christian} \snm{Litterer}\thanksref{T2}\ead[label=e3]{christian.litterer@polytechnique.edu}}\break
\and
\author[D]{\fnms{Samy} \snm{Tindel}\ead[label=e4]{samy.tindel@univ-lorraine.fr}\thanksref{T3}}
\runauthor{Cass, Hairer, Litterer and Tindel}
\affiliation{Imperial College London, University of Warwick, Imperial
College London\break and
Universit\'e de Lorraine}
\address[A]{T. Cass\\
C. Litterer\\
Department of Mathematics\\
Imperial College London\\
The Huxley Building\\
180 Queensgate, London\\
United Kingdom\\
\printead{e1}} 
\address[B]{M. Hairer\\
Mathematics Institute\\
University of Warwick\\
Coventry, CV4 7AL\\
United Kingdom\\
\printead{e2}\hspace*{10pt}}
\address[C]{C. Litterer\\
Centre de Math\'ematiques Appliqu\'ees\\
\'Ecole Polytechnique\\
Route de Sacla\\
91128 Palaiseau\\
France\\
\printead{e3}}
\address[D]{S. Tindel\\
Institut {\'E}lie Cartan Nancy\\
Universit\'e de Lorraine\\
B.P. 239\\
54506 Vand{\oe}uvre-l{\`e}s-Nancy\\
France\\
\printead{e4}}
\thankstext{T1}{Supported by EPSRC Grant EP/D071593/1, a
Wolfson Research Merit award of the Royal Society and a Philip
Leverhulme Prize.}
\thankstext{T2}{Supported by EPSRC Grant EP/H000100/1 and supported in part
by a grant of the
European Research Council (ERC Grant nr. 258237).}
\thankstext{T3}{S. Tindel is member of the BIGS (Biology, Genetics and
Statistics) team at INRIA.}
\end{aug}

\received{\smonth{9} \syear{2012}}
\revised{\smonth{10} \syear{2013}}

%
\begin{abstract}
We consider stochastic differential equations of the form
$dY_{t}= V (
Y_{t} ) \,dX_{t}+V_{0} ( Y_{t} ) \,dt$ driven by a
multi-dimensional Gaussian process. Under the assumption that the vector
fields $V_{0}$ and $V= ( V_{1},\ldots,V_{d} ) $ satisfy
H\"{o}rmander's bracket condition, we demonstrate that $Y_{t}$ admits a smooth
density for any $t\in(0,T]$, provided the driving noise satisfies certain
nondegeneracy assumptions. Our analysis relies on relies on an
interplay of
rough path theory, Malliavin calculus and the theory of Gaussian processes.
Our result applies to a broad range of examples including fractional Brownian
motion with Hurst parameter $H>1/4$, the Ornstein--Uhlenbeck process
and the
Brownian bridge returning after time $T$.
\end{abstract}

%

\end{frontmatter}

\section{Introduction}\label{sec1}

Over the past decade, our understanding of stochastic differential equations
(SDEs) driven by Gaussian processes has evolved considerably. As a natural
counterpart to this development, there is now much interest in investigating
the probabilistic properties of solutions to these equations. Consider
an SDE
of the form%
%
\begin{equation}
dY_{t}=V(Y_{t})\,dX_{t}+V_{0} (
Y_{t} ) \,dt,\qquad Y ( 0 ) =y_{0}\in%
\mathbb{R} 
^{e}, \label{Int-Eq}%
\end{equation}
driven by an $%
\mathbb{R}
^{d}$-valued continuous Gaussian process $X$ along $C_{b}^{\infty}$ vector
fields $V_{0}$ and $V= ( V_{1},\ldots,V_{d} ) $ on $%
\mathbb{R}
^{e}$. Once the existence and uniqueness of $Y$ has been settled, it is
natural to ask about the existence of a smooth density of $Y_{t}$ for \mbox{$t>0$}.
In the context of diffusion processes, the theory is classical and goes back
to H\"{o}rmander \cite{horm} for an analytical approach, and Malliavin
\cite{mall} for a probabilistic one.

For the case where $X$ is fractional Brownian motion, this question was first
addressed by Nualart and Hu \cite{NH}, where the authors show the existence
and smoothness of the density when the vector fields are elliptic, and the
driving Gaussian noise is fractional Brownian motion (fBm) for $H>1/2$.
Further progress was achieved in \cite{BH} where, again for the regime
$H>1/2$, the density was shown to be smooth under H\"{o}rmander's celebrated
bracket condition. Rougher noises are not directly amenable to the analysis
put forward in these two papers. Additional ingredients have since gradually
become available with the development of a broader theory of (Gaussian) rough
paths (see \cite{L,CQ,FV}). The papers \cite{CFV} and
\cite{CF} used this technology to establish the existence of a density under
fairly general assumptions on the Gaussian driving noises. These papers,
however, fall short of proving the smoothness of the density, because
the proof
demands far more quantitative estimates than were available at the time.

More recently, decisive progress was made on two aspects which
obstructed the
extension of this earlier work. First, the paper \cite{CLL}
established sharp
tail estimates on the Jacobian of the flow $J_{t\leftarrow0}^{\mathbf{X}
}(y_{0})$ driven by a wide class of (rough) Gaussian processes. The
tail turns
out to decay quickly enough to allow to conclude the finiteness of all moments
for $J_{t\leftarrow0}^{\mathbf{X}}(y_{0})$. Second, \cite{H3}
obtained a
general, deterministic version of the key Norris lemma (see also \cite
{HT} for
some recent work in the context of fractional Brownian motion). The
lemma of
Norris first appeared in~\cite{N} and has been interpreted as a quantitative
version of the Doob--Meyer decomposition. Roughly speaking, it ensures that
there cannot be too many cancellations between martingale and bounded
variation parts of the decomposition. The work~\cite{H3}, however,
shows that
the same phenomenon arises in a purely deterministic setting, provided that
the one-dimensional projections of the driving process are sufficiently and
uniformly rough. This intuition is made precise through a notion of
``modulus of H\"{o}lder roughness;'' see Definition~\ref{def:rough}
below. Together with an analysis of the higher
order Malliavin derivatives of the flow of (\ref{Int-Eq}), also carried
out in
\cite{H3}, these two results yield a H\"{o}rmander-type theorem for fractional
Brownian motion if $H>1/3$.

In this paper, we aim to realise the broader potential of these
developments by
generalising the analysis to a wide class of Gaussian processes. This class
includes fractional Brownian motion with Hurst parameter $H \in({\frac{1}{4}},
{\frac{1}{2}}]$, the Ornstein--Uhlenbeck process, and the Brownian bridge.
Instead of focusing on particular examples of processes, our approach
aims to
develop a general set of conditions on $X$ under which Malliavin--H\"{o}rmander
theory still works.

The probabilistic proof of H\"ormander's theorem is intricate, and hard to
summarise in a few lines; see \cite{MallHor} for a relatively short
exposition. However, let us highlight some basic features of the method in
order to see where our main contributions lie:

\begin{longlist}[(iii)]
\item[(i)] At the centre of the proof of H\"{o}rmander's theorem is a
quantitative estimate on the nondegeneracy of the Malliavin covariance matrix
$C_{T} ( \omega ) $. Our effort in this direction consists
in a
direct and instructive approach, which reveals an additional structure
of the
problem. In particular, the conditional variance of the process plays an
important role, which does not appear to have been noticed so far. More
specifically, following \cite{CFV} we study the Malliavin covariance
matrix as
a 2D Young integral against the covariance function $R ( s,t
) $.
This provides the convenient representation
\[
v^{T}C_{t}(\omega) v=\int_{ [ 0,t ] \times [ 0,t
] }f_{s}
( v;\omega ) f_{r} ( v;\omega ) \,dR ( s,r )
\]
for some $\gamma$-H\"{o}lder continuous $f ( v;\omega ) $, which
avoids any detours via the fractional calculus that are specific to fBm.
Compared to the setting of \cite{CF}, we have to impose some additional
assumptions on $R ( s,t ) $, but our more quantitative approach
allows us in return to relax the zero--one law condition required in
this paper.

\item[(ii)] An essential step in the proof is achieved when one
obtains some
lower bounds on $v^{T}C_{t} v$ in terms of the supremum norm of $f$.
Toward this
aim, we prove a novel interpolation inequality, which lies at the heart
of this
paper. It is explicit and also sharp in the sense that it collapses to a
well-known inequality for the space $L^{2}([ 0,T]) $ in the case of Brownian
motion. Furthermore, this result should be important in other
applications in
the area, for example, in establishing bounds on the density function (see
\cite{BOS} for a first step in this direction) or studying small-time
asymptotic.

\item[(iii)] H\"{o}rmander's theorem also relies on an accurate
analysis and
control of the higher order Malliavin derivatives of the flow
$J_{t\leftarrow
0}^{\mathbf{X}}(y_{0})$. This turns out the be notationally
cumbersome, but
structurally quite similar to the technology already developed for fBm. For
this step, we therefore rely as much as possible on the analysis
performed in
\cite{H3}. The integrability results in \cite{CLL} then play the first
of two
important roles in showing that the flow belongs to the Shigekawa--Sobolev
space $\mathbb{D}^{\infty}(\mathbb{R}^{e})$.

\item[(iv)] Finally, an induction argument that allows to transfer the bounds
from the interpolation inequality to the higher order Lie brackets of the
vector fields has to be set up. This induction requires another integrability
estimate for $J_{t\leftarrow0}^{\mathbf{X}}(y_{0})$, plus a Norris
type lemma
allowing to bound a generic integrand $A$ in terms of the resulting noisy
integral $\int A \,dX$ in the rough path context. This is the content of our
second main contribution, which can be seen as a generalisation of the Norris
lemma from \cite{H3} to a much wider range of regularities and Gaussian
structures for the driving process $X$. Namely, we extend the result of
\cite{H3} from $p$-rough paths with $p<3$ to general $p$ under the same
``modulus of H\"{o}lder roughness'' assumption. It is interesting to
note that
the argument still only requires information about the roughness of the path
itself and not its lift.
\end{longlist}

Let us further comment on the Gaussian assumptions allowing the
derivation of
the interpolation inequality briefly described in step (ii) above.
First, we
need a standing assumption that regards the regularity of $R(s,t)$ (expressed
in terms of its so called 2D $\rho$-variation, see \cite{FV}) and
complementary Young regularity of $X$ and its Cameron--Martin space.
This is a
standard assumption in the theory of Gaussian rough paths. The first
part of
the condition guarantees the existence of a natural lift of the process
to a
rough path. The complementary Young regularity in turn is necessary to perform
Malliavin calculus, and allows us to obtain the integrability estimates for
$J_{t\leftarrow0}^{\mathbf{X}}(y_{0})$ in \cite{CLL}.

In order to understand the assumptions on which our central interpolation
inequality hinges, let us mention that it emerges from the need to
prove lower
bounds of the type
%
\begin{equation}
\int_{ [ 0,T ] \times [ 0,T ] }f_{s}f_{t} \,dR ( s,t ) \geq
C\llvert f\rrvert _{\gamma; [ 0,T
] }%
^{a}\llvert f\rrvert
_{\infty; [ 0,T ] }^{2-a} \label
{opt}%
\end{equation}
for some exponents $\gamma$ and $a$, and all $\gamma$-H\"{o}lder continuous
functions $f$. After viewing the integral in (\ref{opt}) along a
sequence of
discrete-time approximations to the integral, relation \eqref{opt}
relies on
solving a sequence of finite dimensional partially constrained quadratic
programming (QP) problems. These (QP) problems involve some matrices
$Q$ whose
entries can be written as $Q^{ij}=E[X_{t_{i},t_{i+1}}^{1}
X_{t_{j},t_{j+1}}^{1}]$, where $X_{t_{i},t_{i+1}}^{1}$ denotes the
increment $X_{t_{i+1}}^{1}-X_{t_{i}}^{1}$ of the first component of $X$.
Interestingly enough, some positivity properties of Schur complements computed
within the matrix $Q$ play a prominent role in the resolution of the
aforementioned (QP) problems. In order to guarantee these positivity
properties, we shall make two nondegeneracy type assumptions on the
conditional variance and covariance structure of our underlying process
$X^{1}$ (see Conditions \ref{nondeterm} and \ref{cond dom} below).
This is
quite natural since Schur complements are classically related to
conditional variances in elementary Gaussian analysis. We also believe that
our conditions essentially characterise the class of processes for
which we
can quantify the nondegeneracy of $C_{T}(\omega)$ in terms of the conditional
variance of the process $X$.

The outline of the article is as follows. In Section~\ref{rough
paths}, we
give a short overview of the elements of the theory of rough paths required
for our analysis. Section~\ref{main thm} then states our main result. In
Section~\ref{examples section}, we demonstrate how to verify the
nondegeneracy assumptions required for the driving process in a number of
concrete examples. The remainder of the article is devoted to the proofs.
First, in Section~\ref{Norris lemma section}, we state and prove our general
version of Norris's lemma and we apply it to the class of Gaussian processes
we have in mind.\vadjust{\goodbreak} In Section~\ref{interpol}, we then provide the proof
of an
interpolation inequality of the type (\ref{opt}). In
Section~\ref{differentiability section}, we obtain bounds on the
derivatives of
the solution with respect to its initial condition, as well as on its
Malliavin derivative. Finally, we combine all of these ingredients in
Section~\ref{section proof main theorem} to complete the proof of our
main theorem.

\section{Rough paths and Gaussian processes}
\label{rough paths}

In this section, we introduce some basic notation concerning rough paths,
following the exposition in \cite{CLL}. In particular, we recall the
conditions needed to ensure that a given Gaussian process has a natural rough
path lift.

For $N\in\mathbb{N}$, recall that the truncated tensor algebra
$T^{N}(\mathbb{R}%
^{d})$ is defined by $T^{N}(\mathbb{R}^{d})=\bigoplus_{n=0}^{N}(\mathbb
{R}%
^{d})^{\otimes n}$, with the convention $(\mathbb{R}^{d})^{\otimes
0}=\mathbb{R}$. The space $T^{N}(\mathbb{R}^{d})$ is equipped with a
straightforward
vector space structure, plus an operation $\otimes$ defined by
\[
\pi_{n}(g\otimes h)=\sum_{k=0}^{N}
\pi_{n-k}(g)\otimes\pi_{k}(h),\qquad g,h\in T^{N}\bigl(
\mathbb{R}^{d}\bigr),
\]
where $\pi_{n}$ denotes the projection on the $n$th tensor level. Then
$(T^{N}(\mathbb{R}^{d}),+,\otimes)$ is an associative algebra with unit
element $\mathbf{1} \in(\mathbb{R}^{d})^{\otimes0}$.

At its most fundamental, we will study continuous $%
\mathbb{R}
^{d}$-valued paths parameterised by time on a compact interval $ [
0,T ] $; we denote the set of such functions by $C([0,T],\mathbb{R}
^{d})$. We write $x_{s,t}:=x_{t}-x_{s}$ as a shorthand for the
increments of a
path. Using this notation, we define
the uniform norm and the $p$-variation semi-norm of a path $x$ by
%
\begin{equation}
\label{e:defNorms} \qquad\llVert x\rrVert _{\infty}:=\sup_{t\in [ 0,T ] }
\llvert x_{t}\rrvert,\qquad \llVert x\rrVert _{p\mbox{-}\mathrm{var}; [
0,T ] }:= \biggl( \sup
_{\CD}%
\sum_{[s,t] \in\CD}\llvert
x_{s,t}\rrvert ^{p} \biggr) ^{1/p},
\end{equation}
where the supremum in the second term runs over all partitions $\CD$
of $[0,T]$.
We will use the notation $C^{p\mbox{-}\mathrm{var}}( [ 0,T], \mathbb
{R}^{d}) $ for the linear
subspace of $C([0,T],\break \mathbb{R}^{d})$ consisting of the continuous
paths that
have finite $p$-varia\-tion. Of interest will also be the set of
$\gamma
$-H\"{o}lder continuous functions, denoted by $C^{\gamma}([0,T],
\mathbb
{R}%
^{d})$, which consists of functions satisfying
\[
\llVert x\rrVert _{\gamma; [ 0,T ] }:=\sup_{0\leq
s<t\leq
T}
\frac{\llvert  x_{s,t}\rrvert }{\llvert  t-s\rrvert
^{\gamma}%
}<\infty.
\]

For $s<t$ and $n\geq2$, consider the simplex $\Delta_{st}^{n}=\{(u_{1}
,\ldots,u_{n})\in [ s,t]^{n}; u_{1}<\cdots<u_{n}\} $, while the
simplices over $[0,1]$ will simply be denoted by $\Delta^{n}$. A
continuous map
$\mathbf{x}\dvtx \Delta^{2}\rightarrow T^{N}(\mathbb{R}^{d})$ is called a
multiplicative functional if for $s<u<t$ one has $\mathbf{x}_{s,t}%
=\mathbf{x}_{s,u}\otimes\mathbf{x}_{u,t}$. An important example
arises from
considering paths $x$ with finite variation: for $0<s<t$, we set
%
\begin{equation}
\mathbf{x}_{s,t}^{n}=\sum_{1\leq i_{1},\ldots,i_{n}\leq d}
\biggl( \int%
_{\Delta_{st}^{n}}\,dx^{i_{1}}\cdots dx^{i_{n}}
\biggr) e_{i_{1}}\otimes \cdots\otimes e_{i_{n}},
\label{eq:def-iterated-intg}%
\end{equation}
where $\{e_{1},\ldots,e_{d}\}$ denotes the canonical basis of $\mathbb
{R}^{d}%
$, and then define the \textit{signature} of $x$ as
\[
S_{N}(x)\dvtx \Delta^{2}\rightarrow T^{N}\bigl(
\mathbb{R}^{d}\bigr),\qquad (s,t)\mapsto S_{N}(x)_{s,t}:=1+
\sum_{n=1}^{N}\mathbf{x}_{s,t}^{n}.
\]
$S_{N}(x)$ will be our typical example of multiplicative functional.
Let us
also add the following two remarks:

\begin{longlist}[(ii)]
\item[(i)] A geometric rough path (see Definition~\ref{def:RP} below),
as well as the
signature of any smooth function, takes values in the strict subset
$G^{N}(\mathbb{R}^{d})\subset T^{N}(\mathbb{R}^{d})$ given by the ``group-like
elements''
\[
G^{N}\bigl(\mathbb{R}^{d}\bigr) = \exp^{\otimes}
\bigl(L^{N}\bigl(\mathbb {R}^{d}\bigr) \bigr),
\]
where $L^{N}(\mathbb{R}^{d})$ is the linear span of all elements that
can be
written as a commutator of the type $a \otimes b - b\otimes a$ for two
elements in $T^{N}(\mathbb{R}^{d})$.

\item[(ii)] It is sometimes convenient to think of the indices $w=(i_{1}
,\ldots,i_{n})$ in \eqref{eq:def-iterated-intg} as words based on the alphabet
$\{1,\ldots,d\}$. We shall then write $\mathbf{x}^{w}$ for the iterated
integral $\int_{\Delta_{st}^{n}}\,dx^{i_{1}}\cdots dx^{i_{n}}$.
\end{longlist}

More generally, if $N\geq1$ we can consider the set of such group-valued
paths
\[
\mathbf{x}_{t}= \bigl( 1,\mathbf{x}_{t}^{1},
\ldots,\mathbf {x}_{t}^{N} \bigr) \in G^{N}
\bigl( \mathbb{R}^{d} \bigr).
\]
Note that the group structure provides a natural notion of increment, namely
$\mathbf{x}_{s,t}:=\mathbf{x}_{s}^{-1}\otimes\mathbf{x}_{t}$, and
we can
describe the set of ``norms'' on $G^{N} (
\mathbb{R}
^{d} ) $ which are homogeneous with respect to the natural scaling
operation on the tensor algebra (see~\cite{FV} for definitions and details).
One such example is the Carnot--Caratheodory (CC) norm (see \cite
{FV}), which
we denote by $\llVert  \cdot\rrVert  _{\CC}$. The precise norm used is
mostly irrelevant in finite dimensions because they are all equivalent. The
subset of these so-called homogeneous norms which are symmetric and
sub-additive (again, see \cite{FV}) gives rise to genuine metrics on
$G^{N} (
\mathbb{R}
^{d} ) $, for example, $d_{\CC}$ in the case of the CC norm. In
turn, these
metrics give rise to a notion of homogenous $p$-variation metrics
$d_{p\mbox{-}\mathrm{var}}$ on the set of $G^{N} (
\mathbb{R}
^{d} )$-valued paths. Using the CC norm for definiteness, we will
use the
following homogenous $p$-variation and $\gamma$-H\"{o}lder variation
semi-norms:
%
\begin{eqnarray}\label{homogeneous norm}
\llVert \mathbf{x}\rrVert _{p\mbox{-}\mathrm{var}; [
s,t ] } &=&\max_{i=1,\ldots, \lfloor p \rfloor
}
\biggl( \sup_{\CD}\sum_{[s,t] \in\CD}
\llVert \mathbf{x}_{s,t}\rrVert _{\CC}^{p} \biggr)
^{1/p},
\nonumber
\\[-8pt]
\\[-8pt]
\nonumber
\llVert \mathbf{x}\rrVert _{\gamma,%
[ s,t ] } &=&\sup_{(u,v)\in\Delta_{st}^{2}}
\frac{\llVert
\mathbf{x}_{u,v}\rrVert  _{\CC}}{\llvert
v-u\rrvert ^{\gamma}},
\nonumber
\end{eqnarray}
where the supremum over $\CD$ is as in \eqref{e:defNorms}.

We will also use some metrics on path spaces which are not homogenous.
The most
important will be the following:%
%
\begin{equation}
\mathcal{N}_{\mathbf{x,}\gamma;[s,t]}:=\sum_{k=1}^{N}
\sup_{(u,v)\in
\Delta
_{st}^{2}}\frac{|\mathbf{x}_{u,v}^{k}|_{ (
\mathbb{R}
^{d} ) ^{\otimes k}}}{|v-u|^{k\gamma}}, \label{inhomogeneous}%
\end{equation}
which will be written simply as $\mathcal{N}_{\mathbf{x,}\gamma}$ when
the interval $ [ s,t ] $ is clear from the context.

\begin{definition}
\label{def:RP} The space of weakly geometric $p$-rough paths [denoted
$WG\Omega_{p}(\mathbb{R}^{d})$] is the set of paths $\mathbf
{x}\dvtx \Delta
^{2}\rightarrow G^{\lfloor p\rfloor}( \mathbb{R} ^{d}) $ such that
(\ref{homogeneous norm}) is finite.
\end{definition}

We will also work with the space of geometric $p$-rough paths, which we denote
by $G\Omega_{p}( \mathbb{R}^{d}) $, defined as the $d_{p\mbox{-}\mathrm{var}}%
$-closure of
\[
\bigl\{ S_{ \lfloor p \rfloor} ( x ) \dvtx x\in C^{1\mbox{-}\mathrm{var}} \bigl( [ 0,T ]
,%
\mathbb{R} 
^{d} \bigr)
\bigr\}.
\]
Analogously, if $\gamma>0$ and $N=[1/\gamma]$ we define $C^{0,\gamma
}([0,T];G^{N}(\mathbb{R}^{d}))$ to be the linear subspace of $G\Omega_{N}(
\mathbb{R}^{d}) $ consisting of paths $\mathbf{x}\dvtx \Delta
^{2}\rightarrow
G^{N}(\mathbb{R}^{d})$ such that
\[
\lim_{n\rightarrow\infty}\bigl\Vert\mathbf{x}-S_{N}(x_{n})
\bigr\Vert _{\gamma
;[0,T]}=0
\]
for some sequence $ ( x_{n} ) _{n=1}^{\infty}\subset
C^{\infty
}([0,T];\mathbb{R}^{d})$.

In the following, we will consider RDEs driven by paths $\mathbf{x}$ in
$WG\Omega_{p}(\mathbb{R}^{d}) $, along a collection of vector fields
$V= ( V_{1},\ldots,V_{d} ) $ on $\mathbb{R}^{e}$, as well
as a
deterministic drift along $V_{0}$. From the point of view of existence and
uniqueness results, the appropriate way to measure the regularity of the
$V_{i}$'s turns out to be the notion of Lipschitz-$\gamma$ (short:
Lip-$\gamma$) in the sense of Stein \cite{FV,LCL}. This
notion provides a norm on the space of such vector fields (the
Lip-$\gamma$
norm), which we denote $\llvert \cdot\rrvert _{\mathrm{Lip}\mbox{-}\gamma}$. For the collection $V$ of vector fields, we will often make
use of
the shorthand
\[
\llvert V\rrvert _{\mathrm{Lip}\mbox{-}\gamma}=\max_{i=1,\ldots
,d}\llvert
V_{i}\rrvert _{\mathrm{Lip}\mbox{-}\gamma},
\]
and refer to the quantity $\llvert  V\rrvert _{\mathrm{Lip}\mbox{-}\gamma}$ as the Lip-$\gamma$ norm of $V$.

A theory of such Gaussian rough paths has been developed by a
succession of
authors \cite{CQ,FV07,CFV,FO10} and we will mostly work within their
framework. To be more precise, we will assume that $X_{t}= ( X_{t}%
^{1},\ldots,X_{t}^{d} ) $ is a continuous, centred (i.e., mean zero)
Gaussian process with i.i.d. components on a complete probability space
$ ( \Omega,\mathcal{F},P ) $. Let $\mathcal{W}=C([ 0,T],
\mathbb{R}^{d})$ and suppose that $ ( \mathcal{W},\mathcal
{H},\mu
)
$ is the abstract Wiener space associated with $X$. The function
$R\dvtx [
0,T ] \times [ 0,T ] \rightarrow%
\mathbb{R}
$ will denote the covariance function of any component of $X$, that is,
\[
R(s,t) =E \bigl[ X_{s}^{1}X_{t}^{1}
\bigr].
\]
Following \cite{FV07}, we recall some basic assumptions on the covariance
function of a Gaussian process which are sufficient to guarantee the existence
of a natural lift of a Gaussian rough process to a rough path. We
recall the
notion of \textit{rectangular increments} of $R$ from \cite
{FVupdate}; these
are defined by
\[
R\pmatrix{
s,t
\vspace*{2pt}\cr
u,v }:=E \bigl[ \bigl( X_{t}^{1}-X_{s}^{1}
\bigr) \bigl( X_{v}%
^{1}-X_{u}^{1}
\bigr) \bigr].
\]
The existence of a lift for $X$ is ensured by insisting on a sufficient rate
of decay for the correlation of the increments. This is captured, in a very
general way, by the following two-dimensional $\rho$-variation
constraint on
the covariance function.

\begin{definition}
\label{rho var}Given $1\leq\rho<2$, we say that $R$ has \textit{finite
(two-dimen\-sional) }$\rho$\textit{-variation} if%
%
\begin{equation}
V_{\rho} \bigl( R; [ 0,T ] \times [ 0,T ] \bigr) ^{\rho}:=\sup
_{\CD,\CD'}\mathop{\sum_{[s,t] \in\CD}}_{[s',t'] \in\CD
'}\biggl\llvert R
\pmatrix{
s,t
\vspace*{2pt}\cr
s',t'}
\biggr\rrvert ^{\rho}<\infty. \label{2dvariation}%
\end{equation}
\end{definition}

If a process has a covariance function with finite $\rho$-variation for
$\rho\in [1,2)$ in the sense of Definition~\ref{rho var},
\cite{FV07}, Theorem~35, asserts that $ ( X_{t} ) _{t\in [ 0,T ] }
$ lifts
to a geometric $p$-rough path provided $p>2\rho$. Moreover, there is a unique
\textit{natural lift} which is the limit, in the $d_{p\mbox{-}\mathrm{var}}$-induced topology, of the canonical lift of piecewise linear approximations
to $X$.

A related take on this notion is obtained by enlarging the set of partitions
of $ [ 0,T ] ^{2}$ over which the supremum is taken in
(\ref{2dvariation}). Recall from \cite{FVupdate} that a \textit{rectangular
partition }of the square $ [ 0,T ] ^{2}$ is a collection
$
\{
A_{i}\dvtx i\in I \} $ of rectangles of the form $A_{i}= [ s_{i}%
,t_{i} ] \times [ u_{i},v_{i} ] $, whose union equals
$ [ 0,T ] ^{2}$ and which have pairwise disjoint interiors. The
collection of rectangular partitions is denoted $\mathcal{P}_{\mathrm{rec}} (  [ 0,T ] ^{2} ) $, and $R$ is said to have
\textit{controlled }$\rho$\textit{-variation} if%
%
\begin{equation}
\llvert R\rrvert _{\rho\mbox{-}\mathrm{var}; [ 0,T ]
^{2}%
}^{\rho}:=\mathop{\sup
_{ \{ A_{i}\dvtx i\in I \} \in\mathcal
{P}%
_{\mathrm{rec}} (  [ 0,T ] ^{2} ) }}_{A_{i}= [
s_{i},t_{i} ] \times [ u_{i},v_{i} ] }\sum_{i,j}
\biggl\llvert R\pmatrix{
s_{i},t_{i}
\vspace*{2pt}\cr
u_{i},v_{i}} \biggr\rrvert ^{\rho}<\infty. \label{controlled}%
\end{equation}
We obviously have $V_{\rho}( R; [ 0,T ] ^{2}) \leq\llvert
R\rrvert _{\rho\mbox{-}\mathrm{var}; [ 0,T ] ^{2}}$, and
it is
shown in \cite{FVupdate} that for every $\varepsilon>0$ there exists
$c_{p,\varepsilon}$ such that $\llvert  R\rrvert _{\rho\mbox{-}\mathrm{var}; [
0,T ] ^{2}} \leq c_{p,\varepsilon} V_{\rho+\varepsilon}( R; [
0,T ] ^{2})$. The main advantage of the quantity~(\ref{controlled})
compared to (\ref{2dvariation}) is that the map
\[
[ s,t ] \times [ u,v ] \mapsto\llvert R\rrvert _{\rho\mbox{-}\mathrm{var}; [ s,t ] \times [ u,v ]
}^{\rho}%
\]
is a 2D control in the sense of \cite{FVupdate}.

\begin{definition}
\label{hol rho var}Given $1\leq\rho<2$, we say that $R$ has \textit{finite
(two-dimen\-sional) H\"{o}lder-controlled }$\rho$\textit{-variation} if
$V_{\rho} ( R; [ 0,T ] \times [ 0,T ]
)
<\infty$, and if there exists $C>0$ such that for all $0\leq s\leq t
\leq T$
we have%
%
\begin{equation}
V_{\rho} \bigl( R; [ s,t ] \times [ s,t ] \bigr) \leq C ( t-s )
^{1/\rho}.
\end{equation}
\end{definition}

\begin{remark}
\label{remark f1} This is (essentially) without loss of generality
compared to
Definition \ref{rho var}. To see this, we note that if $R$ also has controlled
$\rho$-variation in the sense of (\ref{controlled}), then we can
introduce the
deterministic time-change $\tau\dvtx [0,T]\rightarrow[0,T]$ given by $\tau
=\sigma^{-1}$, where $\sigma\dvtx  [ 0,T ] \rightarrow [
0,T ] $ is the strictly increasing function defined by
%
\begin{equation}
\sigma ( t ):=\frac{T\llvert  R\rrvert _{\rho
\mbox{-}\mathrm{var}; [ 0,t ] ^{2}}^{\rho}}{\llvert
R\rrvert
_{\rho\mbox{-}\mathrm{var}; [ 0,T ] ^{2}}^{\rho}}. \label{parametrisation}%
\end{equation}
It is then easy to see that $\tilde{R}$, the covariance function of
$\tilde
{X}=X\circ\tau$, is H\"{o}lder-controlled in the sense of Definition
\ref{hol rho var}.
\end{remark}

Two important consequences of assuming that $R$ has finite
H\"{o}lder-cont\-rolled $\rho$-variation are: (i) $\mathbf{X}$ has
$1/p$-H\"{o}lder sample paths for every $p>2\rho$, and (ii) by using
\cite{FV}, Theorem~15.33, we can deduce that
%
\begin{equation}
E \bigl[ \exp \bigl( \eta\| \mathbf{X}\| _{1/p; [
0,T ]
}^{2} \bigr)
\bigr] <\infty\qquad\mbox{for some } \eta>0, \label{gauss}%
\end{equation}
that is, $\mathbf{\|X\|}_{1/p; [ 0,T ] }^{2}$ has a
Gaussian tail.

The mere existence of this lift is unfortunately not sufficient to
apply the
usual concepts of Malliavin calculus. In addition, it will be important to
require a complementary (Young) regularity of the sample paths of $X$
and the
elements of its Cameron--Martin space. The following assumption
captures both
of these requirements.

\begin{condition}
\label{standing assumption} Let $ ( X_{t} ) _{t\in [
0,T ] }= ( X_{t}^{1},\ldots,X_{t}^{d} ) _{t\in [
0,T ] } $ be a Gaussian process with i.i.d. components. Suppose that
the covariance function has finite H\"{o}lder-controlled $\rho
$-variation for
some $\rho\in [1,2)$. We will assume that $X$ has a natural lift
to a
geometric p-rough path and that $\mathcal{H}$, the Cameron--Martin space
associated with $X$, has Young-complementary regularity to $X$ in the
following sense: for some $q\geq1$ satisfying $1/p+1/q>1$, we have the
continuous embedding
\[
\mathcal{H}\hookrightarrow C^{q\mbox{-}\mathrm{var}} \bigl( [ 0,T ],%
\mathbb{R} 
^{d} \bigr).
\]
\end{condition}

The following theorem appears in \cite{FV07} as Proposition 17 (cf.
also the
recent note~\cite{FVupdate}); it shows how the assumption $V_{\rho
} ( R; [ 0,T ] ^{2} ) <\infty$ allows us to embed
$\mathcal{H}$ in the space of continuous paths with finite $\rho$ variation.
The result is stated in~\cite{FV07} for one-dimensional Gaussian processes,
but the generalisation to arbitrary finite dimensions is straightforward.

\begin{theorem}[({\cite{FV07}})]\label{CM_pVar_embedding}Let $ ( X_{t} )
_{t\in
[
0,T ] }= ( X_{t}^{1},\ldots,X_{t}^{d} ) _{t\in [
0,T ] }$ be a mean-zero Gaussian process with independent and
identically distributed components. Let $R$ denote the covariance
function of
(any) one of the components. Then if $R$ is of finite $\rho$-variation for
some $\rho\in [1,2)$ we can embed $\mathcal{H}$ in the space
$C^{\rho\mbox{-}\mathrm{var}} (  [ 0,T ],%
\mathbb{R}
^{d} ) $, in fact,%
%
\begin{equation}
\llvert h\rrvert _{\mathcal{H}}\geq\frac{\llvert
h\rrvert
_{\rho\mbox{-}\mathrm{var}; [ 0,T ] }}{\sqrt{V_{\rho} (
R; [ 0,T ] \times [ 0,T ]  ) }}. \label
{embedding}%
\end{equation}
\end{theorem}

\begin{remark}[{(\cite{FV3})}]\label{fBM embedding} Writing $\mathcal{H}^{H}$ for the
Cameron--Martin space of fBm for $H$ in $ ( 1/4,1/2 ) $, the
variation embedding in \cite{FV3} gives the stronger result that
\[
\mathcal{H}^{H}\hookrightarrow C^{q\mbox{-}\mathrm{var}} \bigl( [ 0,T ]
,%
\mathbb{R} 
^{d} \bigr)
\qquad\mbox{for any } q> ( H+1/2 ) ^{-1}.
\]
\end{remark}

Theorem $\ref{CM_pVar_embedding}$ and Remark \ref{fBM embedding} provide
sufficient conditions for a process to satisfy the fundamental Condition
\ref{standing assumption}, which we summarise in the following remark.

\begin{remark}
\label{remark f2} As already observed, the requirement that $R$ has finite
2D $\rho$-variation, for some $\rho\in [1,2)$, implies both
that $X$
lifts to a geometric $p$-rough path for all $p>2\rho$ and also that
$\mathcal{H} \hookrightarrow C^{\rho\mbox{-}\mathrm{var}} (  [
0,T ],%
\mathbb{R}
^{d} ) $ (Theorem \ref{CM_pVar_embedding}). Complementary
regularity of
$\mathcal{H}$ in the above condition thus can be obtained by $\rho\in
 [1,3/2)$, which covers for example BM, the OU process and the Brownian
bridge (in each case with $\rho=1$). When $X$ is fBm, we know that $X$
admit a
lift to $G\Omega_{p} (
\mathbb{R}
^{d} ) $ if $p>1/H$, and Remark \ref{fBM embedding} therefore ensures
the complementary regularity of $X$ and $\mathcal{H}$ if $H>1/4$.
\end{remark}

\section{Statement of the main theorem}

\label{main thm}

We will begin the section by laying out and providing motivation for the
assumptions we impose on the driving Gaussian signal $X$. We will then
end the
section with a statement of the central theorem of this paper, which is a
version of H\"{o}rmander's theorem for Gaussian RDEs. First, we give some
notation which will feature repeatedly.

\begin{notation}
We define
\[
\mathcal{F}_{a,b}:=\sigma \bigl( X_{v,v^{\prime}}\dvtx a\leq v\leq
v^{\prime
}\leq b \bigr)
\]
to be the $\sigma$-algebra generated by the increments of $X$ between times
$a$ and~$b$.
\end{notation}

The following condition aims to capture the nondegeneracy of $X$, it will
feature prominently in the sequel.

\begin{condition}[(Nondeterminism-type condition)]\label{nondeterm}Let $ (
X_{t} )
_{t\in [ 0,T ] }$ be a continuous Gaussian process. Suppose that
the covariance function $R$ of $X$ has finite H\"{o}lder-controlled
$\rho
$-variation for some $\rho$ in $[1,2)$. We assume that there exists
$\alpha>0$
such that
%
\begin{equation}
\inf_{0\leq s<t\leq T}\frac{1}{ ( t-s ) ^{\alpha}}%
\operatorname{Var}
( X_{s,t}|\mathcal{F}_{0,s}\vee\mathcal {F}%
_{t,T} ) >0. \label{index}%
\end{equation}
Whenever this condition is satisfied, we will call $\alpha$ the
\emph{index
of nondeterminism} if it is the smallest value of $\alpha$ for which
(\ref{index}) is true.
\end{condition}

\begin{remark}
It is worthwhile making a number of comments. First, notice that the
conditional variance
\[
\operatorname{Var} ( X_{s,t}|\mathcal{F}_{0,s}\vee
\mathcal {F}%
_{t,T} )
\]
is actually deterministic by Gaussian considerations. Then for any
$ [
s,t ] \subseteq [ 0,S ] \subseteq [ 0,T ]
$, the
law of total variance can be used to show that
\[
\operatorname{Var} ( X_{s,t}|\mathcal{F}_{0,s}\vee
\mathcal {F}%
_{t,S} ) \geq\operatorname{Var} (
X_{s,t}|\mathcal {F}_{0,s}%
\vee
\mathcal{F}_{t,T} ).
\]
It follows that if (\ref{index}) holds on $ [ 0,T ] $, then
it will
also hold on any interval $ [ 0,S ] \subseteq [
0,T ] $
provided $S>0$.
\end{remark}

Note that Condition \ref{nondeterm} implies the existence of $c>0$
such that
\[
\operatorname{Var} ( X_{s,t}|\mathcal{F}_{0,s}\vee
\mathcal {F}%
_{t,T} ) \geq c ( t-s ) ^{\alpha}.
\]
This is reminiscent of (but not equivalent to) other notions of
nondeterminism which have been studied in the literature. For example, it
should be compared to the similar notion introduced in \cite{berman},
where it
was exploited to show the existence of a smooth local time function
(see also
the subsequent work of Cuzick et al.~\cite{cuzick2} and~\cite{cuz}).
In the
present context, Condition \ref{nondeterm} is also related to the following
condition: for any $f$ of finite $p$-variation over $ [ 0,T
] $
%
\begin{equation}
\int_{0}^{T}f_{s}\,dh_{s}=0\qquad
\forall h\in\mathcal{H } \quad\Rightarrow\quad f=0\qquad \mbox{a.e. on $ [ 0,T ] $.}
\label{CFV nondeg}%
\end{equation}
This has been used in \cite{CF} to prove the existence of the density for
Gaussian RDEs. In some sense, our Condition \ref{nondeterm} is the
quantitative version of (\ref{CFV nondeg}). In this paper, when we
speak of a
nondegenerate Gaussian process $ ( X_{t} ) _{t\in [
0,T ] }$ we will mean the following.

\begin{definition}
\label{definition nondegeneracy}Let $ ( X_{t} ) _{t\in
[
0,T ] }$ be a continuous, real-valued Gaussian process. For any
partition $D= \{
t_{i}\dvtx i=0,1,\ldots,n \}$ of $ [ 0,T ]$, let
$(Q_{ij}^{D})_{1\leq i,j\leq n}$ denote the $n\times n$ matrix given by the
covariance matrix of the increments of $X$ along $D$, that is,%
%
\begin{equation}
Q_{ij}^{D}=R\pmatrix{
t_{i-1},t_{i}
\vspace*{2pt}\cr
t_{j-1},t_{j}}. \label{increment matrix}%
\end{equation}
We say that $X$ is nondegenerate if $Q^{D}$ is positive definite for every
partition $D$ of $ [ 0,T ] $.
\end{definition}

\begin{remark}
An obvious example of a ``degenerate'' Gaussian process is a bridge process
which returns to zero in $ [ 0,T ] $. This is plainly ruled
out by
an assumption of nondegeneracy.
\end{remark}

It is shown in \cite{CFV} that nondegeneracy is implied by (\ref{CFV
nondeg}%
). Thus, nondegeneracy is a weaker condition than (\ref{CFV nondeg}).
It also
has the advantage of being formulated more tangibly in terms of the covariance
matrix. The next lemma shows that Condition~\ref{nondeterm} also
implies that
the process is nondegenerate.

\begin{lemma}
Let $ ( X_{t} ) _{t\in [ 0,T ] }$ be a continuous
Gaussian process which satisfies Condition \ref{nondeterm} then $X$ is
nondegenerate.
\end{lemma}

\begin{pf}
Fix a partition $D$ of $ [ 0,T ] $, and denote the covariance
matrix along this partition by $Q$ with entries as in (\ref{increment
matrix}%
). If $Q$ is not positive definite, then for some nonzero vector
$\lambda= ( \lambda_{1},\ldots,\lambda_{n} ) \in%
\mathbb{R}
^{n}$ we have%
%
\begin{equation}
0=\lambda^{T}Q\lambda=E \Biggl[ \Biggl( \sum
_{i=1}^{n}\lambda_{i}%
X_{t_{i-1},t_{i}} \Biggr) ^{2} \Biggr]. \label{an}%
\end{equation}
Suppose, without loss of generality, that $\lambda_{j}\neq0$. Then from
(\ref{an}), we can deduce that
\[
X_{t_{j-1},t_{j}}=\sum_{i\neq j}^{n}
\frac{\lambda_{i}}{\lambda_{j}}%
X_{t_{i-1},t_{i}}%
\]
with probability one. This immediately implies that
\[
\operatorname{Var} ( X_{t_{j-1},t_{j}}|\mathcal {F}_{0,t_{j-1}}%
\vee\mathcal{F}_{t_{j},T} ) =0,
\]
which contradicts (\ref{index}).
\end{pf}

A crucial step in the proof of the main theorem is to establish lower bounds
on the eigenvalues of the Malliavin covariance matrix in order to obtain
moment estimates for its inverse. In the setting we have adopted, it
transpires that these eigenvalues can be bounded from below by some
power of
the 2D Young integral:%
%
\begin{equation}
\int_{ [ 0,T ] ^{2}}f_{s}f_{t}\,dR ( s,t )
\label{2d}%
\end{equation}
for some suitable (random) function $f\in C^{p\mbox{-}\mathrm{var}} (
[ 0,T ],%
\mathbb{R}
^{d} ) $. By considering the Riemann sum approximations to (\ref{2d}),
the problem of finding a lower bound can be re-expressed in terms of
solving a
sequence of finite-dimensional constrained quadratic programming
problems. By
considering the dual of these problems, we can simplify the constraints which
appear considerably; they become nonnegativity constraints, which are much
easier to handle. Thus, the dual problem has an explicit solution
subject to a
dual feasibility condition. The following condition is what emerges as the
limit of the dual feasibility conditions for the discrete approximations.

\begin{condition}
\label{cond dom}Let $ ( X_{t} ) _{t\in [ 0,T ]
}$ be a
continuous, real-valued Gaussian process. We will assume that $X$ has
nonnegative conditional covariance in that for every $ [
u,v ]
\subseteq [ s,t ] \subseteq [ 0,S ] \subseteq
[
0,T ] $ we have
%
\begin{equation}
\operatorname{Cov} ( X_{s,t},X_{u,v}|
\mathcal{F}_{0,s}\vee \mathcal{F}_{t,S} ) \geq0.
\label{con dom eq}%
\end{equation}
\end{condition}

In Section~\ref{interpol}, we will prove a novel interpolation
inequality. The
significance of Condition \ref{cond dom} will become clearer when we work
through the details of that section. For the moment, we content
ourselves with
an outline. First, for a finite partition $D$ of the interval $ [
0,T ]$, one can consider the discretisation of the process $X_{t}$
conditioned on the increments in $D \cap (  [ 0,s ]
\cup [ t,T ]  ) $. Let $Q^{D}$ be the corresponding
covariance matrix of the increments [see $ ( \ref{increment
matrix}%
)$]. Then the conditional covariance $\operatorname{Cov} (
X_{s,t}^{D},X_{u,v}^{D}|\mathcal{F}_{0,s}^{D}\vee\mathcal
{F}_{t,T}^{D} )
$ of the discretised process can be characterised in terms of a Schur
complement $\Sigma$ of the matrix $Q^{D}$. Using this relation, the
condition
\[
\operatorname{Cov} \bigl( X_{s,t}^{D},X_{u,v}^{D}|
\mathcal {F}_{0,s}^{D}%
\vee\mathcal{F}_{t,T}^{D}
\bigr) \geq0
\]
is precisely what ensures that the row sums for $\Sigma$ are nonnegative.
Conversely, if for any finite partition $D$ all Schur complements of the
matrix $Q^{D}$ have nonnegative row sums, then Condition~\ref{cond
dom} is
satisfied. This relation motivates an alternative sufficient condition that
implies Condition~\ref{cond dom}, which has the advantage that it may
be more
readily verified for a given Gaussian process. In order to state the
condition, recall that an $n\times n$ real matrix $Q$ is diagonally dominant
if
%
\begin{equation}
Q_{ii}\geq\sum_{j\neq i}\llvert
Q_{ij}\rrvert \qquad\mbox{for every $i\in \{ 1,2,\ldots,n \} $.}
\label{diag}%
\end{equation}

\begin{condition}
\label{diagonal dominance}Let $ ( X_{t} ) _{t\in [
0,T ]
} $ be a continuous real-valued Gaussian process. For every $ [
0,S ] \subseteq [ 0,T ]$, we assume that $X$ has diagonally
dominant increments on $ [ 0,S ] $. By this, we mean that
for every
partition $D= \{ t_{i}\dvtx i=0,1,\ldots,n \} $ of $ [
0,S ]
$, the $n\times n$ matrix $ ( Q_{ij}^{D} ) _{1\leq i,j\leq
n}$ with
entries
\[
Q_{ij}^{D}=E [ X_{t_{i-1},t_{i}}X_{t_{j-1},t_{j}} ] =R
\pmatrix{
t_{i-1},t_{i}
\vspace*{2pt}\cr
t_{j-1},t_{j}}
\]
is diagonally dominant.
\end{condition}

Diagonal dominance is obviously in general a stronger assumption than
requiring that a covariance matrix has positive row sums. Consequently,
Condition \ref{diagonal dominance} is particularly useful for negatively
correlated processes, when diagonal dominance of the increments and positivity
of row sums are the same. The condition can then be expressed
succinctly as
\[
E [ X_{0,S}X_{s,t} ] \geq0\qquad \forall [ s,t ] \subseteq [ 0,S ]
\subseteq [ 0,T ].
\]
In fact, it turns out that Condition \ref{diagonal dominance} implies
Condition \ref{cond dom}. This is not obvious a priori, and ultimately depends
on two nice structural features. The first is the observation from linear
algebra that the property of diagonal dominance is preserved under taking
Schur complements (see \cite{zhang} for a proof of this). The second results
from the interpretation of the Schur complement (in the setting of Gaussian
vectors) as the covariance matrix of a certain conditional
distribution. We
will postpone the proof of this until Section~\ref{interpol} when these
properties will be used extensively.

The final condition we will impose is classical, namely H\"{o}rmander's
condition on the vector fields defining the RDE.

\begin{condition}[(H\"{o}rmander)]\label{horm}We assume that%
%
\begin{eqnarray}\label{Lie}%
&&\operatorname{span} \bigl\{ V_{1},\ldots,V_{d}, [
V_{i},V_{j} ], \bigl[ V_{i}, [
V_{j},V_{k} ] \bigr],\ldots\dvtx
\nonumber
\\[-8pt]
\\[-8pt]
\nonumber
&&\hspace*{76pt}\qquad  i,j,k,\ldots =0,1,\ldots,d
\bigr\} |_{y_{0}}=\mathcal{T}_{y_{0}}%
\mathbb{R} 
^{e} \cong%
\mathbb{R} 
^{e}.
\end{eqnarray}
\end{condition}

We are ready to formulate our main theorem.

\begin{theorem}
\label{main theorem}Let $ ( X_{t} ) _{t\in [
0,T ]
}= ( X_{t}^{1},\ldots,X_{t}^{d} ) _{t\in [ 0,T
] }$
be a
continuous Gaussian process, with i.i.d. components associated to the abstract
Wiener space $ ( \mathcal{W},\mathcal{H},\mu ) $. Assume
that some
(and hence every) component of $X$ satisfies:

\begin{longlist}[(1)]
\item[(1)] Condition \ref{standing assumption}, for some $\rho\in [1,2)$;

\item[(2)] Condition \ref{nondeterm}, with index of nondeterminacy
$\alpha<2/\rho$;

\item[(3)] Condition \ref{cond dom}, that is, it has nonnegative
conditional covariance.
\end{longlist}

Fix $p>2\rho$, and let $\mathbf{X\in}G\Omega_{p} ( \mathbb{R}
^{d} ) $ denote the canonical lift of $X$ to a Gaussian rough path.
Suppose $V= ( V_{1},\ldots,V_{d} ) $ is a collection of
$C^{\infty
}$-bounded vector fields on $%
\mathbb{R}
^{e}$, and let $ ( Y_{t} ) _{t\in [ 0,T ] }$ be the
solution to the RDE
\[
dY_{t}=V ( Y_{t} ) \,d\mathbf{X}_{t}+V_{0}
( Y_{t} ) \,dt,\qquad Y ( 0 ) =y_{0}.
\]
Assume that the collection $ ( V_{0},V_{1},\ldots,V_{d} ) $ satisfy
H\"{o}rmander's condition, Condition \ref{horm}, at the starting point $y_{0}$. Then random variable $Y_{t}$ has a smooth density with respect to Lebesgue
measure on $%
\mathbb{R}
^{e}$ for every $t\in(0,T]$.
\end{theorem}

\section{Examples}
\label{examples section}

In this section, we demonstrate how the conditions on $X$ we introduced
in the
last section can be checked for a number of well-known processes. We
choose to
focus on three particular examples: fractional Brownian motion (fBm) with
Hurst parameter $H>1/4$, the Ornstein--Uhlenbeck (OU) process and the
fractional Brownian bridge (fBb) with Hurst parameter $1/3<H\leq1/2$. Together,
these encompass a broad range of Gaussian processes that one encounters in
practice. Of course, there are many more examples, but these should be checked
on a case-by-case basis by analogy with our presentation for these core
examples. We first remark that Condition~\ref{standing assumption} is
straightforward to check in all these cases (see, e.g., \cite{FV} and
\cite{CFV}). Proving that the fBb (returning at $T^{\prime}>0$) with $H>1/3$
satisfies Condition \ref{standing assumption} is a simple calculation
in a
similar style.

 We will now commence with a verification of the nondeterminism
condition, that is, Condition~\ref{nondeterm}.

\subsection{Nondeterminism-type condition}

Recall that the Cameron--Martin space~$\mathcal{H}$ is defined to be
the
completion of the linear space of functions of the form
\[
\sum_{i=1}^{n}a_{i}R (
t_{i},\cdot ),\qquad a_{i}\in%
\mathbb{R} 
\mbox{ and }t_{i}\in [ 0,T ],
\]
with respect to the inner product
\[
\Biggl\langle\sum_{i=1}^{n}a_{i}R
( t_{i},\cdot ),\sum_{j=1}%
^{m}b_{j}R ( s_{j},\cdot ) \Biggr
\rangle_{\mathcal
{H}}=\sum_{i=1}^{n}\sum
_{j=1}^{m}a_{i}b_{j}R
( t_{i},s_{j} ).
\]
Some authors prefer instead to work with the set of step functions
$\mathcal{E}$
\[
\mathcal{E=} \Biggl\{ \sum_{i=1}^{n}a_{i}1_{ [ 0,t_{i} ]
}\dvtx a_{i}
\in%
\mathbb{R} 
,t_{i}\in [ 0,T ] \Biggr\},
\]
equipped with the inner product
\[
\langle1_{ [ 0,t ] },1_{ [ 0,s ] } \rangle _{\mathcal{\tilde{H}}}=R ( s,t ).
\]
If $\mathcal{\tilde{H}}$ denote the completion of $\mathcal{E}$ w.r.t.
$ \langle
\cdot,\cdot \rangle_{\mathcal{\tilde{H}}}$, then it is obvious
that the
linear map $\phi\dvtx \mathcal{E}\rightarrow\mathcal{H}$ defined by
%
\begin{equation}
\phi ( 1_{ [ 0,t ] } ) =R ( t,\cdot ) \label{isom}%
\end{equation}
extends to an isometry between $\mathcal{\tilde{H}}$ and $\mathcal
{H}$. We
also recall that $\mathcal{\tilde{H}}$ is isometric to the Hilbert space
$H^{1} ( Z ) \subseteq L^{2} ( \Omega,\mathcal
{F},P ) $
which is defined to be the $\llvert \cdot\rrvert _{L^{2} (
\Omega ) }$-closure of the set
\[
\Biggl\{ \sum_{i=1}^{n}a_{i}Z_{t_{i}}\dvtx a_{i}
\in%
\mathbb{R} 
,t_{i}\in [ 0,T ], n\in%
\mathbb{N} 
\Biggr\}.
\]
In particular, we have that $\llvert 1_{ [ 0,t ] }\rrvert
_{\mathcal{\tilde{H}}}=\llvert  Z_{t}\rrvert _{L^{2} (
\Omega ) }$. We will now prove that Condition~\ref{nondeterm} holds
whenever it is the case that $\mathcal{\tilde{H}}$ embeds
continuously in
$L^{q} (  [ 0,T ]  ) $ for some $q\geq1$. Hence,
Condition \ref{nondeterm} will simplify in many cases to showing that
\[
|\tilde{h}|_{L^{q} [ 0,T ] }\leq C|\tilde{h}|_{\mathcal
{\tilde
{H}}%
}
\]
for some $C>0$ and all $\tilde{h}\in\mathcal{\tilde{H}}$.

\begin{lemma}
\label{CM embed general}Suppose $ ( Z_{t} ) _{t\in [
0,T ] }$ is a continuous real-valued Gaussian processes. Assume
that for
some $q\geq1$ we have $\mathcal{\tilde{H}\hookrightarrow}L^{q} (
[
0,T ]  ) $. Then $Z$ satisfies Condition \ref{nondeterm} with
index of nondeterminacy less than or equal to $2/q$, that is,
\[
\inf_{0\leq s<t\leq T}\frac{1}{ ( t-s ) ^{2/q}}%
\operatorname{Var}
( Z_{s,t}|\mathcal{F}_{0,s}\vee\mathcal {F}%
_{t,T} ) >0.
\]
\end{lemma}

\begin{pf}
Fix $ [ s,t ] \subseteq [ 0,T ] $ and for
brevity let
$\mathcal{G}$ denote the $\sigma$-algebra $\mathcal{F}_{0,s}\vee
\mathcal{F}_{t,T}$. Then, using the fact that $\operatorname{Var} (
Z_{s,t}|\mathcal{G} ) $ is deterministic and positive, we
have
%
\begin{eqnarray*}
\Var ( Z_{s,t}|\mathcal{G} ) &=& \bigl\llVert \Var ( Z_{s,t}|
\mathcal{G} ) \bigr\rrVert _{L^{2} ( \Omega
) } = E \bigl[ E \bigl[ \bigl(
Z_{s,t}-E [ Z_{s,t}|\mathcal{G} ] \bigr) ^{2}|
\mathcal{G} \bigr] ^{2} \bigr] ^{1/2}
\\
&= &E \bigl[ \bigl( Z_{s,t}-E [ Z_{s,t}|\mathcal{G} ] \bigr)
^{2}%
\bigr] = \bigl\llVert Z_{s,t}-E [
Z_{s,t}|\mathcal{G} ] \bigr\rrVert _{L^{2} ( \Omega ) }^{2}
\\
&= &\inf_{Y\in L^{2} ( \Omega,\mathcal{G},P ) } \llVert Z_{s,t}-Y\rrVert
_{L^{2} ( \Omega )
}^{2}.
\end{eqnarray*}
We can therefore find sequence of random variables $ (
Y_{n} ) _{n=1}^{\infty}\subset L^{2} ( \Omega,\mathcal
{G},P )
$ such that%
%
\begin{equation}
\llVert Z_{s,t}-Y_{n}\rrVert _{L^{2} ( \Omega )
}^{2}=E
\bigl[ ( Z_{s,t}-Y_{n} ) ^{2} \bigr] \downarrow
\operatorname{Var} ( Z_{s,t}|\mathcal{G} ). \label
{approx}%
\end{equation}
Moreover, because $E [ Z_{s,t}|\mathcal{G} ] $ belongs to
the closed
subspace $H^{1} ( Z ) $, we can assume that $Y_{n}$ has the form
\[
Y_{n}=\sum_{i=1}^{k_{n}}a_{i}^{n}Z_{t_{i}^{n},t_{i+1}^{n}}
\]
for some sequence of real numbers
$ \{ a_{i}^{n}\dvtx i=1,\ldots,k_{n} \}$
and a collection of subintervals
\[
\bigl\{ \bigl[ t_{i}^{n},t_{i+1}^{n}
\bigr] \dvtx i=1,\ldots,k_{n} \bigr\},
\]
which satisfy $ [ t_{i}^{n},t_{i+1}^{n} ] \subseteq [
0,s ] \cup [ s,T]$ for every $n\in%
\mathbb{N}
$.

We now exhibit a lower bound for $\llVert  Z_{s,t}-Y_{n}\rrVert
_{L^{2} ( \Omega ) }^{2}$ which is independent of $n$ [and hence
from (\ref{approx}) will apply also to $\operatorname{Var} (
Z_{s,t}|\mathcal{G} ) $]. Let us note that the isometry between the
$H^{1} ( Z ) $ and $\mathcal{\tilde{H}}$ yields
%
\begin{equation}
\llVert Z_{s,t}-Y_{n}\rrVert _{L^{2} ( \Omega )
}^{2}=|
\tilde {h}_{n}|_{\mathcal{\tilde{H}}}^{2}, \label{isom1}%
\end{equation}
where
\[
\tilde{h}_{n} ( \cdot ):=\sum_{i=1}^{k_{n}}a_{i}^{n}1_{ [
t_{i}^{n},t_{i+1}^{n} ] }
( \cdot ) +1_{ [
s,t ]
} ( \cdot ) \in\mathcal{\tilde H}.
\]
The embedding $\mathcal{\tilde{H}\hookrightarrow}L^{q} (  [
0,T ]  ) $ then shows that
\[
|\tilde{h}_{n}|_{\mathcal{\tilde{H}}}^{2}\geq c|\tilde
{h}|_{L^{q} [
0,T ] }^{2}\geq c ( t-s ) ^{2/q}.
\]
The result follows immediately from this together with (\ref{approx}) and
(\ref{isom1}).
\end{pf}

Checking that $\mathcal{\tilde{H}}$ embeds continuously in a suitable
$L^{q} (  [ 0,T ]  ) $ space is something which is
readily done for our three examples. This is what the next lemma shows.

\begin{lemma}
\label{CM embed}If $ ( Z_{t} ) _{t\in [ 0,T ]
}$ is fBm
with Hurst index $H\in ( 0,1/2 ) $ and $q\in [1,2)$ then
$\mathcal{\tilde{H}} \hookrightarrow L^{q} (  [ 0,T ]
)
$. If $ ( Z_{t} ) _{t\in [ 0,T ] }$ is the (centred)
Ornstein--Uhlenbeck process or the Brownian bridge (returning to zero after
time $T$) then $\mathcal{\tilde{H}}$ $\hookrightarrow L^{2} (
[
0,T ]  ) $.
\end{lemma}

\begin{pf}
The proof for each of the three examples has the same structure. We first
identify an isometry $K^{\ast}$ which maps $\mathcal{\tilde{H}}$ surjectively
onto $L^{2} [ 0,T ] $. (The operator $K^{\ast}$ is of course
different for the three examples.) We then prove that the inverse
$ (
K^{\ast} ) ^{-1}$ is a bounded linear operator when viewed as a
map from
$L^{2} [ 0,T ] $ into $L^{q} [ 0,T ] $. For fBm
this is
shown via the Hardy--Littlewood lemma (see \cite{nualart}). For the OU process
and the Brownian bridge, it follows from a direct calculation on the operator
$K^{\ast}$. Equipped with this fact, we can deduce that
%
\begin{eqnarray*}
|\tilde{h}|_{L^{q} [ 0,T ] } &=&\bigl\llvert \bigl( K^{\ast
} \bigr)
^{-1}K^{\ast}\tilde{h}\bigr\rrvert _{L^{q} [ 0,T ] } \leq\bigl
\llvert \bigl( K^{\ast} \bigr) ^{-1}\bigr\rrvert
_{L^{2}\rightarrow L^{q}}\bigl\llvert K^{\ast}\tilde{h}\bigr\rrvert
_{L^{2}%
[ 0,T ] }
\\
& =&\bigl\llvert \bigl( K^{\ast} \bigr) ^{-1}
\bigr\rrvert _{L^{2}\rightarrow
L^{q}}|\tilde{h}|_{\mathcal{\tilde{H}}},
\end{eqnarray*}
which completes the proof.
\end{pf}

\begin{remark}
We can verify the condition in the case of the fBb by more direct
means. One
representation of the fBb is of the form
%
\begin{equation}
X_{t}=B_{t}-a_{t}B_{T}\qquad\mbox{with }
a_{t}=\frac{t^{2H}+T^{2H}%
-(T-t)^{2H}}{2T^{2H}}. \label{eq:anticip-rep-fbb}%
\end{equation}
Then
\begin{eqnarray*}
\operatorname{Var} ( X_{s,t}|\mathcal{F}_{0,s}\vee
\mathcal {F}%
_{t,T} ) &\geq&\operatorname{Var} (
X_{s,t}|\mathcal {F}_{0,s}%
\vee
\mathcal{F}_{t,T}, B_{T} ) =\operatorname{Var} \bigl(
B_{s,t}|\mathcal{F}_{0,s}^{B}\vee
\mathcal{F}_{t,T}^{B} \bigr)\\
& \asymp& c r^{2H}.
\end{eqnarray*}
\end{remark}

As an immediate corollary of the last two lemmas, we can conclude that the
(centred) Ornstein--Uhlenbeck process and the Brownian bridge
(returning to
zero after time $T$) both satisfy Condition~\ref{nondeterm} with
index of
nondeterminism no greater than unity. In the case of fBm $(Z_{t}^{H}%
)_{t\in [ 0,T ] }$, the scaling properties of $Z^{H}$ enable
us to
say more about the nondeterminism index than can be obtained by an immediate
application of Lemmas \ref{CM embed general} and \ref{CM embed}. To
see this,
note that for fixed $ [ s,t ] \subseteq [ 0,T ]
$ we can
introduce a new process
\[
\tilde{Z}_{u}^{H}:= ( t-s ) ^{-H}Z_{u ( t-s ) }^{H}.
\]
$\tilde{Z}$ defines another fBm, this time on the interval $[0,T (
t-s ) ^{-1}]=:[0,\tilde{T}]$. Let $u=s ( t-s ) ^{-1}%
,v=t ( t-s ) ^{-1}$ and denote by $\mathcal{\tilde
{F}}_{a,b}$ the
$\sigma$-algebra generated by the increments of $\tilde{Z}$ in $ [
a,b ] $. Scaling then allows us to deduce that
%
\begin{equation}
\operatorname{Var} ( Z_{s,t}|\mathcal{F}_{0,s}\vee
\mathcal {F}%
_{t,T} ) = ( t-s ) ^{2H}\mathop{
\operatorname{Var}}(\tilde {Z}%
_{u,v}|\mathcal{
\tilde{F}}_{0,u}\vee\mathcal{\tilde{F}}_{v,\tilde{T}}).
\label{scaling}%
\end{equation}
By construction $u-v=1$. And since $\tilde{Z}$ is fBm it follows from Lemmas
\ref{CM embed general} and~\ref{CM embed} that
%
\begin{equation}
\mathop{\inf_{[ u,v]\subseteq [0,\tilde{T}],
}}_{|u-v|=1}%
\mathop{
\operatorname{Var}}(\tilde{Z}_{u,v}|\mathcal{\tilde{F}}_{0,u}%
\vee\mathcal{\tilde{F}}_{v,\tilde{T}})>0. \label{em}%
\end{equation}
It follows from (\ref{scaling}) and (\ref{em}) that $Z^{H}$ satisfies
Condition \ref{nondeterm} with index of nondeterminacy no greater
than $2H$.

\subsection{Nonnegativity of the conditional covariance}

We finally verify that our example processes also satisfy Condition
\ref{cond dom}. We first consider the special case of process with negatively
correlated increments.

\subsubsection{Negatively correlated increments}

From our earlier discussion, it suffices to check that Condition
\ref{diagonal dominance} holds. In other words, that $Q^{D}$ is diagonally
dominant for every partition $D$. This amounts to showing that
\[
E [ Z_{t_{i-1},t_{i}} Z_{0,T} ] \geq0
\]
for every $0\leq t_{i-1}<t_{i}\leq T$. It is useful to have two general
conditions on $R$ which will guarantee that (i) the increments of $Z$ are
negatively correlated, and (ii) diagonal dominance is satisfied. Here
is a
simple characterisation of these properties:

\textit{Negatively correlated increments}: If $i<j$, write
\[
Q_{ij}=E [ Z_{t_{i-1},t_{i}} Z_{t_{j-1},t_{j}} ] =\int
_{t_{i-1}
}^{t_{i}}\int_{t_{j-1}}^{t_{j}}
\partial_{ab}^{2}R(a,b) \,da \,db,
\]
so that a sufficient condition for $Q_{ij}<0$ is $\partial_{ab}^{2}%
R(a,b)\leq0$ for $a<b$. This is trivially verified for fBm with
$H<1/2$. Note
that the distributional derivative $\partial_{ab}^{2}R(a,b)$ might be singular
on the
diagonal, but the diagonal is avoided here.

\textit{Diagonal dominance}: If we assume negatively correlated
increments, then diagonal dominance is equivalent to $\sum_{j}Q_{ij}>0$.
Moreover, if we assume $Z_{0}$ is deterministic and $Z$ is centred we get
\[
\sum_{j}Q_{ij}=E [ Z_{t_{i-1},t_{i}}
Z_{T} ] =\int_{t_{i-1}
}^{t_{i}}
\partial_{a}R(a,T) \,da,
\]
so that a sufficient condition for $\sum_{j}Q_{ij}\geq0$ is $\partial
_{a}R(a,b)\geq0$ for $a<b$. This is again trivially verified for fBm with
$H<1/2$.

\begin{example}
In the case where $ ( Z_{t} ) _{t\in [ 0,T ] }$
is the
Brownian bridge, which returns to zero at time $T^{\prime}>T$ we have
\[
R ( a,b ) =\frac{a}{T^{\prime}} \bigl( T^{\prime}-b \bigr) \qquad\mbox{for
$a<b$.}%
\]
It is then immediate that $\partial_{ab}^{2}R(a,b)=-1/T^{\prime}<0$
$ $and
$\partial_{a}R(a,b)=1-b/T^{\prime}>0$. Similarly, for the centred
Ornstein--Uhlenbeck process, we have
\[
R ( a,b ) =2e^{-b}\sinh ( a ) \qquad\mbox{for $a<b$.}%
\]
From which it follows that $\partial_{ab}^{2}R(a,b)=-2e^{-b}\cosh
(
a ) <0$ and $\partial_{a}R(a,b)=2e^{-b}\cosh ( a )
>0$. In
the more general case of the fractional Brownian bridge returning to
zero at time $T^{\prime}>T$,
the covariance function is given for $a<b$ by
\[
R(a,b)=\frac{1}{2}R_H(a,b) -\frac
{1}{2(T^{\prime})^{2H}}R_H
\bigl(a,T'\bigr) R_H\bigl(b,T'\bigr),
\]
where we used the shorthand $R_H(a,b) = a^{2H}+b^{2H}-(b-a)^{2H}$.
Thus,
\begin{eqnarray*}
\partial_{a}R(a,b)&=&H \bigl[ a^{2H-1}+(b-a)^{2H-1}
\bigr]\\
&&{} -\frac{H}%
{(T^{\prime})^{2H}} \bigl[ a^{2H-1}+\bigl(T^{\prime}-a
\bigr)^{2H-1} \bigr] R_H\bigl(b,T'\bigr).
\end{eqnarray*}
This is positive, since $R_H(b,T') \leq(T^{\prime})^{2H}$ and
$(T^{\prime}-a)^{2H-1}\leq(b-a)^{2H-1}$ whenever
$H<1/2$. The fact that $\partial_{ab}R(a,b)\leq0$ is also easily seen.
\end{example}

\subsubsection{Without negatively correlated increments}

In the three examples, we were able to check Condition \ref{cond dom}
by using
the negative correlation of the increments and showing explicitly the diagonal
dominance. In the case where the increments have positive or mixed correlation,
we may have to check the weaker condition, Condition \ref{cond dom}, directly.
An observation that might be useful in this regard is the following
geometrical interpretation. Recall that we want to want to check that
\[
\operatorname{Cov} ( Z_{s,t},Z_{u,v}|
\mathcal{F}_{0,s}\vee \mathcal{F}_{t,T} ) \geq0.
\]
For simplicity, let $X=Z_{s,t}$, $Y=Z_{u,v}$ and $\mathcal{G=F}_{0,s}%
\vee\mathcal{F}_{t,T}$. The map $P_{\mathcal{G}}\dvtx Z\mapsto E [
Z|\mathcal{G} ] $ then defines a projection from the Hilbert space
$L^{2} ( \Omega,\mathcal{F},P ) $ onto the closed subspace
$L^{2} ( \Omega,\mathcal{G},P ) $, which gives the orthogonal
decomposition
\[
L^{2} ( \Omega,\mathcal{F},P ) =L^{2} ( \Omega,
\mathcal{G}%
,P ) \oplus L^{2} ( \Omega,\mathcal{G},P )
^{\perp}.
\]
A simple calculation then yields
%
\begin{eqnarray*}
\cov ( X,Y|\mathcal{G} )& =&E \bigl[ \cov ( X,Y|\mathcal {G} ) %
\bigr]
=E \bigl[ ( I-P_{\mathcal{G}} ) X ( I-P_{\mathcal
{G}} ) Y%
\bigr]
\\
&=&
\bigl\langle P_{\mathcal{G}}^{\perp}X,P_{\mathcal{G}}^{\perp
}Y
\bigr\rangle_{L^{2} ( \Omega ) },
\end{eqnarray*}
where $P_{\mathcal{G}}^{\perp}$ is the projection onto $L^{2} (
\Omega,\mathcal{G},P ) ^{\perp}$. In other words,
$\operatorname{Cov} ( X,Y|\mathcal{G} ) \geq0$ if and
only if
$\cos\theta\geq0$, where $\theta$ is the angle between the projections
$P_{\mathcal{G}}^{\perp}X$ and $P_{\mathcal{G}}^{\perp}Y$ of, respectively,
$X$ and $Y$ onto the orthogonal complement of $L^{2} ( \Omega
,\mathcal{G},P ) $.

\section{A Norris-type lemma}
\label{Norris lemma section}

In this section, we generalise a deterministic version of the Norris lemma,
obtained in \cite{H3} for $p$ rough paths with $1<p<3$, to general
$p>1$. It
is interesting to note that the assumption on the driving noise we make is
consistent with \cite{H3}. In particular, it still only depends on the
roughness of the basic path and not the rough path lift.

\subsection{Norris' lemma}
To simplify the notation, we will assume that $T=1$ in this subsection; all
the work will therefore be done on the interval $ [ 0,1 ] $. Our
Norris-type lemma relies on the notion of controlled process, which we proceed
to define now. Recall first the definition contained in \cite{Gu} for
second-order rough paths: whenever $\mathbf{x}\in C^{0,\gamma
}([0,1];G^{N}%
(\mathbb{R}^{d}))$ with $\gamma>1/3$, the space $\mathcal
{Q}_{\mathbf{x}
}(\mathbb{R})$ of controlled processes is the set of functions $y\in
C^{\gamma}([0,1];\mathbb{R})$ such that the increment $y_{st}$ can be
decomposed as
\[
y_{st}=y_{s}^{i}x_{s,t}^{i}+r_{s,t},
\]
where the remainder term $r$ satisfies $|r_{s,t}|\leq
c_{y}|t-s|^{2\gamma}$
and where we have used the summation over repeated indices convention. Notice
that $y$ has to be considered in fact as a vector $(y,y^{1},\ldots,y^{d})$.

In order to generalise this notion to lower values of $\gamma$, we
shall index
our controlled processes by words based on the alphabet $\{1,\ldots,d\}
$. To
this end, we need the following additional notation.

\begin{notation}
Let $w=(i_{1},\ldots,i_{n})$ and $\bar w=(j_{1},\ldots,j_{m})$ be two words
based on the alphabet $\{1,\ldots,d\}$. Then $|w|=n$ denotes the
length of
$w$, and $w\bar w$ stands for the concatenation $(i_{1},\ldots,i_{n}$,
$j_{1},\ldots,j_{m})$ of $w$ and $\bar w$. For $L\ge1$, $\mathcal{W}_{L}$
denotes the set of words of length at most $L$.
\end{notation}

Let us now turn to the definition of controlled process based on a
rough path.

\begin{definition}
\label{def:controlled-paths} Let $\mathbf{x}\in C^{0,\gamma}([0,1];G^{N}
(\mathbb{R}^{d}))$, with $\gamma>0$, $N=[1/\gamma]$. A~controlled
path based
on $\mathbf{x}$ is a family $(y^{w})_{w\in\mathcal{W}_{N-1}}$ indexed
by words
of length at most $N-1$, such that for any word $w\in\mathcal
{W}_{N-2}$ we
have
%
\begin{equation}
\qquad y_{s,t}^{w}=\sum_{\bar{w}\in\mathcal{W}_{N-1-|w|}}y_{s}^{w\bar{w}}%
\mathbf{x}_{s,t}^{\bar{w}}+r_{s,t}^{w}\qquad
\mbox{where } \bigl|r_{s,t}^{w}\bigr|\leq c_{y}|t-s|^{(N-|w|)\gamma}.
\label{eq:dcp-controlled-path}%
\end{equation}
In order to take the drift term of \eqref{Int-Eq} into account, we
also assume
that for $w=\varnothing$ we get a decomposition for the increment
$y_{s,t}$ of
the form
%
\begin{equation}\qquad
y_{s,t}=\sum_{\bar{w}\in\mathcal{W}_{N-1}}y_{s}^{\bar{w}}
\mathbf {x}%
_{s,t}^{\bar{w}}+ y_{s}^{0}
(t-s) + r_{st}^{y} \qquad\mbox{where }\bigl |r_{s,t}^{y}\bigr|
\leq c_{y}|t-s|^{N\gamma}. \label{eq:dcp-increment-y}%
\end{equation}
The set of controlled processes is denoted by $\mathcal{Q}_{\mathbf{x}
}^{\gamma}$, and the norm on $\mathcal{Q}_{\mathbf{x}}^{\gamma}$ is
given by
\[
\Vert y\Vert_{\mathcal{Q}_{\mathbf{x}}^{\gamma}}= \bigl\Vert y^{0}\bigr\Vert _{\gamma} + \sum
_{w\in\mathcal{W}_{N-1}}\bigl\Vert y^{w}\bigr\Vert_{\gamma}.
\]
\end{definition}

We next recall the definition of $\theta$-H\"{o}lder-roughness
introduced in
\cite{H3}.

\begin{definition}\label{def:rough}
Let $\theta\in ( 0,1 ) $. A path $x\dvtx [ 0,T ]
\rightarrow\mathbb{R}^{d}$ is called $\theta$-H\"{o}lder rough if
there exists
a constant $c>0$ such that for every $s$ in $ [ 0,T ] $, every
$\varepsilon$ in $(0,T/2]$, and every $\phi$ in $\mathbb{R}^{d}$ with
$\llvert
\phi\rrvert =1$, there exists $t$ in $ [ 0,T ] $ such that
$\varepsilon/2<\llvert  t-s\rrvert <\varepsilon$ and
\[
\bigl\llvert \langle\phi,x_{s,t} \rangle\bigr\rrvert >c\varepsilon
^{\theta}.
\]
The largest such constant is called the modulus of $\theta$-H\"{o}lder
roughness, and is denoted $L_{\theta} ( x ) $.
\end{definition}

A first rather straightforward consequence of this definition is that
if a
rough path $\mathbf{x}$ happens to be H\"{o}lder rough, then the
\emph{derivative processes} $y^{w}$ in the decomposition
\eqref{eq:dcp-controlled-path} of a controlled path $y$ is uniquely determined
by $y$. This can be made quantitative in the following way.

\begin{proposition}
\label{prop:Nlemma1} Let $\mathbf{x}\in C^{0,\gamma}([0,1];G^{N}%
(\mathbb{R}^{d}))$, with $\gamma>0$ and $N=[1/\gamma]$. We also
assume that
$x$ is $\theta$-H\"{o}lder rough for some $\theta<2\gamma$. Let $y$
be a
real-valued controlled path defined as in
Definition~\ref{def:controlled-paths}, and set $\mathcal
{Y}_{n}(y)=\sup_{|w|=n}\Vert
y^{w}\Vert_{\infty}$ for $n\le N-1$. Then there exists a constant $M$
depending only on $d $
such that the bound
%
\begin{equation}
\mathcal{Y}_{n}(y)\leq \frac{M  (\|y\|_{\mathcal{Q}_{\mathbf{x}}^{\gamma}}
\mathcal{N}_{\mathbf{x}}  )^{{\theta}/{(2\gamma)}}
}{L_{\theta}(x)} \mathcal{Y}_{n-1}^{1-{\theta}/{(2\gamma)}}(y)
\label{eq:bnd-Yny}%
\end{equation}
holds for every controlled rough path $\mathcal{Q}_{\mathbf
{x}}^{\gamma
}$ and
every $1\le n \le N-1$.
\end{proposition}

\begin{pf}
For sake of clarity, we shall assume that $y^{0}=0$, leaving to the patient
reader the straightforward adaptation to a nonvanishing drift coefficient.
Now start from the decomposition \eqref{eq:dcp-controlled-path} and
recast it as
\[
y_{s,t}^{w}=\sum_{j=1}^{d}y_{s}^{wj}
\mathbf{x}_{s,t}^{(j)}+\sum_{2\leq
|\bar
{w}|\leq N-1-|w|}y_{s}^{w\bar{w}}
\mathbf{x}_{s,t}^{\bar{w}}+r_{s,t}^{w},
\]
where we have set $wj$ for the concatenation of the word $w$ and the word
$(j)$ for notational sake. This identity easily yields
%
\begin{eqnarray}\label{eq:dcp-controlled-path2}%
\sup_{|t-s|\leq\varepsilon}\Biggl\llvert \sum_{j=1}^{d}y_{s}^{wj}
\mathbf{x} 
_{s,t}^{(j)}\Biggr\rrvert &\leq&2\bigl\Vert
y^{w}\bigr\Vert_{\infty}\nonumber\\
&&{}+\sum_{2\leq
|\bar
{w}|\leq N-1-|w|}\bigl\Vert
y^{w\bar{w}}\bigr\Vert_{\infty}\bigl\Vert\mathbf {x}^{\bar
{w}%
}
\bigr\Vert_{\gamma|\bar{w}|} \varepsilon^{|\bar{w}| \gamma}\\
&&\hspace*{76pt}{}+\bigl\Vert r^{w}%
\bigr\Vert_{\gamma(N-|w|)} \varepsilon^{(N-|w|) \gamma}. \nonumber
\end{eqnarray}
Since $x$ is $\theta$-H\"{o}lder rough by assumption,
there exists $v$, which is independent of j, with $\varepsilon/2\leq
|v-s|\leq\varepsilon$ such that
%
\begin{equation}
\Biggl\llvert \sum_{j=1}^{d}y_{s}^{wj}
\mathbf{x}_{s,v}^{(j)}\Biggr\rrvert >L_{\theta}(x)
\varepsilon^{\theta}\bigl| \bigl( y_{s}^{w1},
\ldots,y_{s}%
^{wd} \bigr) \bigr|.
\label{eqn:low-bnd-ywj}%
\end{equation}
Combining both \eqref{eq:dcp-controlled-path2} and \eqref
{eqn:low-bnd-ywj} for
all words $w$ of length $n-1$, we thus obtain that
%
\begin{eqnarray*}
\mathcal{Y}_{n}(y)&\leq&\frac{c}{L_{\theta}(x)}  \biggl[
\mathcal{Y}%
_{n-1}(y) \varepsilon^{-\theta}
\\
&&\hspace*{35pt}{} +\sup_{|w|=n-1} \biggl(\sum_{2\leq|\bar{w}|\leq N-1-|w|}
\bigl\Vert y^{w\bar
{w}%
}\bigr\Vert_{\infty}\bigl\Vert\mathbf{x}^{\bar{w}}
\bigr\Vert_{\gamma|\bar{w}%
|} \varepsilon^{|\bar{w}| \gamma-\theta}\\
&&\hspace*{151pt}{}+\bigl\Vert r^{w}
\bigr\Vert_{\gamma
(N-|w|)} \varepsilon^{(N-|w|) \gamma-\theta} \biggr) \biggr].
\end{eqnarray*}
Let us further simplify this relation by recalling that we take
supremums over words $w$ such that $|w|=n-1\le N-2$, so that $N-|w|\ge
2$, and we also consider words $\bar{w}$ whose length is at least $2$.
This yields
\[
\mathcal{Y}_{n}(y)\leq\frac{c}{L_{\theta}(x)} \bigl( \mathcal{Y}_{n-1}(y)
\varepsilon^{-\theta} + \|y\|_{\mathcal{Q}_{\mathbf{x}}^{\gamma}} \mathcal{N}_{\mathbf{x},\gamma}
\varepsilon^{2\gamma-\theta} \bigr).
\]
One can optimise the right-hand side of the previous inequality over
$\varepsilon$, by choosing $\varepsilon$ such that the term $\mathcal{Y}
_{n-1}(y) \varepsilon^{-\theta}$ is of the same order as $\mathcal
{N}_{\mathbf{x},\gamma} \varepsilon^{2\gamma-\theta}$. One then
verifies that our claim \eqref{eq:bnd-Yny} follows from
this elementary computation.
\end{pf}

\begin{remark}
Definition \ref{def:controlled-paths} and Proposition \ref
{prop:Nlemma1} can
be generalised to $d$-dimensional controlled processes. In
particular, if $y$ is a $d$-dimensional path, the decomposition
\eqref{eq:dcp-increment-y} becomes
%
\begin{equation}
y_{s,t}^{i}=\sum_{\bar{w}\in\mathcal{W}_{N-1}}y_{s}^{i,\bar
{w}}
\mathbf {x}%
_{s,t}^{\bar{w}}+r_{s,t}^{i,y}\qquad
\mbox{where }\bigl |r_{s,t}^{i,y}\bigr|\leq c_{y}|t-s|^{N\gamma}
\label{eq:dcp-increment-y-d-dim}%
\end{equation}
for all $i=1,\ldots,d$.
\end{remark}

We now show how the integration of controlled processes fits into the general
rough paths theory. For this, we will use the nonhomogeneous norm
$\mathcal{N}_{\mathbf{x},\gamma}=\mathcal{N}_{\mathbf{x},\gamma
, [
0,1 ] }$ introduced in (\ref{inhomogeneous}).

\begin{proposition}
\label{prop:integral-ctrld-path} Let $y$ be a $d$-dimensional controlled
process, given as in Definition \ref{def:controlled-paths} and whose
increments can be written as in \eqref{eq:dcp-increment-y-d-dim}. Then
$(\mathbf{x},\mathbf{y})$ is a geometrical rough path in
$G^{N}(\mathbb
{R}%
^{2d})$. In particular, for $(s,t)\in\Delta^{2}$, the integral $I_{st}
\equiv\int_{s}^{t}y_{s}^{i} \,dx_{s}^{i}$ is well defined and admits the
decomposition
%
\begin{equation}
I_{s,t}=\sum_{j=1}^{d} \biggl(
y_{s}^{j}x_{s,t}^{j}+\sum
_{\bar{w}\in
\mathcal{W}_{N-1}}y_{s}^{\bar{w}}\mathbf{x}_{s,t}^{\bar{w}j}
\biggr) +r_{s,t}%
^{I}, \label{eq:dcp-Ist}%
\end{equation}
where $|r_{s,t}^{I}|\leq\mathcal{N}_{\mathbf{x}}\Vert\mathbf
{y}\Vert
_{\gamma
}|t-s|^{(N+1)\gamma}$.
\end{proposition}

\begin{pf}
Approximate $x$ and $y$ by smooth functions $x^{m},y^{m}$, while preserving
the controlled process structure (namely $y^{m}\in\mathcal
{Q}_{\mathbf
{x}^{m}%
}$). Then one can easily check that $(x^{m},y^{m})$ admits a signature, and
that $I_{s,t}^{m}\equiv\int_{s}^{t}y_{s}^{m,i} \,dx_{s}^{m,i}$ can be
decomposed as \eqref{eq:dcp-Ist}. Limits can then be taken thanks to
\cite{Gu10}, which completes the proof.
\end{pf}

The following theorem is a version of Norris' lemma, and constitutes
the main
result of this section.

\begin{theorem}
\label{thm:NlemmaRP} Let $\mathbf{x}$ be a geometric rough path of order
$N\geq1$ based on the $\mathbb{R}^{d}$-valued function $x$. We also assume
that $x$ is a $\theta$-H\"{o}lder rough path with $2\gamma>\theta$. Let
$y $
be a $\mathbb{R}^{d}$-valued controlled path of the form given in Definition
\ref{def:controlled-paths}, $b\in C^{\gamma}([0,1])$, and set
\[
z_{t}=\sum_{i=1}^{d}\int
_{0}^{t}y_{s}^{i}
\,dx_{s}^{i}+\int_{0}^{t}%
b_{s} \,ds=I_{st}+\int_{0}^{t}b_{s}
\,ds.
\]
Then there exist constants $r>0$ and $q>0$ such that, setting
%
\begin{equation}
\mathcal{R}=1+{L_{\theta}(x)}^{-1}+\mathcal{N}_{\mathbf{x},\gamma}%
+\Vert\mathbf{y}\Vert_{\mathcal{Q}_{\mathbf{x}}^{\gamma}}+\Vert b\Vert _{C^{\gamma}},
\label{eq:def-R}%
\end{equation}
one has the bound
\[
\Vert y\Vert_{\infty}+\Vert b\Vert_{\infty}\leq M\mathcal{R}^{q}
\Vert z\Vert_{\infty}^{r}
\]
for a constant $M$ depending only on $T$, $d$ and $y$.
\end{theorem}

\begin{pf}
We shall divide this proof in several steps. In the following computations,
the symbol $\kappa$ will stand for an exponent for $\mathcal{R}$ and
$M$ will stand
for an arbitrary multiplicative constant. The exact values of these two
constants
are irrelevant and can change from line to line without warning.

\textit{Step \textup{1:} Bounds on $y$}. Combining
\eqref{eq:dcp-Ist}, the bound on $r^{I}$ given in
Proposition~\ref{prop:integral-ctrld-path} and the definition of
$\mathcal{R}%
$, we easily get the relation
%
\begin{equation}
\label{e:boundz} \|z\|_{\infty} \le M \mathcal{R}^{\kappa}.
\end{equation}
We now resort to relation \eqref{eq:bnd-Yny} applied to the controlled path
$z$ and for $n=1$, which means that $\mathcal{Y}_{n}(z)\asymp\|y\|
_{\infty}$
and $\mathcal{Y}_{n-1}(z)\asymp\|z\|_{\infty}$. With the definition of
$\mathcal{R}$ in mind, this yields the bound
%
\begin{equation}
\label{eq:bnd-norris1}\|y\|_{\infty} \le M \|z\|_{\infty}^{1-{\theta
}/{(2\gamma)}}
\mathcal{R}^{\kappa},
\end{equation}
which corresponds to our claim for $y$.

Along the same lines and thanks to relation \eqref{eq:bnd-Yny} for
$n>1$, we
iteratively get the bounds
%
\begin{equation}
\label{eq:bnd-Yny-z}\mathcal{Y}_{n}(y) \le M \|z\|_{\infty
}^{(1-{\theta}/{(2\gamma)})^{n}}
\mathcal{R}^{\kappa},
\end{equation}
which will be useful in order to complete the bound we have announced
for~$b$.

\textit{Step \textup{2:} Bounds on $r^{I}$ and $I$.} In order
to get
an appropriate bound on $r$, it is convenient to consider $\mathbf{x}$
as a
rough path with H\"older regularity $\beta<\gamma$, still satisfying the
inequality $2\beta>\theta$. Notice furthermore that $\mathcal{N}%
_{\mathbf{x},\beta}\le\mathcal{N}_{\mathbf{x},\gamma}$. Consider now
$w\in\mathcal{W}_{n}$. According to \eqref{eq:bnd-Yny-z}, we have
\[
\bigl\|y^{w}\bigr\|_{\infty} \le M \|z\|_{\infty}^{(1-{\theta}/{(2\gamma)
})^{n}%
}
\mathcal{R}^{\kappa},
\]
while $\|y^{w}\|_{\gamma}\le M \mathcal{R}$ by definition. Hence,
invoking the
inequality
\[
\bigl\|y^{w}\bigr\|_{\beta} \le2 \bigl\|y^{w}
\bigr\|^{{\beta}/{\gamma}}_{\gamma} \bigl\| y^{w}\bigr\|^{1
- {\beta}/{\gamma}}_{\infty}
,
\]
which follows immediately from the definition of the H\"older norm, we obtain
the bound
\[
\bigl\|y^{w}\bigr\|_{\beta} \le M \|z\|_{\infty}^{(1-{\theta}/{(2\gamma)}%
)^{n}(1-{\beta}/{\gamma})}
\mathcal{R}^{\kappa},
\]
which is valid for all $w\in\mathcal{W}_{n}$ and all $n\le N-1$.
Summing up,
we end up with the relation
\[
\|\mathbf{y}\|_{\beta} \le M \|z\|_{\infty}^{(1-{\theta
}/{(2\gamma)
})^{N-1}(1-{\beta}/{\gamma})}
\mathcal{R}^{\kappa}.
\]

Now according to Proposition \ref{prop:integral-ctrld-path}, we get
$r_{s,t}^{I}\leq\mathcal{N}_{\mathbf{x},\beta}\Vert\mathbf
{y}\Vert
_{\beta
}|t-s|^{(N+1)\beta}$ and the above estimate yields
\[
\bigl\Vert r^{I}\bigr\Vert_{(N+1)\beta}\leq M \Vert z\Vert_{\infty}^{(1-{\theta
}/{(2\gamma)})^{N-1}(1-{\beta}/{\gamma})}
\mathcal{R}^{\kappa}.
\]
Plugging this estimate into the decomposition \eqref{eq:dcp-Ist} of $I_{st}$
we end up with
%
\begin{equation}
\Vert I\Vert_{\infty}\leq M \Vert z\Vert_{\infty}^{(1-{\theta
}/{(2\gamma)
})^{N-1}(1-{\beta}/{\gamma})}
\mathcal{R}^{\kappa}. \label
{eq:bnd-I-infty}%
\end{equation}

\textit{Step \textup{3:} Bound on $b$.} Combining the bound
\eqref{eq:bnd-I-infty} with \eqref{e:boundz} and the fact that the exponent
of $\|z\|_\infty$ appearing in \eqref{eq:bnd-I-infty} is less than $1$,
we have
\[
\biggl\|\int_{0}^{\cdot}b_{s} \,ds
\biggr\|_{\infty}\leq M \Vert z\Vert _{\infty
}^{(1-{\theta}/{(2\gamma)})^{N-1}(1-{\beta}/{\gamma
})}
\mathcal{R}%
^{\kappa}.
\]
Once again, we use an interpolation inequality to strengthen this bound.
Indeed, we have (see \cite{HM11}, Lemma 6.14, for further details)
\[
\Vert\partial_{t}f\Vert_{\infty}\leq M\Vert f\Vert_{\infty}
\max \biggl({\frac{1}{T}},\Vert f\Vert_{\infty}^{-{{1}/{(\gamma+1)}}}\Vert
\partial_{t}f\Vert_{\gamma}^{{1}/{(\gamma+1)}} \biggr),
\]
and applying this inequality to $f_{t}=\int_{0}^{t}b_{s} \,ds$, it follows
that
%
\begin{equation}
\Vert b\Vert_{\infty}\leq M \Vert z\Vert_{\infty}^{(1-{\theta
}/{(2\gamma)
})^{N-1}(1-{\beta}/{\gamma})({\gamma}/{(\gamma
+1)})}
\mathcal {R}^{\kappa
}. \label{eq:bnd-norris2}%
\end{equation}
Gathering the bounds \eqref{eq:bnd-norris1} and \eqref{eq:bnd-norris2}, our
proof is now complete.
\end{pf}

\begin{remark}
One might be motivated to consider situations in which the drift and
the noise
have different natural parameterisations (see, e.g., the recent work
\cite{FS}). More precisely suppose $\mathbf{X}$ is a Gaussian rough
path in
$WG\Omega_{p} (
\mathbb{R}
^{d} ) $ (with general $p$-variation regularity) and let $Y$ be the
solution to
\[
dY_{t}=V ( Y_{t} ) \,d\mathbf{X} + V_{0} (
Y_{t} ) \,dt, \qquad Y ( 0 ) =y_{0}.
\]
Then, as we have already observed in Remark \ref{remark f1}, we can
use the
parameterisation $\tau\dvtx  [ 0,T ] \rightarrow [
0,T ]$,
the inverse of $\sigma$ in (\ref{parametrisation}), to force $\tilde
{X}%
_{t}:=X_{\tau ( t ) }$ to have a H\"{o}lder-controlled covariance
function. This leads us to consider the solution to
%
\begin{equation}
d\tilde{Y}_{t}=V ( \tilde{Y}_{t} ) \,d\mathbf{\tilde{X}} +
V_{0} ( \tilde{Y}_{t} ) \,d\tau(t), \qquad\tilde{Y} ( 0 )
=y_{0},\label{repar}%
\end{equation}
whereupon $\tilde{Y}_{t}=Y_{\tau ( t ) }$. In particular, for
proving smoothness of the density of $Y_{T}(=\tilde{Y}_{T})$, one
needs never
to consider any parameterisation in which the noise is not of H\"{o}lder-type
regularity. This is a useful remark because Condition \ref{nondeterm}
explicitly involves the H\"{o}lder-parameterisation. To deal with the
situation presented by (\ref{repar}), one should adapt the previous
theorem to
accommodate RDEs of the form
\[
z_{t}=\sum_{i=1}^{d}\int
_{0}^{t}y_{s}^{i}
\,dx_{s}^{i}+\int_{0}^{t}b_{s}%
\,d\tau ( s ).
\]
\end{remark}

\subsection{Small-ball estimates for \texorpdfstring{${L}_{\theta}(X)$}{L_{theta}(X)}}

We now take $X$ to be a Gaussian process satisfying Condition \ref
{nondeterm}%
. As the reader might have noticed, equation \eqref{eq:def-R} above involves
the random variable $L_{\theta} ( X ) ^{-1}$, for which we will
need some tail estimates. The nondeterminism condition naturally gives rise
to such estimates as the following lemma makes clear.

\begin{lemma}
\label{sbp}Suppose $ ( X_{t} ) _{t\in [ 0,T ]
}$ is a
zero-mean, $%
\mathbb{R}
^{d}$-valued, continuous Gaussian process with i.i.d. components, with each
component having a continuous covariance function $R$. Suppose that one (and
hence every) component of $X$ satisfies Condition \ref{nondeterm}. Let
$\alpha_{0}>0$ be the index of nondeterminism for $X$ and suppose
$\alpha
\geq\alpha_{0}$. Then there exist positive and finite constants
$C_{1}$ and
$C_{2}$ such that for any interval $I_{\delta}\subseteq [
0,T
] $
of length $\delta$ and $0<x<1$ we have
%
\begin{equation}
P \Bigl( \inf_{\llvert \phi\rrvert =1}\sup_{s,t\in
I_{\delta}%
}\bigl\llvert
\langle\phi,X_{s,t} \rangle\bigr\rrvert \leq x \Bigr) \leq
C_{1}\exp \bigl( -C_{2}\delta x^{-2/\alpha} \bigr).
\label
{sbpbound}%
\end{equation}
\end{lemma}

\begin{pf}
The proof is similar to Theorem 2.1 of Monrad and Rootzen \cite
{monrad}; we
need to adapt it because our nondeterminism condition is different.

We start by introducing two simplifications. First, for any $\phi$ in $
\mathbb{R}
^{d}$ with $\llvert \phi\rrvert =1$, we have
%
\begin{equation}
\bigl( \langle\phi,X_{t} \rangle \bigr) _{t\in [
0,T ] }\stackrel{D}
{=} \bigl( X_{t}^{1} \bigr) _{t\in [
0,T ]
},
\label{distr}%
\end{equation}
which implies that
%
\begin{equation}
P \Bigl( \sup_{s,t\in I_{\delta}}\bigl\llvert \langle\phi
,X_{s,t}%
\rangle\bigr\rrvert \leq x \Bigr) =P \Bigl( \sup
_{s,t\in
I_{\delta}
}\bigl\llvert X_{s,t}^{1}\bigr\rrvert
\leq x \Bigr). \label{cmp}%
\end{equation}
We will prove that the this probability is bounded above by
\[
\exp \bigl( -c\delta x^{2/\alpha} \bigr)
\]
for a positive real constant $c$, which will not depend on $T$, $\delta
$ or
$x$. The inequality~(\ref{sbpbound}) will then follow by a well-known
compactness argument (see \cite{H3} and~\cite{N}). The second simplification
is to assume that $\delta=1$. We can justify this by working with the scaled
process
\[
\tilde{X}_{t}=\delta^{\alpha/2}X_{t/\delta},
\]
which is still a Gaussian process only now defined on the interval
$[0,\tilde{T}]:= [ 0,T\delta ]$. Furthermore, the scaled process
also satisfies Condition \ref{nondeterm} since
%
\begin{eqnarray*}
\Var ( \tilde{X}_{s,t}|\mathcal{\tilde{F}}_{0,s}\vee\mathcal
{\tilde {F}}%
_{t,\tilde{T}} ) &=&\delta^{\alpha}\Var (
X_{s/\delta
,t/\delta}|%
\mathcal{F}_{0,s/\delta}\vee
\mathcal{F}_{t/\delta,T} )
\\
&\geq& c\delta^{\alpha} \biggl( \frac{t-s}{\delta} \biggr) ^{\alpha}
=c ( t-s ) ^{\alpha}.
\end{eqnarray*}
Thus, if we can prove the result for intervals of length $1$, we can
deduce the bound on (\ref{cmp}) we want from the identity
\[
P \Bigl( \sup_{s,t\in I_{\delta}}\bigl\llvert X_{s,t}^{1}
\bigr\rrvert \leq x \Bigr) =P \biggl( \sup_{s,t\in I_{1}} \bigl|
\tilde{X}_{s,t}^{1}%
\bigr| \leq\frac{x}{\delta^{\alpha/2}}
\biggr).
\]

To complete the proof, we begin by defining the natural number
$n:=\lfloor
x^{-2/\alpha}\rfloor\geq1$ and the dissection $D ( I )
= \{
t_{i}\dvtx i=0,1,\ldots,n+1 \} $ of $I=I_{1}$, given by
%
\begin{eqnarray*}
t_{i} &=&\inf I+ix^{2/\alpha},\qquad i=0,1,\ldots,n,
\\
t_{n+1} &=&\inf I+1=\sup I.
\end{eqnarray*}
Then it is trivial to see that%
%
\begin{equation}
P \Bigl( \sup_{s,t\in I}\bigl\llvert X_{s,t}^{1}
\bigr\rrvert \leq x \Bigr) \leq P \Bigl( \max_{i=1,2,\ldots,n}
|X_{t_{i-1},t_{i}}^{1} |\leq x \Bigr). \label{discrete small ball}%
\end{equation}
To estimate (\ref{discrete small ball}), we successively condition on the
components of
\[
\bigl(X_{t_{0},t_{1}}^{1},\ldots,X_{t_{n-1},t_{n}}^{1}
\bigr).
\]
More precisely, the distribution of $X_{t_{n-1},t_{n}}^{1}$ conditional on
$(X_{t_{0},t_{1}}^{1},\ldots,\break X_{t_{n-2},t_{n-1}}^{1})$ is Gaussian
with a
variance $\sigma^{2}$. Condition \ref{nondeterm} ensures that
$\sigma
^{2}$ is
bounded below by~$cx^{2}$. When $Z$ is a Gaussian random variable with fixed
variance, $P ( \llvert  Z\rrvert \leq x ) $ will be maximised
when the mean is zero. We therefore obtain the following upper bound:
\[
P \Bigl( \sup_{s,t\in I}\bigl\llvert X_{s,t}^{1}
\bigr\rrvert \leq x \Bigr) \leq \biggl(\int_{-x/\sigma}^{x/\sigma}
\frac{1}{\sqrt{2\pi}}\exp \biggl( -\frac{1}{2}y^{2} \biggr) \,dy
\biggr)^{n}.
\]
Using $x/\sigma\leq\sqrt{c}$, we can finally deduce that
\[
P \Bigl( \sup_{s,t\in I}\bigl\llvert X_{s,t}^{1}
\bigr\rrvert \leq x \Bigr) \leq\exp ( -Cn ) \leq\exp \biggl( -\frac{Cx^{-2/\alpha
}}{2}
\biggr),
\]
where $C:=\log [ 2\Phi ( \sqrt{c} ) -1 ]
^{-1}\in (
0,\infty ) $.
\end{pf}

\begin{remark}
As well as \cite{monrad}, these small-ball estimates should be compared
to the
estimates obtained by Li and Linde in \cite{Li} and Molchan \cite{M}
in the
case of fractional Brownian motion.
\end{remark}

\begin{corollary}
\label{l theta integrability}Suppose $ ( X_{t} ) _{t\in
[
0,T ] }$ is a zero-mean, $%
\mathbb{R}
^{d}$-valued, continuous Gaussian process with i.i.d. components satisfying
the conditions of Lemma~\ref{sbp}. Then for every $\theta>\alpha/2$,
the path
$ ( X_{t} ) _{t\in [ 0,T ] }$ is almost surely
$\theta
$-H\"{o}lder rough. Furthermore, for $0<x<1$, there exist positive finite
constants $C_{1}$ and $C_{2}$ such that the modulus of $\theta$-H\"{o}lder
roughness, $L_{\theta} ( X ) $, satisfies
\[
P \bigl( L_{\theta} ( X ) <x \bigr) \leq C_{1}\exp \bigl(
-C_{2}x^{-2/\alpha} \bigr).
\]
In particular, under these assumptions we have that $L_{\theta} (
X ) ^{-1}$ is in $\bigcap_{p>0}L^{p} ( \Omega ) $.
\end{corollary}

\begin{pf}
The argument of \cite{H3} applies in exactly the same way to show that
$L_{\theta} ( X ) $ is bounded below by
\[
\frac{1}{2\cdot8^{\theta}}D_{\theta} ( X ),
\]
where
\[
D_{\theta} ( X ):=\inf_{\llVert \phi\rrVert
=1}\inf_{n\geq1}
\inf_{k\leq2^{n}}\sup_{s,t\in I_{k,n}}\frac{\llvert
\langle
\phi,X_{s,t} \rangle\rrvert }{ ( 2^{-n}T )
^{\theta}}
\]
and $I_{k,n}:= [  ( k-1 ) 2^{-n}T,k2^{-n}T ] $.
We can
deduce that for any $x\in ( 0,1 ) $
\[
P \bigl( D_{\theta} ( X ) <x \bigr) \leq\sum_{n=1}^{\infty}%
\sum_{k=1}^{2^{n}}P \biggl(\inf
_{\llVert \phi\rrVert
=1}\sup_{s,t\in
I_{k,n}}\frac{\llvert  \langle\phi,X_{s,t} \rangle
\rrvert
}{ ( 2^{-n}T ) ^{\theta}}<x
\biggr),
\]
whereupon we can apply Lemma \ref{sbp} to yield
\[
P \bigl( D_{\theta} ( X ) <x \bigr) \leq c_{1}\sum
_{n=1}^{\infty
}2^{n}\exp \bigl(
-c_{2}2^{-n ( 1-2\theta/\alpha )
}T^{-2\theta
/\alpha}x^{-2/\alpha} \bigr).
\]
By exploiting the fact that $\theta>\alpha/2$, we can then find positive
constants $c_{3}$ and $c_{4}$ such that
%
\begin{eqnarray*}
P \bigl( D_{\theta} ( X ) <x \bigr) &\leq &c_{3}\sum
_{n=1}^{\infty
}\exp \bigl( -c_{4}nx^{-2/\alpha}
\bigr) =c_{3}\frac{\exp ( -c_{4}x^{-2/\alpha} ) }{1-\exp (
-c_{4}x^{-2/\alpha} ) }
\\
&\leq& c_{5}\exp \bigl( -c_{4}x^{-2/\alpha} \bigr),
\end{eqnarray*}
which completes the proof.
\end{pf}

\section{An interpolation inequality}
\label{interpol}

Under the standing assumptions on the Gaussian process, the \textit{Malliavin
covariance matrix} of the random variable $U_{t\leftarrow0}^{\mathbf
{X}%
} ( y_{0} ) \equiv Y_{t}$ can be represented as a 2D Young
integral (see \cite{CFV})%
%
\begin{equation}
C_{t}=\sum_{i=1}^{d}\int
_{ [ 0,t ] ^{2}}J_{t\leftarrow
s}^{\mathbf{X}} ( y_{0} )
V_{i} ( Y_{s} ) \otimes J_{t\leftarrow s^{\prime}}^{\mathbf{X}} (
y_{0} ) V_{i} ( Y_{s^{\prime}} ) \,dR \bigl(
s,s^{\prime} \bigr). \label{eq:def-malliavin-matrix}%
\end{equation}

In practice, showing the smoothness of the density boils down to getting
integrability estimates on the inverse of $\inf_{\llVert  v\rrVert
=1}v^{T}C_{T}v$, the smallest eigenvalue of~$C_{T}$. For this reason,
we will
be interested in
\[
v^{T}C_{T}v=\sum_{i=1}^{d}
\int_{ [ 0,T ] ^{2}} \bigl\langle v,J_{t\leftarrow s}^{\mathbf{X}} (
y_{0} ) V_{i} ( Y_{s} ) \bigr\rangle \bigl
\langle v,J_{t\leftarrow s^{\prime}%
}^{\mathbf{X}} ( y_{0} ) V_{i} (
Y_{s^{\prime}} ) \bigr\rangle \,dR \bigl( s,s^{\prime} \bigr).
\]
We will return to study the properties of $C_{T}$ more extensively in Section~\ref{section proof main theorem}. For the moment, we look to
generalise this
perspective somewhat. Suppose $f\dvtx [ 0,T ] \rightarrow%
\mathbb{R}
$ is some (deterministic) real-valued H\"{o}lder-continuous function, where
$\gamma$ is Young-complementary to $\rho$, 2D-variation regularity of
$R$. Our
aim in this section is elaborate on the nondegeneracy of the 2D Young
integral%
%
\begin{equation}
\int_{ [ 0,T ] }f_{s}f_{t}\,dR ( s,t ).
\label
{1st2D}%
\end{equation}
More precisely, what we want is to use Conditions \ref{nondeterm} and
\ref{cond dom} to give a quantitative version of the nondegeneracy
statement
%
\begin{equation}
\int_{ [ 0,T ] }f_{s}f_{t}\,dR ( s,t ) =0
\quad\Rightarrow \quad f\equiv0. \label{cfv}%
\end{equation}
To give an idea of the type of estimate we might aim for, consider the case
where $R\equiv R^{\mathrm{BM}}$ is the covariance function of Brownian motion.
The 2D
Young integral~(\ref{1st2D}) then collapses to the square of the $L^{2}$-norm
of $f$:%
%
\begin{equation}
\biggl\llvert \int_{ [ 0,T ] }f_{s}f_{t}\,dR^{\mathrm{BM}}
( s,t ) \biggr\rrvert =\llvert f\rrvert _{L^{2} [ 0,T ] }^{2},
\label{bmcov}%
\end{equation}
and the interpolation inequality (Lemma A3 of \cite{HP}) gives%
%
\begin{equation}
\llVert f\rrVert _{\infty; [ 0,T ] }\leq2\max \bigl( T^{-1/2}\llvert f
\rrvert _{L^{2} [ 0,T ] },\llvert f\rrvert _{L^{2} [ 0,T ] }^{2\gamma/ ( 2\gamma
+1 )
}\llVert
f\rrVert_{\gamma\mbox{\fontsize{8.36pt}{10pt}\selectfont{-H\"{o}l}}; [0,T ]}^{1/( 2\gamma+1 )}\bigr). \label{itl}%
\end{equation}
Therefore, in the setting of Brownian motion at least, (\ref{itl}) and
(\ref{bmcov}) quantifies~(\ref{cfv}). The problem is that the proof of
(\ref{itl}) relies heavily properties of the $L^{2}$-norm, in
particular, we
use the fact that
\[
\mbox{if }f ( u ) \geq c>0\mbox{ for all }u\in [ s,t ] \mbox{ then }\llvert f
\rrvert _{L^{2} [ s,t ]
}\geq c ( t-s ) ^{1/2}.
\]
We cannot expect for this to naively generalise to inner products resulting
from other covariance functions. We therefore have to re-examine the
proof of
the inequality (\ref{itl}) with this generalisation in mind.

It is easier to first consider a discrete version of the problem.
Suppose $D$
is some (finite) partition of $ [ 0,T ] $. Then the Riemann sum
approximation to (\ref{1st2D}) along $D$ can be written as
\[
f ( D ) ^{T}Qf ( D ),
\]
where $Q$ is the matrix (\ref{increment matrix}) and $f (
D ) $
the vector with entries given by the values of $f$ at the points in the
partition. The next sequence of results is aimed at addressing the
following question.

%
\begin{problem}
Suppose $\llvert  f\rrvert _{\infty; [ s,t ] }\geq
1$ for some
interval $ [ s,t ] \subseteq [ 0,T ] $. Can we
find a
positive lower bound $f ( D ) ^{T}Qf ( D ) $
which holds
uniformly over some sequence of partitions whose mesh tends to zero?
\end{problem}

To take a first step toward securing an answer, let $D= \{
t_{i}\dvtx i=0,1,\ldots,n \} $ and define
\[
Z:= ( Z_{1},\ldots,Z_{n} ):= ( X_{t_{0},t_{1}},\ldots,X_{t_{n-1}
,t_{n}}
) \sim N ( 0,Q ).
\]
Suppose that $Q$ has the block decomposition
\[
Q=\pmatrix{
Q_{11} &
Q_{12}
\vspace*{2pt}\cr
Q_{12}^{T} & Q_{22}}\qquad\mbox{with } Q_{11}\in\mathbb{R}^{k,k},Q_{12}
\in \mathbb{R}^{k,n-k},Q_{22}\in\mathbb{R}^{n-k,n-k}.
\]
In other words, $Q_{11}$ is the covariance matrix of $ ( Z_{1}%
,\ldots,Z_{k} ) $ and $Q_{22}$ is the covariance matrix of $ (
Z_{k+1},\ldots,Z_{n} ) $.  We are interested in finding the infimum
of the
quadratic form $x^{T}Qx$ over the subset
\[
\bigl\{ ( x_{1},\ldots,x_{n} ) \in%
\mathbb{R} 
^{n}\dvtx x_{j}\geq b,\forall
j=k+1,\ldots,n \bigr\},
\]
where $b>0$. To simplify the problem, we recall that the description of the
condition distribution
\[
( Z_{k+1},\ldots,Z_{n} ) |\sigma ( Z_{1},\ldots,Z_{k}
) \sim N ( \bar{\mu},\bar{Q} ),
\]
where the mean and covariance are given by
\[
\bar{\mu}=Q_{12}^{T}Q_{11}^{-1} (
Z_{1},\ldots,Z_{k} ) ^{T},\qquad
\bar{Q}=Q_{22}-Q_{12}^{T}Q_{11}^{-1}Q_{12}.
\]
$\bar{Q}$ is the so-called \textit{Schur complement} of $Q_{11}$ in
$Q$. It
follows that $x_{1}Z_{1}+x_{2}Z_{2}+\cdots+x_{n}Z_{n}| ( Z_{k+1}%
,\ldots,Z_{n} ) \sim N ( \sum_{i=1}^{k}x_{i}Z_{i}+\sum_{i=k+1}
^{n}x_{i}\bar{\mu}_{i}$, $\sum_{i,j=1}^{k}x_{i}\bar
{Q}_{i,j}x_{j} )$, and
hence
\begin{eqnarray*}
E \bigl[ ( x_{1}Z_{1}+\cdots+x_{n}Z_{n}
) ^{2} \bigr] & =&E \bigl[ E \bigl[ ( x_{1}Z_{1}+x_{2}Z_{2}+\cdots+x_{n}Z_{n}
) ^{2}%
|\sigma ( Z_{1},\ldots,Z_{k} )
\bigr] \bigr]
\\
& =&\sum_{i,j=1}^{k}x_{i}
\bar{Q}_{i,j}x_{j}+E \Biggl[ \Biggl( \sum
_{i=1}%
^{k}x_{i}\bar{
\mu}_{i}+\sum_{i=k+1}^{n}x_{i}Z_{i}
\Biggr) ^{2} \Biggr].
\end{eqnarray*}
We may always choose the unconstrained variables $x_{1},\ldots,x_{k}$  in order
that the second term is zero, therefore,
%
\begin{equation}
\qquad\inf_{x_{k+1}\geq b,\ldots,x_{n}\geq b}E \bigl[ ( x_{1}Z_{1}+\cdots+x_{n}%
Z_{n} ) ^{2} \bigr] =\inf_{x_{k+1}\geq b,\ldots,x_{n}\geq b}\sum
_{i,j=1}^{k}x_{i}
\bar{Q}_{i,j}x_{j}. \label{qp}%
\end{equation}
At first glance, it may appear that the minimiser in the right-hand
side is
$ ( x_{k+1},\ldots,x_{n} ) = ( b,\ldots,b )$, but this
is not
always true.\setcounter{footnote}{3}\footnote{For example, suppose $b=1$ and $\bar{Q}$ is the
$2\times2$
positive definite, symmetric matrix given by $\bar{Q}=\bigl(
{5 \atop -2}\enskip
{-2 \atop 1}
\bigr) $. Then $ ( 1,1 ) \bar{Q} ( 1,1 )
^{T}=2$, but
$ ( 1,1.1 ) \bar{Q} ( 1,1.1 ) ^{T}=1.8$.} The following
lemma, however, gives a simple condition on $\bar{Q}$ which ensures
that it is.

\begin{lemma}
\label{QP lemma}Let $b>0$ and $\mathbf{b}$ in $%
\mathbb{R}
^{n}$ denote the vector $ ( b,\ldots,b ) $. Suppose $ (
\bar{Q}_{ij} ) _{i,j\in \{ 1,2,\ldots,n \} }$ is a real
$n\times n$ positive definite matrix and assume $\bar{Q}$ has
nonnegative row
sums, that is,
%
\begin{equation}
\sum_{j=1}^{n}\bar{Q}_{ij}\geq0\qquad
\mbox{for all }i\in \{ 1,\ldots,n \}. \label{feas}%
\end{equation}
Then the infimum of the quadratic form $x^{T}\bar{Q}x$ over the subset
\[
\mathcal{C}= \bigl\{ ( x_{1},\ldots,x_{n} ) \in%
\mathbb{R} 
^{n}\dvtx x_{j}
\geq b\mathbf{,\forall}j=1,\ldots,n \bigr\}
\]
is attained at $x=\mathbf{b}$, and hence
\[
\inf_{x\in\mathcal{C}}x^{T}\bar{Q}x=\mathbf{b}^{T}
\bar{Q}\mathbf {b=}b^{2}%
\sum_{i,j=1}^{n}
\bar{Q}_{ij}.
\]
\end{lemma}

\begin{pf}
Without loss of generality, we may assume that $b=1$. We can then reformulate
the statement as describing the smallest value for the following constrained
quadratic programming problem:
\[
\min x^{T}\bar{Q}x\qquad \mbox{subject to } x\geq\mathbf{1},
\]
where $\mathbf{1}:= ( 1,\ldots,1 ) \in%
\mathbb{R}
^{n}$ and $x\geq\mathbf{1}$ means $x_{i}\geq\mathbf
{1}_{i}=1,\mathbf
{\forall
}i=1,\ldots,n$. The Lagrangian function of this quadratic programming problem
(see, e.g., \cite{boyd}, page 215) is given by
\[
L(x,\lambda)=x^{T}\bar{Q}x+\lambda^{T} ( -x+\mathbf{1} ).
\]
Solving for
\[
\nabla_{x}L(x,\lambda)=2\bar{Q}x-\lambda=0
\]
and using the strict convexity of the function we deduce that $x^{\ast}
=\frac{1}{2}Q^{-1}\lambda$ is the minimiser of $L$. Hence, the (Lagrangian)
dual function $g ( \lambda ):=\inf_{x}L ( x,\lambda
) $
is given by
\[
g ( \lambda ) =-\tfrac{1}{4}\lambda^{T}\bar {Q}^{-1}
\lambda +\lambda^{T}\mathbf{1}%
\]
and the dual problem consists of
\[
\max g ( \lambda )\qquad \mbox{subject to } \lambda\geq0.
\]
As $Q^{-1}$ is positive definite the function $g$ is strictly concave
and the
local maximum $\lambda^{\ast}=2\bar{Q}\mathbf{1}$ that is obtained
by solving
$\nabla_{\lambda}g ( \lambda ) =0$ with
%
\begin{equation}
\nabla_{\lambda}g ( \lambda ) =-\tfrac{1}{2}\bar{Q}^{-1}%
\lambda+\mathbf{1} \label{dual1}%
\end{equation}
is also the unique global maximum. In order to prove that $\lambda
^{\ast}$
solves the dual problem, we therefore need only check that it is
feasible for
the dual problem, that is, we must show that $\lambda^{\ast}\geq
\mathbf
{0}$. But
since the components of $\lambda^{\ast}$ are just twice the sum of the
respective rows of $\bar{Q}$, this feasibility condition follows at
once from
assumption \ref{feas}.
\end{pf}

When $Q$ arises as the covariance matrix of the increments of a Gaussian
process, we need to know when the Schur complement of some sub-block of $Q$
will satisfy condition (\ref{feas}). In the context of Gaussian
vectors, these
Schur complements have a convenient interpretation; they are the covariance
matrices which result from partially conditioning on some of the components.
It is this identification which motivates the positive conditional covariance
condition (Condition \ref{cond dom}).

In order to present the proof of the interpolation inequality as transparently
as possible, we first gather together some relevant technical comments. To
start with, suppose we have two sets of real numbers
\[
D= \{ t_{i}\dvtx i=0,1,\ldots,n \} \subset\tilde{D}= \{ \tilde
{t}_{i}\dvtx i=0,1,\ldots,\tilde{n} \} \subseteq [ 0,T ]
\]
ordered in such a way that $0\leq t_{0}<t_{1}<\cdots<t_{n}\leq T$, and
likewise for $\tilde{D}$. Suppose $s$ and $t$ be real numbers with
$s<t$ and
let $Z$ be a continuous Gaussian process. We need to consider how the variance
of the increment $Z_{s,t}$ changes when we condition on
\[
\mathcal{F}^{D}:=\sigma ( Z_{t_{i-1},t_{i}}\dvtx i=1,\ldots,n ),
\]
compared to conditioning the larger $\sigma$-algebra
\[
\mathcal{F}^{\tilde{D}}:=\sigma ( Z_{\tilde{t}_{i-1},\tilde
{t}_{i}%
}\dvtx i=1,\ldots,\tilde{n} ).
\]
To simplify the notation a little, we introduce
\[
\mathcal{G}=\sigma \bigl( Z_{\tilde{t}_{i-1},\tilde{t}_{i}}\dvtx \{ \tilde {t}_{i-1},
\tilde{t}_{i} \} \cap\tilde{D}\setminus D \neq \varnothing \bigr)
,
\]
so that
\[
\mathcal{F}^{\tilde{D}}=\mathcal{F}^{D}\vee\mathcal{G}.
\]
Because
%
\begin{equation}
( Z_{s,t},Z_{t_{0},t_{1}},\ldots,Z_{t_{\tilde{n}-1},t_{\tilde
{n}}%
} ) \in%
\mathbb{R} 
^{\tilde{n}+1}
\label{info}%
\end{equation}
is Gaussian, the joint distribution of $Z_{s,t}$ and the vector (\ref{info})
conditional on $\mathcal{F}^{D}$ (or indeed $\mathcal{F}^{\tilde{D}}$)
is once
again Gaussian, with a random mean but a deterministic covariance
matrix. A
simple calculation together with the \textit{law of total variance}
gives that
%
\begin{eqnarray*}
\Var \bigl( Z_{s,t}|\mathcal{F}^{D} \bigr) &=&E \bigl[ \Var
\bigl( Z_{s,t}|%
\mathcal{F}^{D}\vee\mathcal{G}
\bigr) \bigr] +\Var \bigl( E \bigl[ Z_{s,t}|%
\mathcal{F}^{D}\vee\mathcal{G} \bigr] \bigr)
\\
&\geq &E \bigl[ \Var \bigl( Z_{s,t}|\mathcal{F}^{D}\vee
\mathcal {G} \bigr) 
\bigr] =\Var \bigl( Z_{s,t}|
\mathcal{F}^{\tilde{D}} \bigr),
\end{eqnarray*}
which is the comparison we sought. We condense these observations
into the following lemma.

\begin{lemma}
\label{monotone}Let $ ( Z_{t} ) _{t\in [ 0,T ]
}$ be a
Gaussian process, and suppose that $D$ and $\tilde{D}$ are two
partitions of
$ [ 0,T ] $ with $D\subseteq\tilde{D}$. Then for any
$ [
s,t ] \subseteq [ 0,T ] $ we have
\[
\operatorname{Var} \bigl( Z_{s,t}|\mathcal{F}^{D} \bigr)
\geq \operatorname{Var} \bigl( Z_{s,t}|\mathcal{F}^{\tilde{D}}
\bigr).
\]
\end{lemma}

Our aim is to show how the optimisation problem of Lemma \ref{QP lemma}
can be
used to exhibit lower bounds on 2D Young integrals with respect to $R$. In
order to do this, we need to take a detour via two technical lemmas.
The first
is the following continuity result for the conditional covariance,
which we
need approximate when passing to a limit from a discrete partition. The
situation we will often have is two subintervals $ [ s,t ]
\subseteq [ 0,S ] $ of $ [ 0,T ] $, and a
sequence of
sets $ ( D_{n} ) _{n=1}^{\infty}$of the form
\[
D_{n}=D_{n}^{1}\cup D_{n}^{2}.
\]
$ ( D_{n}^{1} ) _{n=1}^{\infty}$ and $ (
D_{n}^{2} )
_{n=1}^{\infty}$ here will be nested sequences of partitions of $ [
0,s ] $ and $[t,S]$, respectively, with $\operatorname{mesh} (
D_{n}^{i} )
\rightarrow0$ as $n\rightarrow\infty$ for $i=1,2$. If
\[
\mathcal{F}^{D}:=\sigma \bigl( Z_{u,v}\dvtx \{ u,v \} \subseteq
D \bigr),
\]
then we can define a filtration $ ( \mathcal{G}_{n} )
_{n=1}^{\infty}$ by $\mathcal{G}_{n}:=\mathcal{F}^{D_{n}^{1}}\vee
\mathcal{F}^{D_{n}^{2}}$ and ask about the convergence of
\[
\operatorname{Cov} ( Z_{p,q}Z_{u,v}|\mathcal{G}_{n}
)
\]
as $n\rightarrow\infty$ for subintervals $ [ p,q ] $ and
$ [
u,v ] $ are subintervals of $ [ 0,S ] $. The following lemma
records the relevant continuity statement.

\begin{lemma}
\label{continuity} For any $p,q,u,v$ such that $ [ p,q ] $ and
$ [ u,v ] $ are subintervals of $ [ 0,S ]
\subseteq
[ 0,T ] $ we have
\[
\operatorname{Cov} ( Z_{p,q}Z_{u,v}|\mathcal{G}_{n}
) \rightarrow\operatorname{Cov} ( Z_{p,q}Z_{u,v}|\mathcal
{F}_{0,s}%
\vee\mathcal{F}_{t,S} )
\]
as $n\rightarrow\infty$.
\end{lemma}

\begin{pf}
The martingale convergence theorem gives
\[
\operatorname{Cov} ( Z_{p,q}Z_{u,v}|\mathcal{G}_{n}
) \rightarrow\operatorname{Cov} \Biggl( Z_{p,q}Z_{u,v}\bigg|\bigvee
_{n=1}^{\infty
}\mathcal{G}_{n} \Biggr),\qquad \mbox{a.s. and
in }L^{p}\mbox{ for all }p\geq1.
\]
The continuity of $Z$ and the fact that $\operatorname{mesh} ( D_{n} )
\rightarrow0$ easily implies that, modulo null sets, one has $\bigvee
_{n=1}^{\infty}\mathcal{G}_{n}=\mathcal{F}_{0,s}\vee\mathcal{F}_{t,T}$.
\end{pf}

We now introduce another condition on $Z$, which we will later discard. This
condition is virtually the same as Condition~\ref{cond dom}, the only
difference being that we insist on the strict positivity of the
conditional variance.

\begin{condition}
\label{prime} Let $ ( Z_{t} ) _{t\in [ 0,T ] }$
be a
real-valued continuous Gaussian process. We will assume that for every
$ [ u,v ] \subseteq [ s,t ] \subseteq [
0,S ]
\subseteq [ 0,T ] $ we have
%
\begin{equation}
\operatorname{Cov} ( Z_{s,t},Z_{u,v}|
\mathcal{F}_{0,s}\vee \mathcal{F}_{t,S} ) >0.
\end{equation}
\end{condition}

The second technical lemma we need will apply whenever we work with a Gaussian
process that satisfies Condition \ref{prime}. It delivers a nested
sequence of
partitions, with mesh tending to zero, and such that the discretisation
of $Z$
along each partition will satisfy the dual feasibility condition
[i.e., (\ref{feas}) in Lemma \ref{QP lemma}].

\begin{lemma}
\label{technical}Let $ ( Z_{t} ) _{t\in [ 0,T ]
}$ be a
continuous Gaussian process that satisfies Condition \ref{prime}. Then for
every $0\leq s<t\leq S\leq T$ there exists a nested sequence of partitions
\[
( D_{m} ) _{m=1}^{\infty}= \bigl( \bigl\{
t_{i}^{m}%
\dvtx i=0,1,\ldots,n_{m} \bigr\}
\bigr) _{m=1}^{\infty}%
\]
of $ [ 0,S ] $ with the following properties:

\begin{longlist}[(1)]
\item[(1)] The mesh of $D_{m}$ converges to $0$ as $m\rightarrow\infty$.

\item[(2)] One has $ \{ s,t \} \subseteq D_{m}$ for all $m$.

\item[(3)] If $Z_{1}^{m}$ and $Z_{2}^{m}$ are the jointly Gaussian vectors,
%
\begin{eqnarray*}
Z_{1}^{m} &= &\bigl( Z_{t_{i}^{m},t_{i+1}^{m}}\dvtx t_{i}^{m}
\in D_{m}\cap \bigl(  [0,s ) \cup [ t,S)\bigr) \bigr),
\\
Z_{2}^{m} &= &\bigl( Z_{t_{i}^{m},t_{i+1}^{m}}\dvtx t_{i}^{m}
\in D_{m}\cap [ s,t) \bigr),
\end{eqnarray*}
with respective covariance matrices $Q_{11}^{m}$ and $Q_{22}^{m}$,
then the Gaussian vector $ ( Z_{1}^{m},Z_{2}^{m} ) $ has a
covariance matrix of the form
\[
Q^{m}=\pmatrix{
Q_{11}^{m} & Q_{12}^{m}
\vspace*{2pt}\cr
\bigl( Q_{12}^{m} \bigr) ^{T} &
Q_{22}^{m}},
\]
and the Schur complement of $Q_{11}^{m}$ in $Q^{m}$ has nonnegative
row sums.
\end{longlist}
\end{lemma}

\begin{pf}
See the \hyperref[app]{Appendix}.
\end{pf}

The next result shows how we can bound from below the 2D Young
integral of a
H\"{o}lder-continuous $f$ against $R$. The lower bound thus obtained is
expressed in terms of the minimum of $f$, and the conditional variance
of the
Gaussian process.

\begin{proposition}
\label{comparison}Suppose that $R\dvtx [ 0,T ] ^{2}\rightarrow%
\mathbb{R}
$ is the covariance function of some continuous Gaussian process $ (
Z_{t} ) _{t\in [ 0,T ] }$. Suppose that $R$ has
finite 2D
$\rho
$-variation for some $\rho$ in $[1,2)$ and that $Z$ is nondegenerate
and has
a nonnegative conditional covariance (i.e., satisfies Condition \ref
{cond dom}%
). Let $\gamma\in ( 0,1 ) $ be such that $1/\rho+\gamma
>1$ and
assume $f\in C^{\gamma} (  [ 0,T ],%
\mathbb{R}
) $. Then for every $ [ s,t ] \subseteq [
0,T ] $
we have the following lower bound on the 2D-Young integral of $f$
against $R$:
\[
\int_{ [ 0,T ] ^{2}}f_{u}f_{v}\,dR ( u,v ) \geq
\Bigl( \inf_{u\in [ s,t ] }\bigl\llvert f ( u ) \bigr\rrvert
^{2} \Bigr) \operatorname{Var} ( Z_{s,t}|
\mathcal{F}_{0,s}%
\vee\mathcal{F}_{t,T} ).
\]
\end{proposition}

\begin{remark}
We emphasise again that $\mathcal{F}_{a,b}$ is the $\sigma$-algebra generated
by the increments of the form $Z_{u,v}$ for $u,v\in [ a,b ] $.
\end{remark}

\begin{pf*}{Proof of Proposition \ref{comparison}}
Fix $ [ s,t ] \subseteq [ 0,T ] $, and take
$b:=\break \inf_{u\in [ s,t ] }\llvert  f ( u )
\rrvert $.

\textit{Step} 1: We first note that there is no
loss of
generality in assuming the stronger Condition \ref{prime} instead of Condition
\ref{cond dom}. To see this, let $ ( B_{t} ) _{t\in [
0,T ] }$ be a Brownian motion, which is independent of $ (
Z_{t} ) _{t\in [ 0,T ] }$, and for every $\varepsilon
>0$ define
the perturbed process
\[
Z_{t}^{\varepsilon}:=Z_{t}+\varepsilon B_{t}.
\]
It is easy to check that $Z^{\varepsilon}$ satisfies the conditions in the
statement. Let $\mathcal{F}_{p,q}^{\varepsilon}$ be the $\sigma$-algebra
generated by the increments $Z_{u,v}^{\varepsilon}$ between times $p$ and $q$
[note that $\mathcal{F}_{p,q}^{\varepsilon}$ actually equals $\mathcal
{F}%
_{p,q}\vee\sigma ( B_{l,m}\dvtx u\leq l<m\leq q ) $], and note
that we
have
\[
\operatorname{Cov} \bigl( Z_{s,t}^{\varepsilon},Z_{u,v}^{\varepsilon}%
|\mathcal{F}_{0,s}^{\varepsilon}\vee\mathcal{F}_{t,T}^{\varepsilon
}
\bigr) =\operatorname{Cov} ( Z_{s,t},Z_{u,v}|
\mathcal{F}_{0,s}\vee \mathcal{F}_{t,T} ) +
\varepsilon^{2} ( u-v ) >0
\]
for every $0\leq s<u<v\leq t\leq T$. It follows that $Z^{\varepsilon}$ satisfies
Condition~\ref{prime}. Let $R^{\varepsilon}$ denote the covariance
function of
$Z^{\varepsilon}$. If we could prove the result with the additional
hypothesis of
Condition~\ref{prime}, then it would follow that
%
\begin{eqnarray}\label{x prime}
\int_{ [ 0,T ] ^{2}}f_{u}f_{v}\,dR^{\varepsilon}
( u,v ) &\geq& b^{2}\Var \bigl( Z_{s,t}^{\varepsilon}|
\mathcal{F}_{0,s}^{\varepsilon
}\vee \mathcal{F}_{t,T}^{\varepsilon}
\bigr)
\nonumber
\\[-8pt]
\\[-8pt]
\nonumber
&=& b^{2}\Var ( Z_{s,t}|\mathcal{F}_{0,s}\vee
\mathcal {F}_{t,T} ) +b^{2}\varepsilon^{2} ( t-s ).
\end{eqnarray}
Because
\[
\int_{ [ 0,T ] ^{2}}f_{u}f_{v}\,dR^{\varepsilon}
( u,v ) =\int_{ [ 0,T ] ^{2}}f_{u}f_{v}\,dR ( u,v
) +\varepsilon ^{2}\llvert f\rrvert _{L^{2} [ 0,T ] }^{2},
\]
the result for $Z$ will then follow from (\ref{x prime}) by letting
$\varepsilon$
tend to zero.

\textit{Step} 2: We now prove the result under the
additional assumption of Condition~\ref{prime}. By considering $-f$ if
necessary, we may assume that $f$ is bounded from below by $b$ on
$ [
s,t ] $. Since we now assume Condition \ref{prime} we can use Lemma~\ref{technical} to obtain a nested sequence of partitions $ (
D_{r} ) _{r=1}^{\infty}$ such that $ \{ s,t \}
\subset D_{r}$
for all~$r$, $\operatorname{mesh} ( D_{r} ) \rightarrow0$ as $r\rightarrow
\infty$,
and such that the dual feasibility condition (property 3 in the Lemma~\ref{technical}) holds. Suppose $D= \{ t_{i}\dvtx i=0,1,\ldots
,n \}
$ is
any partition of $ [ 0,T ] $ in this sequence (i.e., $D=D_{r}$ for
some $r$). Then for some $l<m\in \{ 0,1,\ldots,n-1 \} $ we have
$t_{l}=s$ and $t_{m}=t$. Denote by $f ( D ) $ the column vector
\[
f ( D ) = \bigl( f ( t_{0} ),\ldots,f ( t_{n-1} ) \bigr)
^{T}\in%
\mathbb{R} 
^{n},
\]
and $Q= ( Q_{i,j} ) _{1\leq i,j<n}$ the symmetric $n\times n$
matrix with entries
\[
Q_{ij}=R\pmatrix{
t_{i-1},t_{i}
\vspace*{2pt}\cr
t_{j-1},t_{j}} =E [ Z_{t_{i-1},t_{i}}Z_{t_{j-1},t_{j}} ].
\]
From the nondegeneracy of $Z$, it follows that $Q$ is positive
definite. The
Riemann sum approximation to the 2D integral of $f$ against $R$ along the
partition $D$ can be written as%
%
\begin{eqnarray}\label{Riemann}%
\sum_{i=1}^{n}\sum
_{j=1}^{n}f_{t_{i-1}}f_{t_{j-1}}R\pmatrix{
t_{i-1},t_{i}
\vspace*{2pt}\cr
t_{j-1},t_{j}} &=&\sum_{i=1}^{n}\sum
_{j=1}^{n}f_{t_{i-1}}f_{t_{j-1}}Q_{i,j}
\nonumber
\\[-8pt]
\\[-8pt]
\nonumber
&=&f
( D ) ^{T}Qf ( D ).
\end{eqnarray}
If necessary, we can ensure that last $m-l$ components of $f (
D ) $ are bounded below by $b$. To see this, we simply permute its
coordinates using any bijective map $\tau\dvtx  \{ 1,\ldots,n \}
\rightarrow \{ 1,\ldots,n \} $ which has the property that
\[
\tau ( l+j ) =n-m+l+j\qquad\mbox{for }j=0,1,\ldots,m-l.
\]
Fix one such map $\tau$, and let $f_{\tau} ( D ) $ denote the
vector resulting from applying $\tau$ to the coordinates of $f (
D ) $. Similarly, let $Q_{\tau}= ( Q_{i,j}^{\tau} )
_{1\leq
i,j<n}$ be the $n\times n$ matrix
\[
Q_{ij}^{\tau}=Q_{\tau ( i ) \tau ( j ) },
\]
and note that $Q^{\tau}$ is the covariance matrix of the Gaussian vector
\[
Z= ( Z_{t_{\tau ( 1 ) -1,}t_{\tau ( 1 ) }}%
,\ldots,Z_{t_{\tau ( n ) -1},t_{\tau ( n )
}} ).
\]

A simple calculation shows that
\[
f ( D ) ^{T}Qf ( D ) =f_{\tau} ( D ) ^{T}Q_{\tau}f_{\tau}
( D ).
\]
We can apply Lemma \ref{QP lemma} because condition (\ref{feas}) is guaranteed
to hold by the properties of the sequence $ ( D_{r} )
_{r=1}^{\infty}$. We deduce that
%
\begin{equation}
f ( D ) ^{T}Qf ( D ) =f_{\tau} ( D ) ^{T}Q_{\tau}f_{\tau}
( D ) \geq b^{2}\sum_{i,j=1}^{m-l}S_{ij},
\label{permute}%
\end{equation}
where $S$ is the $ ( m-l ) \times ( m-l ) $ matrix
obtained by taking the Schur complement of the leading principal $ (
n-m+l ) \times ( n-m+l ) $ minor of $\tilde{Q}$. As already
mentioned, the distribution of a Gaussian vector conditional on some of its
components remains Gaussian and the conditional covariance is described
by a
suitable Schur complement. In this case, this means we have that
%
\begin{eqnarray}\label{Schur cov}%
&&S=\operatorname{Cov} \bigl[ ( Z_{t_{l},t_{l+1}},\ldots,Z_{t_{m-1}
,t_{m}} )
|Z_{t_{j-1},t_{j}},
\nonumber
\\[-8pt]
\\[-8pt]
\nonumber
&&\hspace*{44pt} j\in \{ 1,\ldots,l \} \cup \{ m+1,\ldots,n \} \bigr]
.
\end{eqnarray}
If we define
\[
\mathcal{F}^{D}:=\sigma \bigl( Z_{t_{j-1},t_{j}}\dvtx  j\in \{ 1,
\ldots,l \} \cup \{ m+1,\ldots,n \} \bigr),
\]
to be the $\sigma$-algebra generated by the increments of $Z$ in
$D\setminus [ s,t ] $, then using (\ref{Schur cov}) we
arrive at
%
\begin{eqnarray}\label{cond var increment}
\sum_{i,j=1}^{m-l}S_{ij} &=&\sum
_{i,j=1}^{m-l-1}E \bigl[ ( Z_{t_{l+i-1},t_{l+i}} ) (
Z_{t_{l+j-1},t_{l+j}} ) |\mathcal{F}%
^{D} \bigr]
\nonumber
\\
& &{}-\sum_{i,j=1}^{m-l-1}E \bigl[ (
Z_{t_{l+i-1},t_{l+i}} ) |\mathcal{%
F}^{D} \bigr] E \bigl[ (
Z_{t_{l+j-1},t_{l+j}} ) |\mathcal {F}^{D}%
\bigr]
\\
&=&E \bigl[ ( Z_{s,t} ) ^{2}|\mathcal{F}^{D}
\bigr] -E \bigl[ Z_{s,t}|\mathcal{F}^{D} \bigr]
^{2} =\Var \bigl( Z_{s,t}|\mathcal{F}^{D} \bigr).
\nonumber
\end{eqnarray}

To complete the proof, we note that $\mathcal{F}^{D}\subseteq\mathcal{F}
_{0,s}\vee\mathcal{F}_{t,T}$, and exploit the monotonicity of the conditional
variance described by Lemma \ref{monotone} to give%
%
\begin{equation}
\operatorname{Var} \bigl( Z_{s,t}|\mathcal{F}^{D} \bigr)
\geq \operatorname{Var} ( Z_{s,t}|\mathcal{F}_{0,s}\vee
\mathcal {F}%
_{t,T} ). \label{projection}%
\end{equation}
Then by combining (\ref{projection}), (\ref{cond var increment}) and
(\ref{permute}) in (\ref{Riemann}), we obtain
\[
\sum_{i=1}^{n}\sum
_{j=1}^{n}f_{t_{i-1}}f_{t_{j-1}}Q_{i,j}
\geq b^{2}%
\operatorname{Var} ( Z_{s,t}|
\mathcal{F}_{0,s}\vee\mathcal {F}%
_{t,T} ).
\]
Because this inequality holds for any $D\in ( D_{r} )
_{r=1}^{\infty}$, we can apply it for $D=D_{r}$ and let $r\rightarrow
\infty$,
which yields
%
\[
\int_{ [ 0,T ] ^{2}}f_{u}f_{v}\,dR ( u,v ) \geq
b^{2}\Var ( Z_{s,t}|\mathcal{F}_{0,s}\vee\mathcal
{F}_{t,T} ),
\]
whereupon the proof is complete.
\end{pf*}

We now deliver on a promise we made in Section~\ref{main thm} by
proving that
the diagonal dominance of the increments implies the positivity of the
conditional covariance.

\begin{corollary}
Let $ ( Z_{t} ) _{t\in [ 0,T ] }$ be a real-valued
continuous Gaussian process. If $Z$ satisfies Condition
\ref{diagonal dominance} then it also satisfies Condition~\ref{cond dom}.
\end{corollary}

\begin{pf}
Fix $s<t$ in $ [ 0,T ] $, let $(D_{n})_{n=1}^{\infty}$ be a
sequence of partitions having the properties described in the statement of
Lemma \ref{continuity} and suppose $ [ u,v ] \subseteq$
$ [
s,t ] $. From the conclusion of Lemma \ref{continuity}, we have that
%
\begin{equation}
\operatorname{Cov} ( Z_{s,t}Z_{u,v}|\mathcal{G}_{n}
) \rightarrow\operatorname{Cov} ( Z_{s,t}Z_{u,v}|\mathcal
{F}_{0,s}%
\vee\mathcal{F}_{t,T} )
\label{ap}%
\end{equation}
as $n\rightarrow\infty$. Let $Z_{n}$ be the Gaussian vector whose components
consist of the increments of $Z$ over all the consecutive points in the
partition $D_{n}\cup \{ s,u,v,t \} $. Let $Q$ denote the covariance
matrix of $Z_{n}$. The left-hand side of (\ref{ap}) is the sum of all the
entries in some row of a particular Schur complement of $Q$. $Z$ is
assumed to
have diagonally dominant increments. Any such Schur complement of $Q$ will
therefore be diagonally dominant, since diagonal dominance is preserved under
Schur-complementation (see \cite{zhang}). As diagonally dominant
matrices have
nonnegative row sums, it follows that $\operatorname{Cov} (
Z_{s,t}Z_{u,v}|\mathcal{G}_{n} ) $ is nonnegative, and hence
the limit
in (\ref{ap}) is also.
\end{pf}

We are now in a position to generalise the $L^{2}$-interpolation inequality
(\ref{itl}) stated earlier.

\begin{theorem}[(Interpolation)]\label{interpolation}Let $ ( Z_{t} ) _{t\in
[
0,T ] }$ be a continuous Gaussian process with covariance function
$R\dvtx [ 0,T ] ^{2}\rightarrow%
\mathbb{R}
$. Suppose $R$ has finite two-dimensional $\rho$-variation for some
$\rho$ in
$[1,2)$. Assume that $Z$ is nondegenerate in the sense of Definition
\ref{definition nondegeneracy}, and has positive conditional
covariance (i.e., satisfies Condition~\ref{cond dom}). Suppose $f\in C (  [
0,T ],%
\mathbb{R}
) $ with $\gamma+1/\rho>1$. Then for every $0<S\leq T$ at least
one of
the following inequalities is always true:
%
\begin{equation}
\llVert f\rrVert _{\infty; [ 0,S ] }\leq2E \bigl[ Z_{S}%
^{2} \bigr] ^{-1/2} \biggl( \int_{ [ 0,S ]
^{2}}f_{s}f_{t}\,dR
( s,t ) \biggr) ^{1/2}, \label{L2}%
\end{equation}
or, for some interval $ [ s,t ] \subseteq [ 0,S
] $ of
length at least
\[
\biggl( \frac{\llVert  f\rrVert  _{\infty; [ 0,S ]
}}{2\llVert
f\rrVert  _{\gamma; [ 0,S ] }} \biggr) ^{1/\gamma},
\]
we have%
%
\begin{equation}
\frac{1}{4}\llVert f\rrVert _{\infty; [ 0,S ] }^{2}%
\operatorname{Var} ( Z_{s,t}|\mathcal{F}_{0,s}\vee\mathcal
{F}%
_{t,S} ) \leq\int_{ [ 0,S ] ^{2}}f_{v}f_{v^{\prime
}}\,dR
\bigl( v,v^{\prime} \bigr). \label{inf}%
\end{equation}
\end{theorem}

\begin{pf}
We take $S=T$, the generalisation to $0<S<T$ needing only minor
changes. $f$
is continuous and, therefore, achieves its maximum in $ [
0,T
] $.
Thus, by considering $-f$ if necessary, we can find $t\in [
0,T ] $
such that
\[
f ( t ) =\llVert f\rrVert _{\infty;  [ 0,T ] }.
\]
There are two possibilities which together are exhaustive. In the first case,
$f$ never takes any value less than half its maximum, that is,
\[
\inf_{u\in [ 0,T ] }f ( u ) \geq\tfrac
{1}{2}\llVert f\rrVert
_{\infty; [ 0,T ] }.
\]
Hence, we can apply Proposition \ref{comparison} to deduce (\ref
{L2}). In
the second case, there exists $u\in [ 0,T ] $ such that
$f (
u ) =2^{-1}\llVert  f\rrVert  _{\infty; [ 0,T ] }$. Then,
assuming that $u<t$ (the argument for $u>t$ leads to the same outcome),
we can
define
\[
s=\sup \bigl\{ v<t\dvtx f ( v ) \leq\tfrac{1}{2}\llVert f\rrVert
_{\infty; [ 0,T ] } \bigr\}.
\]
By definition $f$ is then bounded below by $\llVert  f\rrVert
_{\infty
; [ 0,T ] }/2$ on $ [ s,t ] $. The H\"{o}lder continuity
of $f$ gives a lower bound on the length of this interval in an elementary
way
\[
\tfrac{1}{2}\llVert f\rrVert _{\infty; [ 0,T ] }=\bigl\llvert f ( t ) -f ( s )
\bigr\rrvert \leq\llVert f\rrVert _{\gamma; [ 0,T ] }\llvert t-s\rrvert
^{\gamma},
\]
which yields
\[
\llvert t-s\rrvert \geq \biggl( \frac{\llVert  f\rrVert
_{\infty
; [ 0,T ] }}{2\llVert  f\rrVert  _{\gamma; [
0,T ] }%
} \biggr) ^{1/\gamma}.
\]
Another application of Proposition \ref{comparison} then gives (\ref{inf}).
\end{pf}

\begin{corollary}
\label{interpolation2}Assume Condition \ref{nondeterm} so that the
$\rho
$-variation of $R$ is H\"{o}lder-controlled, and for some $c>0 $ and some
$\alpha\in ( 0,1 ) $ we have the lower bound on the conditional
variance
\[
\operatorname{Var} ( Z_{s,t}|\mathcal{F}_{0,s}\vee
\mathcal {F}%
_{t,T} ) \geq c ( t-s ) ^{\alpha}.
\]
Theorem \ref{interpolation} then allows us to bound $\llVert  f\rrVert
_{\infty; [ 0,T ] }$ above by the maximum of
\[
2E \bigl[ Z_{T}^{2} \bigr] ^{-1/2} \biggl( \int
_{ [ 0,T ]
^{2}%
}f_{s}f_{t}\,dR ( s,t ) \biggr)
^{1/2}%
\]
and
\[
\frac{2}{\sqrt{c}} \biggl( \int_{ [ 0,T ]
^{2}}f_{s}f_{t}\,dR
( s,t ) \biggr) ^{\gamma/ ( 2\gamma+\alpha ) } \llVert f\rrVert _{\gamma; [ 0,T ] }^{\alpha/ ( 2\gamma
+\alpha ) }.
\]
\end{corollary}

\begin{pf}
This is immediate from Theorem \ref{interpolation}.
\end{pf}

In particular, if $Z$ is a Brownian motion we have
$\operatorname{Var} ( Z_{s,t}|\mathcal{F}_{0,s}\vee\mathcal
{F}%
_{t,T} ) = ( t-s ) $, hence Corollary \ref{interpolation2}
shows that
\[
\llVert f\rrVert _{\infty; [ 0,T ] }\leq2\max \bigl( T^{-1/2}\llvert f
\rrvert _{L^{2} [ 0,T ] },\llvert f\rrvert _{L^{2} [ 0,T ] }^{2\gamma/ ( 2\gamma
+1 )
}\llVert
f\rrVert _{\gamma; [ 0,T ] }^{1/ (
2\gamma+1 ) } \bigr),
\]
which is exactly (\ref{itl}). We have therefore achieved out goal of
generalising this inequality.

\begin{remark}
Another application where we anticipate estimates of this kind being
useful is
when estimating short-time density asymptotics (see, e.g., the recent works
\cite{BauO,I}). Here, frequent use is made of the asymptotic behaviour
of the
Malliavin covariance matrix as $t\downarrow0$.
\end{remark}

\section{Malliavin differentiability of the flow}
\label{differentiability section}

\subsection{High-order directional derivatives}

Let $\mathbf{x}$ be in $WG\Omega_{p} (
\mathbb{R}
^{d} ) $ and suppose that the vector fields $V= ( V_{1}%
,\ldots,V_{d} ) $ and $V_{0}$ are smooth and bounded. For $t\in
[
0,T ] $ we let $U_{t\leftarrow0}^{\mathbf{x}} ( \cdot
) $
denote the map defined by
\[
U_{t\leftarrow0}^{\mathbf{x}} ( \cdot ) \dvtx y_{0}\mapsto
y_{t},
\]
where $y$ is the solution to the RDE%
%
\begin{equation}
dy_{t}=V ( y_{t} ) \,d\mathbf{x}_{t}+V_{0}
( y_{t} ) \,dt,\qquad y ( 0 ) =y_{0}. \label{flow}%
\end{equation}
It is well known (see \cite{FV}) that the flow [i.e., the map
$y_{0}\mapsto
U_{t\leftarrow0}^{\mathbf{x}} ( y_{0} ) $] is
differentiable; its
derivative (or Jacobian) is the linear map
\[
J_{t\leftarrow0}^{\mathbf{x}} ( y_{0} ) ( \cdot ) \equiv
\frac{d}{d\varepsilon}U_{t\leftarrow0}^{\mathbf
{x}} ( y_{0}+\varepsilon
\cdot ) \bigg\rrvert _{\varepsilon=0}\in L \bigl( 
\mathbb{R}
^{e},%
\mathbb{R}
^{e} \bigr).
\]

If we let $\Phi_{t\leftarrow0}^{\mathbf{x}} ( y_{0} )$,
denote the
pair
\[
\Phi_{t\leftarrow0}^{\mathbf{x}} ( y_{0} ) = \bigl(
U_{t\leftarrow
0}^{\mathbf{x}} ( y_{0} ),J_{t\leftarrow0}^{\mathbf
{x}}
( y_{0} ) \bigr) \in%
\mathbb{R}
^{e}\oplus L \bigl( 
\mathbb{R}
^{e},%
\mathbb{R}
^{e} \bigr),
\]
and if $W= ( W_{1},\ldots,W_{d} ) $ is the collection
vector fields
given by
\[
W_{i} ( y,J ) = \bigl( V_{i} ( y ),\nabla
V_{i} ( y ) \cdot J \bigr),\qquad i=1,\ldots,d
\]
and
\[
W_{0} ( y,J ) = \bigl( V_{0} ( y ),\nabla
V_{0} ( y ) \cdot J \bigr)
\]
then $\Phi_{t\leftarrow0}^{\mathbf{x}} ( y_{0} ) $ is the
solution\footnote{A little care is needed because the vector fields have
linear growth (and hence are not Lip-$\gamma$). But one can exploit the
``triangular'' dependence structure in the vector fields to rule out the
possibility of explosion. See \cite{FV} for details.} to the RDE
\[
d\Phi_{t\leftarrow0}^{\mathbf{x}}=W \bigl( \Phi_{t\leftarrow
0}^{\mathbf
{x}%
}
\bigr) \,d\mathbf{x}_{t}+W_{0} \bigl( \Phi_{t\leftarrow0}^{\mathbf
{x}%
}
\bigr) \,dt,\Phi_{t\leftarrow0}^{\mathbf{x}}|_{t=0}= (
y_{0},I ).
\]
In fact, the Jacobian is invertible as a linear map and the inverse,
which we
will denote $J_{0\leftarrow t}^{\mathbf{x}} ( y_{0} ) $, is
also a
solution to an RDE [again jointly with the base flow $U_{t\leftarrow
0}^{\mathbf{x}} ( y_{0} ) $]. We also recall the relation
\[
J_{t\leftarrow s}^{\mathbf{x}} ( y ):= \frac
{d}{d\varepsilon
}U_{t\leftarrow s}^{\mathbf{x}}
( y+\varepsilon\cdot ) \bigg\rrvert _{\varepsilon=0}=J_{t\leftarrow0}^{\mathbf{x}} (
y ) \cdot J_{0\leftarrow s}^{\mathbf{x}} ( y ).
\]

\begin{notation}
In what follows, we will let
%
\begin{equation}
\qquad M_{\cdot\leftarrow0}^{\mathbf{x}} ( y_{0} ) \equiv \bigl(
U_{t\leftarrow0}^{\mathbf{x}} ( y_{0} ),J_{t\leftarrow
0}^{\mathbf{x}}
( y_{0} ),J_{0\leftarrow t}^{\mathbf
{x}} ( y_{0} )
\bigr) \in%
\mathbb{R} 
^{e}\oplus%
\mathbb{R} 
^{e\times e}\oplus%
\mathbb{R} 
^{e\times e}. \label{M}%
\end{equation}
\end{notation}

For any path $h$ in $C^{q\mbox{-}\mathrm{var}} (  [ 0,T ]
,%
\mathbb{R}
^{d} ) $ with $1/q+1/p>1$, we can canonically define the
translated rough
path $T_{h}\mathbf{x}$ (see \cite{FV}). Hence, we have the directional
derivative
\[
D_{h}U_{t\leftarrow0}^{\mathbf{x}} ( y_{0} ) \equiv
 \frac
{d}{d\varepsilon}U_{t\leftarrow0}^{T_{\varepsilon h}\mathbf{x}} ( y_{0} )
\bigg\rrvert _{\varepsilon=0}.
\]
It is not difficult to show that
\[
D_{h}U_{t\leftarrow0}^{\mathbf{x}} ( y_{0} ) =\sum
_{i=1}^{d}%
\int
_{0}^{t}J_{t\leftarrow s}^{\mathbf{x}} (
y_{0} ) V_{i} \bigl( U_{s\leftarrow0}^{\mathbf{x}} (
y_{0} ) \bigr) \,dh_{s}^{i},
\]
which implies by Young's inequality that%
%
\begin{equation}
\bigl\llvert D_{h}U_{t\leftarrow0}^{\mathbf{x}} (
y_{0} ) \bigr\rrvert \leq C\bigl\llVert M_{\cdot\leftarrow0}^{\mathbf{x}}
( y_{0} ) \bigr\rrVert _{p\mbox{-}\mathrm{var}; [ 0,t ] }\llvert h\rrvert
_{q\mbox{-}\mathrm{var}; [ 0,t ] }. \label{linear bound}%
\end{equation}
In this section, we will be interested in the form of the higher order
directional derivatives
\[
D_{h_{1}}\cdots D_{h_{n}}U_{t\leftarrow0}^{\mathbf{x}} (
y_{0} ):= \frac{\partial^{n}}{\partial\varepsilon_{1},\ldots,\partial
\varepsilon
_{n}%
}U_{t\leftarrow0}^{T_{\varepsilon_{n}h_{n}}\cdots T_{\varepsilon_{1}h_{1}}%
\mathbf{x}}
( y_{0} ) \bigg\rrvert _{\varepsilon_{1}=\cdots
=\varepsilon_{n}=0}.
\]
Our aim will be to obtain bounds of the form (\ref{linear bound}); to
do this
in a systematic way is a challenging exercise. We rely on the treatment
presented in \cite{H3}. For the reader's convenience when comparing
the two
accounts, we note that \cite{H3} uses the notation
\[
\bigl( D_{s}U_{t\leftarrow0}^{\mathbf{x}} ( y_{0} )
\bigr) _{s\in [ 0,T ] }= \bigl( D_{s}^{1}U_{t\leftarrow
0}^{\mathbf{x}%
}
( y_{0} ),\ldots,D_{s}^{d}U_{t\leftarrow0}^{\mathbf
{x}}
( y_{0} ) \bigr) _{s\in [ 0,T ] }\in%
\mathbb{R} 
^{d}%
\]
to identify the derivative. The relationship between
$D_{s}U_{t\leftarrow
0}^{\mathbf{x}} ( y_{0} ) $ and\break $D_{h}U_{t\leftarrow
0}^{\mathbf
{x}%
} ( y_{0} ) $ is simply that
\[
D_{h}U_{t\leftarrow0}^{\mathbf{x}} ( y_{0} ) =\sum
_{i=1}^{d}%
\int
_{0}^{t}D_{s}^{i}U_{t\leftarrow0}^{\mathbf{x}}
( y_{0} ) \,dh_{s}^{i}.
\]
Note, in particular, $D_{s}U_{t\leftarrow0}^{\mathbf{x}} (
y_{0} )
=0$ if $t<s$.

\begin{proposition}
\label{induction}Assume $\mathbf{x}$ is in $WG\Omega_{p} (
\mathbb{R}
^{d} ) $ and let $V= ( V_{1},\ldots,V_{d} ) $ be a collection
of smooth and bounded vector fields. Denote the solution flow to the RDE
(\ref{flow}) by
\[
U_{t\leftarrow0}^{\mathbf{x}} ( y_{0} ) = \bigl(
U_{t\leftarrow
0}^{\mathbf{x}} ( y_{0} ) _{1},
\ldots,U_{t\leftarrow
0}^{\mathbf
{x}%
} ( y_{0} ) _{e}
\bigr) \in%
\mathbb{R} 
^{e}.%
\]
Suppose $q\geq1$ and $n\in%
\mathbb{N}
$ and let $ \{ h_{1},\ldots,h_{n} \} $ be any subset of
$C^{q\mbox{-}\mathrm{var}} (  [ 0,T ],%
\mathbb{R}
^{d} ) $. Then the directional derivative $D_{h_{1}}\cdots D_{h_{n}
}U_{t\leftarrow0}^{\mathbf{x}} ( y_{0} ) $ exists for any
$t\in [ 0,T ] $. Moreover, there exists a collection of finite
indexing sets
\[
\bigl\{ \mathbf{K}_{ ( i_{1},\ldots,i_{n} ) }\dvtx ( i_{1}%
,
\ldots,i_{n} ) \in \{ 1,\ldots,d \} ^{n} \bigr\},
\]
such that for every $j\in \{ 1,\ldots,e \} $ we have the identity
%
\begin{eqnarray}\label{high order rep}
&&D_{h_{1}}\cdots D_{h_{n}}U_{t\leftarrow0}^{\mathbf{x}} (
y_{0} ) _{j}\nonumber\\
&&\qquad  =\sum
_{i_{1},\ldots,i_{n}=1}^{d}\sum_{k\in\mathbf{K}_{ (
i_{1},\ldots,i_{n} ) }}\int
_{0<t_{1}<\cdots
<t_{n}<t}f_{1}^{k} ( t_{1} )
\cdots\\
&&\hspace*{177pt}{} f_{n}^{k} ( t_{n} ) f_{n+1}^{k}
( t ) \,dh_{t_{1}}^{i_{1}}\cdots dh_{t_{n}}^{i_{n}}\nonumber
\end{eqnarray}
for some functions $f_{l}^{k}$ which are in $C^{p\mbox{-}\mathrm{var}} (
[ 0,T ],%
\mathbb{R}
) $ for every $l$ and $k$, that is,
\[
\bigcup_{ ( i_{1},\ldots,i_{n} ) \in \{ 1,\ldots
,d \}
^{n}%
}\bigcup_{k\in\mathbf{K}_{ ( i_{1},\ldots,i_{n} ) }} \bigl\{ f_{l}%
^{k}\dvtx l=1,\ldots,n+1 \bigr\} \subset C^{p\mbox{-}\mathrm{var}} \bigl( [ 0,T ]
,%
\mathbb{R} 
\bigr).
\]
Furthermore, there exists a constant $C$, which depends only on $n$ and $T$
such that
%
\begin{equation}
\bigl\llvert f_{l}^{k}\bigr\rrvert _{p\mbox{-}\mathrm{var}; [
0,T ]
}\leq C
\bigl( 1+\bigl\llVert M_{\cdot\leftarrow0}^{\mathbf{x}} ( y_{0} )
\bigr\rrVert _{p\mbox{-}\mathrm{var}; [ 0,T ]
} \bigr) ^{p} \label{high order bound}%
\end{equation}
for every $l=1,\ldots,n+1$, every $k\in\mathbf{K}_{ (
i_{1},\ldots
,i_{n} ) }$ and every $ ( i_{1},\ldots,i_{n} ) \in
\{
1,\ldots,d \} ^{n}$.
\end{proposition}

\begin{pf}
We observe that $D_{h_{1}}\cdots D_{h_{n}}U_{t\leftarrow0}^{\mathbf
{x}} (
y_{0} ) _{j}$ equals
%
\begin{equation}
\sum_{i_{1},\ldots,i_{n}=1}^{d}\int_{0<t_{1}<\cdots
<t_{n}<t}D_{t_{1}\cdots
t_{n}}^{i_{1}\cdots i_{n}}U_{t\leftarrow0}^{\mathbf{x}}
( y_{0} ) _{j}\,dh_{t_{1}}^{i_{1}}\cdots
dh_{t_{n}}^{i_{n}}. \label{high order}%
\end{equation}
The representation for the integrand in (\ref{high order}) derived in
Proposition 4.4 in \cite{H3} then allows us to deduce (\ref{high
order rep})
and (\ref{high order bound}).
\end{pf}

\subsubsection{Malliavin differentiability}

We now switch back to the context of a continuous Gaussian process
$ (
X_{t} ) _{t\in [ 0,T ] }= ( X_{t}^{1},\ldots
,X_{t}%
^{d} ) _{t\in [ 0,T ] }$ with i.i.d. components
associated to
the abstract Wiener space $ ( \mathcal{W},\mathcal{H},\mu
) $.
Under the assumption of finite 2D $\rho$-variation, we have already remarked
that, for any $p>2\rho$, $X$ has a unique natural lift to a geometric
$p$-rough path $\mathbf{X}$. But the assumption of finite $\rho
$-variation on
the covariance also gives rise to the embedding
%
\begin{equation}
\mathcal{H\hookrightarrow}C^{q\mbox{-}\mathrm{var}} \bigl( [ 0,T ],%
\mathbb{R} 
^{d} \bigr)
\label{CM embedding}%
\end{equation}
for the Cameron--Martin space, for any $1/p+1/q>1$, \cite{CFV}, Proposition~2. The
significance of this result it twofold. First, it is proved in \cite{CFV}, Proposition~3, that it implies the existence of a (measurable) subset
$\mathcal{V\subset W}$ with $\mu ( \mathcal{V} ) =1$ on which
\[
T_{h}\mathbf{X} ( \omega ) \equiv\mathbf{X} ( \omega +h )
\]
for all $h\in\mathcal{H}$ simultaneously. It follows that the Malliavin
derivative\break $\mathcal{D}U_{t\leftarrow0}^{\mathbf{X} ( \omega
)
} ( y_{0} ) \dvtx \mathcal{H\rightarrow}%
\mathbb{R}
^{e}$
%
\begin{equation}
\mathcal{D}U_{t\leftarrow0}^{\mathbf{X} ( \omega )
} ( y_{0} ) \dvtx h\mathcal{
\mapsto D}_{h}U_{t\leftarrow0}^{\mathbf
{X} (
\omega ) } ( y_{0} )
:= \frac{d}{d\varepsilon
}U_{t\leftarrow0}^{\mathbf{X} ( \omega+\varepsilon h )
} ( y_{0}
) \bigg\rrvert _{\varepsilon=0}, \label{SGD}%
\end{equation}
coincides with the directional derivative of the previous section, that
is,%
%
\begin{equation}
 \frac{d}{d\varepsilon}U_{t\leftarrow0}^{\mathbf{X} (
\omega
+\varepsilon
h ) } ( y_{0} )
\bigg\rrvert _{\varepsilon=0}= \frac
{d}{d\varepsilon}U_{t\leftarrow0}^{T_{\varepsilon h}\mathbf{x}}
( y_{0} ) \bigg\rrvert _{\varepsilon=0}. \label{coincide}%
\end{equation}
The second important consequence results from combining (\ref{CM embedding}),
(\ref{coincide}) and (\ref{linear bound}), namely that
%
\begin{equation}
\bigl\llVert \mathcal{D}U_{t\leftarrow0}^{\mathbf{X} ( \omega
)
} ( y_{0} )
\bigr\rrVert _{\opn}\leq C\bigl\llVert M_{\cdot
\leftarrow
0}^{\mathbf{X} ( \omega ) }
( y_{0} ) \bigr\rrVert _{p\mbox{-}\mathrm{var}; [ 0,t ] }. \label{op norm bound}%
\end{equation}

If we can show that the right-hand side of (\ref{op norm bound}) has finite
positive moments of all order, then these observations lead to the conclusion
that
\[
Y_{t}=U_{t\leftarrow0}^{\mathbf{X}}(y_{0})\in\bigcap
_{p>1}\mathbb {D}^{1,p} \bigl( %
\mathbb{R} 
^{e} \bigr),
\]
where $\mathbb{D}^{k,p}$ is the Shigekawa--Sobolev space (see Nualart
\cite{nualart}). The purpose of Proposition~\ref{induction} is to
extend this
argument to the higher order derivatives. We will make this more precise
shortly, but first we remark that the outline just given is what
motivates the
assumption
\[
\mathcal{H\hookrightarrow}C^{q\mbox{-}\mathrm{var}} \bigl( [ 0,T ],%
\mathbb{R} 
^{d} \bigr)
\]
detailed in Condition \ref{standing assumption}.\hskip.2pt\footnote{The
requirement of
complementary regularity in the Condition \ref{standing assumption} then
amounts to $\rho\in [1,3/2)$. This covers BM, the OU process
and the
Brownian bridge (all with $\rho=1$) and fBm for $H>1/3$ (taking $\rho=1/2H$).
For the special case of fBm, one can actually improve on this general embedding
statement via Remark \ref{fBM embedding}. The requirement of complementary
then leads to the looser restriction $H>1/4$.}

The following theorem follows from the recent paper \cite{CLL}. It
asserts the
sufficiency of Condition \ref{standing assumption} to show the
existence of
finite moments for the $p$-variation of the Jacobian of the flow (and
its inverse).

\begin{theorem}[{[Cass--Litterer--Lyons (CLL)]}] \label{CLL}Let $ (
X_{t} ) _{t\in [
0,T ] }$ be a continuous, centred Gaussian process in $%
\mathbb{R}
^{d}$ with i.i.d. components. Let $X$ satisfy Condition
\ref{standing assumption}, so that for some $p\geq1$, $X$ admits a natural
lift to a geometric $p$-rough path $\mathbf{X}$. Assume $V= (
V_{0},V_{1},\ldots,V_{d} ) $ is any collection of smooth bounded vector
fields on $%
\mathbb{R}
^{e}$ and let $U_{t\leftarrow0}^{\mathbf{X}} ( \cdot ) $
denote the
solution flow to the RDE
%
\begin{eqnarray*}
dU_{t\leftarrow0}^{\mathbf{X}} ( y_{0} ) &=&V \bigl(
U_{t\leftarrow
0}^{\mathbf{X}} ( y_{0} ) \bigr) \,d\mathbf
{X}_{t}+V_{0} \bigl( U_{t\leftarrow0}^{\mathbf{X}} (
y_{0} ) \bigr) \,dt,
\\
U_{0\leftarrow0}^{\mathbf{X}} ( y_{0} ) &=&y_{0}.
\end{eqnarray*}
Then the map $U_{t\leftarrow0}^{\mathbf{X}} ( \cdot ) $
is differentiable with derivative $J_{t\leftarrow0}^{\mathbf{X}} (
y_{0} ) \in%
\mathbb{R}
^{e\times e}$;\break $J_{t\leftarrow0}^{\mathbf{X}} ( y_{0} ) $ is
invertible as a linear map with inverse denoted by $J_{0\leftarrow
t}^{\mathbf{X}} ( y_{0} ) $. Furthermore, if we define
\[
M_{\cdot\leftarrow0}^{\mathbf{X}} ( y_{0} ) \equiv \bigl(
U_{t\leftarrow0}^{\mathbf{X}} ( y_{0} ),J_{t\leftarrow
0}^{\mathbf{X}}
( y_{0} ),J_{0\leftarrow t}^{\mathbf
{X}} ( y_{0} )
\bigr) \in%
\mathbb{R} 
^{e}\oplus%
\mathbb{R}
^{e\times e}\oplus%
\mathbb{R} 
^{e\times e},
\]
and assume $X$ satisfies Condition \ref{standing assumption}, we have that
\[
\bigl\llVert M_{\cdot\leftarrow0}^{\mathbf{X}} ( y_{0} ) \bigr\rrVert
_{p\mbox{-}\mathrm{var}; [ 0,T ] }\in\bigcap_{q\geq
1}L^{q} (
\mu ).
\]
\end{theorem}

\begin{pf}
This follows from by repeating the steps of \cite{CLL} generalised to
incorporate a drift term.
\end{pf}

\begin{remark}
Under the additional assumption that the covariance $R$ has finite
H\"older-controlled $\rho$-variation, it is possible to prove a
version of
this theorem showing that
\[
\bigl\llVert M_{\cdot\leftarrow0}^{\mathbf{X}} ( y_{0} ) \bigr\rrVert
_{1/p}\in\bigcap_{q\geq1}L^{q} (
\mu ).
\]
\end{remark}

\subsection{Proof that \texorpdfstring{$U_{t\leftarrow0}^{\mathbf{X}(\cdot)}(y_{0})\in\mathbb{D}^{\infty}(\mathbb{R}^{e})$}{$U_{t leftarrow 0}^{X(cdot)}(y_{0})in D^{infty}(R^{e})$}}

We have already seen that appropriate assumptions on the covariance
lead to
the observation that for all $h\in\mathcal{H}$,
\[
D_{h}U_{t\leftarrow0}^{\mathbf{X} ( \omega ) } ( y_{0} ) \equiv
 \frac{d}{d\varepsilon}U_{t\leftarrow0}^{T_{h}\mathbf
{X} (
\omega ) } ( y_{0} )
\bigg\rrvert _{\varepsilon=0}%
\]
for all $\omega$ in a set of $\mu$-full measure. We will show that
the Wiener
functional $\omega\mapsto U_{t\leftarrow0}^{\mathbf{X} ( \omega
)
} ( y_{0} ) $ belongs to the Sobolev space $\mathbb
{D}^{\infty
} (
\mathbb{R}
^{e} ) $. Recall that
\[
\mathbb{D}^{\infty} \bigl( 
\mathbb{R}
^{e} \bigr):=\bigcap_{p>1}
\bigcap_{k=1}^{\infty}\mathbb {D}^{k,p}
\bigl( 
\mathbb{R} 
^{e} \bigr),
\]
where $\mathbb{D}^{k,p}$ is the usual Shigekawa--Sobolev space, which is
defined as the completion of the smooth random variables with respect
to a
Sobolev-type norm (see Nualart \cite{nualart}). There is an equivalent
characterisation of the spaces $\mathbb{D}^{k,p}$ (originally due to Kusuoka
and Stroock), which is easier to use in the present context. We briefly recall
the main features of this characterisation starting with the following
definitions. Suppose $E$ is a given Banach space and $F\dvtx \mathcal
{W\rightarrow
}E$ is a measurable function. Recall (see Sugita \cite{sugita}) that
$F$ is
called ray absolutely continuous (RAC) if for every $h\in\mathcal
{H}$, there
exists a measurable map $\tilde{F}_{h}\dvtx \mathcal{W\rightarrow}E$
satisfying
\[
F ( \cdot ) =\tilde{F}_{h} ( \cdot ), \qquad 
\mu
\mbox{-a.e.,}%
\]
and for every $\omega\in\mathcal{W}$
\[
t\mapsto\tilde{F}_{h} ( \omega+th ) \qquad\mbox{is absolutely continuous
in }t\in%
\mathbb{R} 
.
\]
And furthermore, $F$ is called stochastically G\^ateaux differentiable
(SGD) if
there exists a measurable $G\dvtx \mathcal{W\rightarrow}L ( \mathcal
{H}%
,E ) $, such that for any $h\in\mathcal{H}$
\[
\frac{1}{t} \bigl[ F ( \cdot+th ) -F ( \cdot ) \bigr] \stackrel{\mu} {
\rightarrow}G ( \omega ) ( h ) \qquad\mbox{as }t\rightarrow0,
\]
where $\stackrel{\mu}{\rightarrow}$ indicates convergence in $\mu$-measure.

If $F$ is SGD, then its derivative $G$ is unique $\mu$-a.s. and we
denote it
by $\mathcal{D}F$. Higher order derivatives are defined inductively in the
obvious way. Hence, $\mathcal{D}^{n}F$ $ ( \omega ) $ (if
it exists)
is a multi-linear map (in $n$ variables) from $\mathcal{H}$ to $E$.

We now define the spaces $\tilde{\mathbb{D}}^{k,p} (
\mathbb{R}
^{e} ) $ for $1<p<\infty$ by
\[
\tilde{\mathbb{D}}^{1,p} \bigl( 
\mathbb{R}
^{e} \bigr):= \bigl\{ F\in L^{p} \bigl(
\mathbb{R} 
^{e} \bigr) \dvtx F\mbox{
is RAC and SGD, } \mathcal{D}F\in L^{p} \bigl( L \bigl(
\mathcal{H},%
\mathbb{R} 
^{e} \bigr) \bigr) \bigr\},
\]
and for $k=2,3,\ldots.$
\[
\tilde{\mathbb{D}}^{k,p} \bigl( 
\mathbb{R}
^{e} \bigr):= \bigl\{ F\in\tilde{\mathbb{D}}^{k-1,p}
\bigl( 
\mathbb{R} 
^{e} \bigr) \dvtx \mathcal{D}F\in\tilde{\mathbb{D}}^{k-1,p} \bigl( L \bigl(
\mathcal{H},%
\mathbb{R} 
^{e} \bigr) \bigr) \bigr\}.
\]

\begin{theorem}[{(Sugita \cite{sugita})}]\label{sugita}For $1<p<\infty$ and $k\in%
\mathbb{N,}
$ we have $\tilde{\mathbb{D}}^{k,p} (
\mathbb{R}
^{e} ) =\mathbb{D}^{k,p} (
\mathbb{R}
^{e} ) $.
\end{theorem}

It follows immediately from this result that we have
\[
\mathbb{D}^{\infty} \bigl( 
\mathbb{R}
^{e} \bigr) =\bigcap_{p>1}
\bigcap_{k=1}^{\infty}\tilde{\mathbb{D}}^{k,p} \bigl( 
\mathbb{R} 
^{e} \bigr).
\]

With these preliminaries out the way, we can prove the following.

\begin{proposition}
\label{sobolev}Suppose $ ( X_{t} ) _{t\in [ 0,T
] }$ is
an $%
\mathbb{R}
^{d}$-valued, zero-mean Gaussian process with i.i.d. components
associated with
the abstract Wiener space $ ( \mathcal{W},\mathcal{H},\mu )
$. Assume that for some $p\geq1$, $X$ lifts to a geometric $p$-rough path
$\mathbf{X}$. Let $V= ( V_{0},V_{1},\ldots,V_{d} ) $ be a
collection of $C^{\infty}$-bounded vector fields on $%
\mathbb{R}
^{e}$, and let $U_{t\leftarrow0}^{\mathbf{X} ( \omega )
} (
y_{0} ) $ denote the solution flow of the RDE
\[
dY_{t}=V ( Y_{t} ) \,d\mathbf{X}_{t} ( \omega )
+V_{0} ( Y_{t} ) \,dt,\qquad Y ( 0 ) =y_{0}.
\]
Then, under the assumption that $X$ satisfies Condition
\ref{standing assumption}, we have that the Wiener functional
\[
U_{t\leftarrow0}^{\mathbf{X} ( \cdot ) } ( y_{0} ) \dvtx \omega\mapsto
U_{t\leftarrow0}^{\mathbf{X} ( \omega )
} ( y_{0} )
\]
is in $\mathbb{D}^{\infty} (
\mathbb{R}
^{e} ) $ for every $t\in [ 0,T ] $.
\end{proposition}

\begin{pf}
We have already remarked that Condition \ref{standing assumption}
implies that
on a set of $\mu$-full measure
%
\begin{equation}
T_{h}\mathbf{X} ( \omega ) \equiv\mathbf{X} ( \omega +h )
\label{trans}%
\end{equation}
for all $h\in\mathcal{H}$. It easily follows that $U_{t\leftarrow
0}^{\mathbf{X} ( \cdot ) } ( y_{0} ) $ is RAC.
Furthermore, its stochastic G\^ateaux derivative is precisely the map
$\mathcal{D}U_{t\leftarrow0}^{\mathbf{X} ( \omega )
} (
y_{0} ) $ defined in (\ref{SGD}). The relation~(\ref{trans}) implies
that the directional and Malliavin derivatives coincide (on a set of
$\mu
$-full measure), hence $\mathcal{D}U_{t\leftarrow0}^{\mathbf{X} (
\omega ) } ( y_{0} ) \in L ( \mathcal{H},%
\mathbb{R}
^{e} ) $ is the map
\[
\mathcal{D}U_{t\leftarrow0}^{\mathbf{X} ( \omega )
} ( y_{0} ) \dvtx h\mapsto
D_{h}U_{t\leftarrow0}^{\mathbf{X} (
\omega ) } ( y_{0} ).
\]
We have shown in (\ref{op norm bound}) that
%
\begin{equation}
\bigl\llVert \mathcal{D}U_{t\leftarrow0}^{\mathbf{X} ( \omega
)
} ( y_{0} )
\bigr\rrVert _{\opn}\leq C\bigl\llVert M_{\cdot
\leftarrow
0}^{\mathbf{X}} (
y_{0} ) \bigr\rrVert _{p\mbox{-}\mathrm{var}; [
0,T ] }, \label{op norm bound 2}%
\end{equation}
where
%
\begin{equation}
M_{\cdot\leftarrow0}^{\mathbf{X}} ( y_{0} ) \equiv \bigl(
U_{t\leftarrow0}^{\mathbf{X}} ( y_{0} ),J_{t\leftarrow
0}^{\mathbf{X}}
( y_{0} ),J_{0\leftarrow t}^{\mathbf
{X}} ( y_{0} )
\bigr). \label{eq:def-M}%
\end{equation}
It follows from Theorem \ref{CLL} that
\[
\bigl\llVert M_{\cdot\leftarrow0}^{\mathbf{X}} ( y_{0} ) \bigr\rrVert
_{p\mbox{-}\mathrm{var}; [ 0,T ] }\in\bigcap_{p\geq
1}L^{p} (
\mu ).
\]
Using this together with (\ref{op norm bound 2}) proves that
$U_{t\leftarrow
0}^{\mathbf{X} ( \cdot ) } ( y_{0} ) $ is in
$\bigcap_{p>1}\tilde{\mathbb{D}}^{1,p} (
\mathbb{R}
^{e} ) $ which equals $\bigcap_{p>1}\mathbb{D}^{1,p} (
\mathbb{R}
^{e} ) $ by Theorem \ref{sugita}.

We prove that $U_{t\leftarrow0}^{\mathbf{X} ( \cdot )
} (
y_{0} ) $ is in $\bigcap_{p>1}\tilde{\mathbb{D}}^{k,p} (
\mathbb{R}
^{e} ) $ for all $k\in%
\mathbb{N}
$ by induction. If $U_{t\leftarrow0}^{\mathbf{X} ( \cdot )
} (
y_{0} ) \in\tilde{\mathbb{D}}^{k-1,p} (
\mathbb{R}
^{e} ) $ then, by the uniqueness of the stochastic G\^ateaux derivative,
we must have
\[
\mathcal{D}^{k-1}U_{t\leftarrow0}^{\mathbf{X} ( \omega )
} ( y_{0}
) ( h_{1},\ldots,h_{k-1} ) =D_{h_{1}}\cdots
D_{h_{k}%
}U_{t\leftarrow0}^{\mathbf{X} ( \omega ) } ( y_{0} ).
\]
It is then easy to see that $\mathcal{D}^{k-1}U_{t\leftarrow
0}^{\mathbf
{X}%
( \omega ) } ( y_{0} ) $ is RAC and SGD.
Moreover, the
stochastic G\^ateaux derivative is
\[
\mathcal{D}^{k}U_{t\leftarrow0}^{\mathbf{X} ( \omega )
} ( y_{0} )
\dvtx ( h_{1},\ldots,h_{k} ) =D_{h_{1}}\cdots
D_{h_{k}
}U_{t\leftarrow0}^{\mathbf{X} ( \omega ) } ( y_{0} ).
\]
It follows from Proposition \ref{induction} together with Condition
\ref{standing assumption} that we can bound the operator norm of
$\mathcal{D}^{k}U_{t\leftarrow0}^{\mathbf{X} ( \omega )
} (
y_{0} ) $ in the following way:
\[
\bigl\|\mathcal{D}^{k}U_{t\leftarrow0}^{\mathbf{X} ( \omega )
} ( y_{0}
) \bigr\|_{\opn}\leq C \bigl( 1+\bigl\llVert M_{\cdot\leftarrow0}%
^{\mathbf{X} ( \omega ) } ( y_{0} ) \bigr\rrVert _{p\mbox{-}\mathrm{var}; [ 0,T ] } \bigr)
^{ ( k+1
) p}%
\]
for some nonrandom constants $C>0$. The conclusion that
$U_{t\leftarrow
0}^{\mathbf{X} ( \cdot ) } ( y_{0} ) \in\break
\bigcap_{p>1}\mathbb{D}^{k,p} (
\mathbb{R}
^{e} ) $ follows at once from Theorems \ref{CLL} and \ref{sugita}.
\end{pf}

\subsection{Note added in proof} Shortly before the article went to press,
the authors were made aware of a
mistake in the proof of Proposition \ref{sobolev}:
control of the operator norms of
$\mathcal{D}^{k}U_{t\leftarrow0}^{\mathbf{X}(  \omega)  }
( y_{0})$ does not imply
Malliavin smoothness; instead control of the stronger Hilbert--Schmidt
norms is required. In a more
restrictive setting, this stronger control has been obtained in \cite{H3}.
At the level of generality
considered in this article, this result has recently been obtained
in \cite{Ina13}, so that the statement of
Proposition \ref{sobolev} does hold and none of our results are affected.

\section{Smoothness of the density: The proof of the main theorem}
\label{section proof main theorem} \label{sec:proof-main-result}

This section is devoted to the proof of our H\"ormander-type Theorem
\ref{main theorem}. As mentioned in the \hyperref[sec1]{Introduction}, apart from rather
standard considerations concerning probabilistic proofs of H\"ormander's
theorem (see, e.g., \cite{H3}), this boils down to the following steps:

\begin{longlist}[(1)]
\item[(1)] Let $W$ be a smooth and bounded vector field in
$\mathbb{R}^{e}$. Following \cite{H3}, denote by $ (
Z_{t}^{W} )
_{t\in [ 0,T ] }$ the process
%
\begin{equation}
Z_{t}^{W}=J_{0\leftarrow t}^{\mathbf{X}}W \bigl(
U_{t\leftarrow
0}^{\mathbf{X}%
} ( y_{0} ) \bigr).
\label{defn Z}%
\end{equation}
Then assuming Conditions \ref{nondeterm} and \ref{cond dom} we get a
bound on
$|Z^{W}|_{\infty}$ in terms of the Malliavin matrix $C_{T}$ defined at
\eqref{eq:def-malliavin-matrix}. This will be the content of Proposition
\ref{CMatrix}.

\item[(2)] We invoke iteratively our Norris lemma (Theorem
\ref{thm:NlemmaRP}) to processes like $Z^{W}$ in order to generate enough
upper bounds on Lie brackets of our driving vector fields at the origin.
\end{longlist}

In order to perform this second step, we first have to verify the assumptions
of Theorem \ref{thm:NlemmaRP} for the process $M_{\cdot\leftarrow
0}^{\mathbf{x}} ( y_{0} ) $ defined by \eqref{eq:def-M}.
Namely, we
shall see that $M_{\cdot\leftarrow0}^{\mathbf{x}} ( y_{0}
) $
is a
process controlled by $\mathbf{X}$ in the sense of Definition
\ref{def:controlled-paths} and relation~\eqref{eq:dcp-increment-y-d-dim}.

\begin{proposition}
\label{CLL mod}Suppose $ ( X_{t} ) _{t\in [ 0,T
] }$
satisfies the condition of Theorem~\ref{CLL}. In particular, $X$ has a
lift to
$\mathbf{X,}$ a geometric-p rough path for some $p>1$ which is in
$C^{0,\gamma}([0,T];G^{\lfloor p\rfloor}(\mathbb{R}^{d}))$ for
$\gamma=1/p$.
Then $M_{\cdot\leftarrow0}^{\mathbf{x}} ( y_{0} ) $ is a process
controlled by $\mathbf{X}$ in the sense of Definition
\ref{def:controlled-paths} and
\[
\bigl\llVert M_{\cdot\leftarrow0}^{\mathbf{X}} ( y_{0} ) \bigr\rrVert
_{\mathcal{Q}_{\mathbf{X}}^{\gamma}}\in\bigcap_{p\geq1}L^{p} (
\Omega ).
\]
\end{proposition}

\begin{pf}
For notational sake, the process $M_{\cdot\leftarrow0}^{\mathbf
{X}} (
y_{0} ) $ will be denoted by $M$ only. It is readily checked that
$M$ is
solution to a rough differential equation driven by~$\mathbf{X}$, associated
to the vector fields given by
%
\begin{equation}
\label{eq:coefficients-M}F_{i} ( y,J,K ) = \bigl( V_{i} ( y ),
\nabla V_{i} ( y ) \cdot J, -K \cdot\nabla V_{i} ( y )
\bigr), \qquad i=0,\ldots,d.
\end{equation}
This equation can be solved either by genuine rough paths methods or within
the landmark of algebraic integration. As mentioned in Proposition
\ref{prop:integral-ctrld-path}, both notions of solution coincide
thanks to
approximation procedures. This finishes the proof of our claim $M\in
\mathcal{Q}_{\mathbf{X}}^{\gamma}$.

In order to prove integrability of $M$ as an element of $\mathcal{Q}%
_{\mathbf{X}}^{\gamma}$, let us write the equation governing the
dynamics of
$M$ under the form
\[
dM_{t}= F_{0}(M_{t}) \,dt + \sum
_{i=1}^{d}F_{i}(M_{t}) \,d
\mathbf{X}_{t}^{i},
\]
where $\mathbf{X}$ is our Gaussian rough path of order at most $N=3$. The
expansion of $M$ as a controlled process is simply given by the Euler scheme
introduced in \cite{FV}, Proposition 10.3. More specifically, $M$
admits a
decomposition \eqref{eq:dcp-increment-y} of the form:
\[
M_{s,t}^{j}=M_{s}^{j,0} (t-s) +
M_{s}^{j,i_{1}} \mathbf{X}_{s,t}^{1,i_{1}}
+M_{s}^{j,i_{1},i_{2}} \mathbf{X}_{s,t}^{2,i_{1},i_{2}}+R_{s,t}^{j,M},
\]
with
\begin{eqnarray*}
M_{s}^{j,0} &=& F_{0}^{j}(M_{s}),\qquad
M_{s}^{j,i_{1}}=F_{i_{1}}^{j}%
(M_{s}), \qquad M_{s}^{j,i_{1},i_{2}}=F_{i_{2}}F_{i_{1}}^{j}(M_{s}%
),\\
\bigl|R_{s,t}^{j,M}\bigr|&\leq& c_{M}|t-s|^{3\gamma}.
\end{eqnarray*}
With the particular form \eqref{eq:coefficients-M} of the coefficient
$F$ and
our assumptions on the vector fields $V$, it is thus readily checked that
\[
\Vert M\Vert_{\mathcal{Q}_{\mathbf{X}}^{\gamma}}\leq c_{V} \bigl( 1+\Vert J
\Vert_{\infty}^{2}+\Vert J^{-1}\Vert_{\infty}^{2}+
\Vert J\Vert _{\gamma
}+\Vert U\Vert_{\gamma} \bigr),
\]
and the right-hand side of the latter relation admits moments of all order
thanks to Theorem~\ref{CLL} and the remark which follows it.
\end{pf}

Define $\mathcal{L}_{\mathbf{x}} ( y_{0},\theta,T ) $ to
be the
quantity
\[
\mathcal{L}_{\mathbf{x}} ( y_{0},\theta,T ):=1+L_{\theta
}
( x ) ^{-1}+\llvert y_{0}\rrvert +\bigl\llVert
M_{\cdot
\leftarrow
0}^{\mathbf{x}} ( y_{0} ) \bigr\rrVert
_{\mathcal{Q}_{\mathbf
{x}%
}^{\gamma}}+\mathcal{N}_{\mathbf{x},\gamma}.%
\]

\begin{corollary}
\label{holder roughness integrability}Under the assumptions of Proposition
\ref{CLL mod}, we have
\[
\mathcal{L}_{\mathbf{x}} ( y_{0},\theta,T ) \in\bigcap
_{p\geq
1}L^{p} ( \Omega ).
\]
\end{corollary}

\begin{pf}
We recall that the standing assumptions imply that $\mathbf{\|X\|}%
_{\gamma;[ 0,T ] }$ has a Gaussian tail [see (\ref{gauss})
from Section~\ref{rough paths}]. It is easily deduce from this that
\[
\mathcal{N}_{\mathbf{X},\gamma}\in\bigcap_{p\geq1}L^{p}
( \Omega ).
\]
Similarly, we see from Corollary \ref{l theta integrability} and Proposition
\ref{CLL mod} that $L_{\theta} ( x ) ^{-1}$ and $\llVert
M_{\cdot\leftarrow0}^{\mathbf{x}} ( y_{0} ) \rrVert
_{\mathcal{Q}_{\mathbf{x}}^{\gamma}}$ have moments of all orders and
the claim follows.
\end{pf}

\begin{definition}
We define the sets of vector fields $\mathcal{V}_{k}$ for $k\in%
\mathbb{N}
$ inductively by
\[
\mathcal{V}_{1}= \{ V_{i}\dvtx i=1,\ldots,d \},
\]
and then
\[
\mathcal{V}_{n+1}= \bigl\{ [ V_{i},W ] \dvtx i=0,1,\ldots,d,W
\in\mathcal{V}_{n} \bigr\}.
\]
\end{definition}

\begin{proposition}
\label{CMatrix} Let $ ( X_{t} ) _{t\in [ 0,T ]
}= (
X_{t}^{1},\ldots,X_{t}^{d} ) _{t\in [ 0,T ] }$ be a
continuous Gaussian process, with i.i.d. components associated to the abstract
Wiener space $ ( \mathcal{W},\mathcal{H},\mu ) $. Assume
that $X$
satisfies the assumptions of Theorem \ref{main theorem}. Then there
exist real
numbers $p$ and $\theta$ satisfying $2/p>$ $\theta>\alpha/2$ such
that: \textup{(i)}
$X$ is $\theta$-H\"{o}lder rough and \textup{(ii)} $X$ has a natural lift to a
geometric $p$ rough path $\mathbf{X}$ in $C^{0,1/p}([0,T];G^{\lfloor
p\rfloor
}(\mathbb{R}^{d}))$. For $t\in(0,T]$, let
\[
C_{t}=\sum_{i=1}^{d}\int
_{ [ 0,t ] ^{2}}J_{t\leftarrow
s}^{\mathbf{X}} ( y_{0} )
V_{i} ( Y_{s} ) \otimes J_{t\leftarrow s^{\prime}}^{\mathbf{X}} (
y_{0} ) V_{i} ( Y_{s^{\prime}} ) \,dR \bigl(
s,s^{\prime} \bigr),
\]
and suppose $k\in%
\mathbb{N}
\cup \{ 0 \} $. Then there exist constants $\mu=\mu (
k ) >0$ and $C=C ( t,k ) >0$ such that for all $W\in
\mathcal{V}_{k}$ and all $v\in%
\mathbb{R}
^{e}$ with $\llvert  v\rrvert =1$, we have
%
\begin{equation}
\bigl\llvert \bigl\langle v,Z_{\cdot}^{W} \bigr\rangle\bigr
\rrvert _{\infty;
[ 0,t ] }\leq C\mathcal{L}_{\mathbf{X}} ( y_{0}%
,\theta,t ) ^{\mu} \bigl( v^{T}C_{t}v \bigr)
^{\mu}. \label
{conc}%
\end{equation}
\end{proposition}

\begin{pf}
Let us prove the first assertion. To do this, we note that the
constraint on
$\rho$ implies that $X$ lifts to a geometric $p$-rough path for any
$p>2\rho$.
Because the $\rho$-variation is assumed to be H\"{o}lder-controlled, it
follows from \cite{FV} that $\mathbf{X}$ is in
$C^{0,1/p}([0,T];G^{\lfloor
p\rfloor}(\mathbb{R}^{d}))$. By assumption $\alpha<2/\rho$, therefore
we may
always choose $p$ close enough to $2\rho$ in order that
\[
\frac{2}{p}>\frac{\alpha}{2}.
\]
On the other hand, $X$ is $\theta$-H\"{o}lder rough for any $\theta
>\alpha/2$
by Corollary \ref{l theta integrability}. Hence, there always exist
$p$ and
$\theta$ with the stated properties.

We have that
%
\begin{equation}
v^{T}C_{t}v=\sum_{i=1}^{d}
\Lambda_{t}^{i} \qquad\mbox{with } \Lambda_{t}^{i}
\equiv\int_{ [ 0,t ] ^{2}}f^{i} ( s ) f^{i} \bigl(
s^{\prime} \bigr) \,dR \bigl( s,s^{\prime} \bigr),
\label{covariance matrix}%
\end{equation}
where we have set $f^{i} ( s ):= \langle
v,J_{t\leftarrow
s}^{\mathbf{X}} ( y_{0} ) V_{i} ( y_{s} )
\rangle
= \langle v,Z_{s}^{V_{i}} \rangle$.
Furthermore, because the hypotheses of Theorem \ref{interpolation} and
Corollary \ref{interpolation2} are satisfied, we can deduce that%
%
\begin{equation}
\bigl\llvert f^{i}\bigr\rrvert _{\infty; [ 0,t ] }\leq2\max \biggl[
\frac{\llvert  \Lambda_{t}^{i} \rrvert  ^{1/2}}{E [
X_{t}^{2} ]
^{1/2}},\frac{1}{\sqrt{c}}\bigl\llvert \Lambda_{t}^{i}
\bigr\rrvert ^{\gamma
/ (
2\gamma+\alpha ) }\bigl\llvert f^{i}\bigr\rrvert
_{\gamma; [
0,t ] }^{\alpha/ ( 2\gamma+\alpha ) } \biggr] \label{in}%
\end{equation}
for $i=1,\ldots,d$. On the other hand, Young's inequality for 2D
integrals (see
\cite{FV07}) gives
%
\begin{equation}
\bigl\llvert \Lambda_{t}^{i} \bigr\rrvert \lesssim \bigl[
\bigl\llvert f^{i}\bigr\rrvert _{\gamma; [ 0,t ] }+\bigl\llvert
f^{i} ( 0 ) \bigr\rrvert \bigr] ^{2}V_{\rho} \bigl(
R; [ 0,t ] ^{2} \bigr). \label{2D Young}%
\end{equation}
From (\ref{2D Young}), (\ref{in}) and the relation $v^{T}C_{t}v=\sum_{i=1}^{d}
\Lambda_{t}^{i}$, it follows that there exists some $C_{1}>0$,
depending on
$t$ and $c$, such that we have
\[
\bigl\llvert f^{i}\bigr\rrvert _{\infty; [ 0,t ] }\leq C_{1}
\bigl( v^{T}C_{t}v \bigr) ^{\gamma/ ( 2\gamma+\alpha ) }\max
_{i=1,\ldots,d} \bigl[ \bigl\llvert f^{i} ( 0 ) \bigr\rrvert
+\bigl\llvert f^{i}\bigr\rrvert _{\gamma; [ 0,t ] } \bigr]
^{\alpha/ ( 2\gamma+\alpha ) }.
\]
Using the fact that for some $\nu>0$,
\[
\bigl\llvert f^{i} ( 0 ) \bigr\rrvert +\bigl\llvert f^{i}
\bigr\rrvert _{\gamma; [ 0,t ] }\leq C_{2}\mathcal{L}_{\mathbf
{X}} (
y_{0},\theta,t ) ^{\nu}\qquad\mbox{for }i=1,\ldots,d,
\]
it is easy to deduce that (\ref{conc}) holds whenever $W\in\mathcal{V}_{1}$.

The proof of (\ref{conc}) for arbitrary $k\in%
\mathbb{N}
$ now follows by induction. The key relation comes from observing that
\[
\bigl\langle v,Z_{u}^{W} \bigr\rangle= \bigl\langle v,W
( y_{0} ) \bigr\rangle+\sum_{i=1}^{d}
\int_{0}^{u} \bigl\langle v,Z_{r}^{
[ V_{i},W ] }
\bigr\rangle \,dX_{r}^{i},
\]
in the sense of Proposition \ref{prop:integral-ctrld-path}. Hence, assuming
the induction hypothesis, we can use Theorem \ref{thm:NlemmaRP} to
obtain a
bound of the form (\ref{conc}) on
\[
\bigl\llvert \bigl\langle v,Z_{\cdot}^{ [ V_{i},W ] } \bigr\rangle \bigr
\rrvert _{\infty; [ 0,t ] }%
\]
for all $W\in\mathcal{V}_{k}$. Since $\mathcal{V}_{k+1}= \{
[
V_{i},W ] \dvtx i=0,1,\ldots,d,W\in\mathcal{V}_{n} \} $, the
result is
then established.\vadjust{\goodbreak}
\end{pf}

We are now in a position to prove our main theorem. Since the structure
of the
argument is the classical one, we will minimise the amount of detail
where possible.

\begin{pf*}{Proof of Theorem \ref{main theorem}}
This involves assembling together the
pieces we have developed in the paper. First, let $2/p>$ $\theta
>\alpha
/2$ be
chosen such that $X$ is $\theta$-H\"{o}lder rough and $X$ has a
natural lift
to a geometric $p$ rough path $\mathbf{X}$ in
$C^{0,1/p}([0,1];G^{\lfloor
p\rfloor}(\mathbb{R}^{d}))$. This is always possible by the first
part of
Proposition~\ref{CMatrix}. Let $0<t\leq T$ and note that we have shown in
Proposition \ref{sobolev} that $U_{t\leftarrow0}^{\mathbf{X} (
\omega ) } ( y_{0} ) $ is in $\mathbb{D}^{\infty
} (
\mathbb{R}
^{e} ) $. The result will therefore follow by showing that for every
$q>0$, there exists $c_{1}=c_{1} ( q ) $ such that
\[
P \Bigl( \inf_{\llvert  v\rrvert =1} \langle v,C_{t}v \rangle <
\varepsilon \Bigr) \leq c_{1}\varepsilon^{q}
\]
for all $\varepsilon\in ( 0,1 ) $. It is classical that proving
$ ( \det C_{t} ) ^{-1}$ has finite moments of all order is
sufficient for $U_{t\leftarrow0}^{\mathbf{X} ( \omega )
} (
y_{0} ) $ to have a smooth density (see, e.g., \cite{nualart}).

\textit{Step} 1: From H\"{o}rmander's condition, there
exists $N\in%
\mathbb{N}
$ with the property that
\[
\operatorname{span} \Biggl\{ W ( y_{0} ) \dvtx W\in\bigcup_{i=1}^{N}
\mathcal{V}%
_{i} \Biggr\} =%
\mathbb{R} 
^{e}.
\]
Consequently, we can deduce that
%
\begin{equation}
a:=\inf_{\llvert  v\rrvert =1}\sum_{W\in\bigcup
_{i=1}^{N}\mathcal
{V}_{i}%
}\bigl
\llvert \bigl\langle v,W ( y_{0} ) \bigr\rangle \bigr\rrvert >0.
\label{lb}%
\end{equation}
For every $W\in\bigcup_{i=1}^{N}\mathcal{V}_{i}$, we have
%
\begin{equation}
\bigl\llvert \bigl\langle v,Z_{\cdot}^{W} \bigr\rangle\bigr
\rrvert _{\infty;
[ 0,t ] }\geq\bigl\llvert \bigl\langle v,W ( y_{0} )
\bigr\rangle\bigr\rrvert, \label{lb2}%
\end{equation}
and hence using (\ref{lb}), (\ref{lb2}) and Proposition \ref{CMatrix}
we end
up with
%
\begin{equation}
a\leq\inf_{\llvert  v\rrvert =1}\sup_{W\in\bigcup
_{i=1}^{N}\mathcal
{V}_{i}%
}\bigl\llvert
\bigl\langle v,Z_{\cdot}^{W} \bigr\rangle\bigr\rrvert
_{\infty; [ 0,t ] }\leq c_{1}\mathcal{L}_{\mathbf
{X}} (
y_{0},\theta,t ) ^{\mu}\inf_{\llvert  v\rrvert
=1}\bigl
\llvert v^{T}C_{t}v\bigr\rrvert ^{\pi}
\label{key bound}%
\end{equation}
for some positive constants $c_{1}$, $\mu=\mu_{N}$ and $\pi=\pi_{N}$.

\textit{Step} 2: From (\ref{key bound}) can deduce
that for
any $\varepsilon\in ( 0,1 ) $
\[
P \Bigl( \inf_{\llvert  v\rrvert =1}\bigl\llvert v^{T}C_{t}v
\bigr\rrvert <\varepsilon \Bigr) \leq P \bigl( \mathcal{L}_{\mathbf{X}} (
y_{0} 
,\theta,t ) ^{\mu}\mathcal{>}c_{2}
\varepsilon^{-k} \bigr)
\]
for some constants $c_{2}>0$ and $k>0$ which do not depend on $\varepsilon
$. It
follows from Corollary \ref{holder roughness integrability} that for every
$q>0$ we have
\[
P \Bigl( \inf_{\llvert  v\rrvert =1}\bigl\llvert v^{T}C_{t}v
\bigr\rrvert <\varepsilon \Bigr) \leq c_{3}\varepsilon^{kq},
\]
where $c_{3}=c_{3} ( q ) >0$ does not depend on $\varepsilon$.
\end{pf*}

\begin{appendix}\label{app}
\section*{Appendix}

\begin{pf*}{Proof of Lemma \ref{technical}}
We prove the result for $S=T$, the
modifications for $S<T$ are straightforward. Consider three nested sequences
$(A_{m})_{m=1}^{\infty}$, $ ( B_{m} ) _{m=1}^{\infty}$ and
$ (
C_{m} ) _{m=1}^{\infty}$ consisting of partitions of $ [
0,s ] $, $ [ s,t ] $ and $ [ t,T ]$, respectively,
and suppose that the mesh of each sequence tends to zero as $m$ tends to
infinity. For each $m_{1} $ and $m_{2}$ in $%
\mathbb{N}
$ let $D^{m_{1},m_{2}}$ denote the partition of $ [ 0,T ] $ defined
by
\[
D^{m_{1},m_{2}}=A_{m_{1}}\cup B_{m_{2}}\cup C_{m_{1}}.
\]
We now construct an increasing function $r\dvtx %
\mathbb{N}
\rightarrow%
\mathbb{N}
$ such that
\[
( D_{m} ) _{m=1}^{\infty}= \bigl( D^{r ( m )
,m}
\bigr) _{m=1}^{\infty}%
\]
together form a nested sequence of partitions of $ [ 0,T ] $ having
the needed properties.

We do this inductively. First, let $m=1$, then for every two
consecutive points
$u<v$ in the partition $B_{m}$ Lemma \ref{continuity} implies that
\[
\operatorname{Cov} \bigl( Z_{s,t}Z_{u,v}|
\mathcal{F}^{A_{n}}\vee \mathcal{F}^{C_{n}} \bigr) \rightarrow
\operatorname{Cov} ( Z_{s,t}Z_{u,v}|\mathcal{F}_{s,t}
\vee\mathcal{F}_{t,T} )
\]
as $n\rightarrow\infty$. $Z$ has positive conditional covariance, therefore
the right-hand side of the last expression is positive. This means we can
choose $r ( 1 ) $ to ensure that
%
\setcounter{equation}{0}
\begin{equation}
\operatorname{Cov} \bigl( Z_{s,t}Z_{u,v}|
\mathcal{F}^{A_{r (
1 )
}}\vee\mathcal{F}^{C_{r ( 1 ) }} \bigr) \geq0
\label
{cov part} 
\end{equation}
for every two consecutive points $u$ and $v$ in $B_{m}$ [the total
number of
such pairs does not depend on $r(1)$]. We then let $D_{1}=D^{r (
1 ),1}$, both properties 2 and 3 in the statement are easy to check;
the latter follows from (\ref{cov part}), when we interpret~the Schur
complement as the covariance matrix of $Z_{2}^{1}$ conditional on $Z_{1}^{1}$
(see also the proof of Proposition \ref{comparison}). Having specified
$r ( 1 ) <\cdots<r ( k-1 )$, we need only
repeat the
treatment outlined above by choosing some natural number $r (
k )
>r ( k-1 ) $ to ensure that
\[
\operatorname{Cov} \bigl( Z_{s,t}Z_{u,v}|
\mathcal{F}^{A_{r (
k )
}}\vee\mathcal{F}^{C_{r ( k ) }} \bigr) \geq0
\]
for each pair of consecutive points $u<v$ in $B_{k}$. It is easy to verify
that $ ( D_{m} ) _{m=1}^{\infty}$ constructed in this way
has the
properties we need.
\end{pf*}
\end{appendix}


%



\printaddresses

\end{document}